\documentclass[reqno]{amsart}

\usepackage{amsmath,amsthm,amssymb,amsfonts}
\usepackage{mathrsfs}
\usepackage{enumerate}

\newtheoremstyle{mystyle}
  {}
  {}
  {\normalfont}
  { }
  {\bfseries}
  {}
  {10pt}
  { }
\theoremstyle{mystyle}

\newtheorem{theorem}{Theorem}
\newtheorem{lemma}{Lemma}
\newtheorem{remark}{Remark}

\DeclareMathOperator*{\argsup}{arg\,sup}

\usepackage{graphicx}

\usepackage[foot]{amsaddr}
\usepackage[top=3cm,bottom=3cm,left=2.5cm,right=2.5cm,marginparwidth=1.75cm]{geometry}

\usepackage{comment}

\title[Adaptive testing method for ergodic diffusions]{Adaptive testing method for ergodic diffusion processes based on high frequency data}
\author[T Kawai]{Tetsuya Kawai $^{1}$}
\author[M Uchida]{Masayuki Uchida $^{1,2}$}
\address{$^{1}$Graduate School of Engineering Science, Osaka University}
\address{$^{2}$Center for Mathematical Modeling and Data Science (MMDS), Osaka University and JST CREST}

\begin{document}

\begin{abstract}
We consider parametric tests for multidimensional ergodic diffusions based on high frequency data. We propose two-step testing method for diffusion parameters and drift parameters. To construct test statistics of the tests, we utilize the adaptive estimator and provide three types of test statistics: likelihood ratio type test,  Wald type test and Rao's score type test. It is proved that these test statistics converge in distribution to the chi-squared distribution under null hypothesis and have consistency of the tests against alternatives. Moreover, these test statistics converge in distribution to the non-central chi-squared distribution under local alternatives. We also give some simulation studies of the behavior of the three types of test statistics.
\end{abstract}


\maketitle

\section{Introduction}\label{sec1}
We consider a $d$-dimensional diffusion process satisfying the following stochastic differential equation:
\begin{equation}\notag
\begin{cases}
dX_t = b(X_t,\beta) dt + a(X_t,\alpha)dW_t \quad t\in[0,T],\\
X_0 = x_0,
\end{cases}
\end{equation}
where $W_{t}$ is the $r$-dimensional standard Wiener process, $\alpha\in\Theta_{\alpha}\subset\mathbb{R}^{p_1},\ \beta \in \Theta_{\beta} \subset \mathbb{R}^{p_2}$, $\theta =(\alpha,\beta)$, $\Theta := \Theta_{\alpha} \times \Theta_{\beta}$ being compact and convex parameter space, $a:{\mathbb{R}}^{d} \times {\Theta}_{\alpha} \to {\mathbb{R}}^{d}\otimes {\mathbb{R}}^{r}$ and $b:{ \mathbb{R}}^{d} \times {\Theta}_{\beta} \to {\mathbb{R}}^{d}$ are known except for the parameter $\theta$. 
We assume the true parameter $\theta_0=(\alpha_0,\beta_0)$ belongs to $\mathrm{Int}(\Theta)$. 
The data are discrete observations $(X_{t^n_i})_{0\le i \le n}$, where $t^n_i=ih_n$ for $i =0, 1, \ldots, n$, 
and the discretization step $h_n$ satisfies  $h_n\to 0,\ nh_n\to \infty$ and $n h_n^2\to 0$ 
as $n\to\infty$.

Diffusion processes 
are used as the mathematical models to describe the random development of the phenomena
depending on time in many fields such as physics, neuroscience, meteorology, 
epidemiology and finance.
For these models, the data are discretely observed.
The statistical inference for ergodic diffusion process based on discrete observations has been studied by many researchers; see Florens-Zmirou \cite{florens_1989},\ 
Yoshida \cite{yoshida1992estimation, yoshida2011QLA}, Genon-Catalot and Jacod \cite{genon_1993},\ 
Kessler \cite{Kessler_1995, Kessler_1997},\ 
Uchida and Yoshida \cite{Uchida_2012} and reference therein. In particular, 
Kessler \cite{Kessler_1995, Kessler_1997}
proposed the adaptive maximum likelihood (ML) type estimator
and joint ML type estimator which has asymptotic efficiency 
under $nh_n^p \to 0$,  
where $p$ is an arbitrary integer with $p\geq2$. 
Uchida and Yoshida \cite{Uchida_2012} 
presented 
the polynomial type large deviation of adaptive statistical random fields 
under $nh_n^p \to 0$
and moment convergence of the adaptive ML type estimator 
with asymptotic efficiency. 

The parametric testing problem for ergodic diffusions has been studied as the following joint test:
\begin{equation}\label{j.test}
\begin{cases}
H_0:\ \alpha_1=\cdots=\alpha_{r_1}=0,\quad\beta_1=\cdots=\beta_{r_2}=0,\quad(1\leq r_1\leq p_1,\ 1\leq r_2\leq p_2)\\
H_1:\  \text{not}\ H_0.
\end{cases}
\end{equation}
Kitagawa and Uchida \cite{Kitagawa_2014} proposed three kinds of test statistics 
(likelihood ratio type test,  Wald type test and Rao's score type test)
and proved their asymptotic properties. De Gregorio and Iacus \cite{de2013family, de2019empirical} constructed the test statistics 
by means of 
$\phi$-divergence measure
and the empirical $L^2$-distance
when $r_1=p_1$ and $ r_2=p_2$.

In this paper, we consider the following set of tests instead of \eqref{j.test}:
\begin{equation}\label{ada.test}
\begin{cases}
H_0^{(1)}:\ \alpha_1=\cdots=\alpha_{r_1}=0,\\
H_1^{(1)}:\ \text{not}\ H_0^{(1)},
\end{cases}
\quad
\begin{cases}
H_0^{(2)}:\ \beta_1=\cdots=\beta_{r_2}=0,\\
H_1^{(2)}:\ \text{not}\ H_0^{(2)}.
\end{cases}
\end{equation}
The set of tests \eqref{ada.test} give more information about parameters than the test \eqref{j.test}. The test \eqref{j.test} provides only two interpretations: $H_0$ is rejected or $H_0$ is not rejected.  
On the other hand, the tests \eqref{ada.test} gives the following four conclusions:
(i) $H_0^{(1)}$ is rejected and $H_0^{(2)}$ is rejected; \
(ii) $H_0^{(1)}$ is rejected and $H_0^{(2)}$ is not rejected; \
(iii) $H_0^{(1)}$ is not rejected and $H_0^{(2)}$ is rejected; \ 
(iv) $H_0^{(1)}$ is not rejected and $H_0^{(2)}$ is not rejected.
We utilize the adaptive ML type estimator of Uchida and Yoshida \cite{Uchida_2012} 
and construct three types of test statistics. Furthermore, we prove that these test statistics converge in distribution to the chi-squared distribution under null hypothesis, have consistency of the tests under alternatives and converge in distribution to the non-central chi-squared distribution under local alternatives. 
For the asymptotic null distribution of
adaptive test statistics based on local means for noisy ergodic diffusion processes
and consistency of the tests under alternatives ,
see Nakakita and Uchida \cite{nakakita2019adaptive}.

The paper is organized as follows. 
In Section \ref{sec2}, notation and assumptions are introduced. 
In Section \ref{sec3}, we state main results. 
Two quasi log likelihood functions are  constructed and
three kinds of adaptive test statistics are proposed.
Moreover,  
their asymptotic properties
are shown. 
In Section \ref{sec4}, 
we give some examples and simulation results of the asymptotic performance for
three types of test statistics for 1-dimensional ergodic diffusion processes.
Section 5 is devoted to the proofs of the results presented in Section 3.

\section{Notation and assumptions}\label{sec2}
Let $\partial_i:=\partial/\partial{x_i},\ \partial_{\alpha_i}:=\partial/\partial{\alpha_i},\ \partial_{\beta_i}:=\partial/\partial{\beta_i},\ \partial_{\alpha}:=(\partial_{\alpha_1},\ldots,\partial_{\alpha_{p_1}})^\top,\ \partial_{\beta}:=(\partial_{\beta_1},\ldots,\partial_{\beta_{p_2}})^\top$,\ $\partial^2_{\alpha}:=\partial_\alpha\partial_\alpha^\top,\ \partial^2_{\beta}:=\partial_\beta\partial_\beta^\top,\ \partial^2_{\alpha\beta}:=\partial_\alpha\partial_\beta^\top$, where $\top$ is the transpose of a matrix. 
For ($m\times n$)-matrices $A$ and $B,$ 
it is defined that $A^{\otimes 2 } :=AA^\top$,\ $\|A\|^2:=\mathrm{tr}(A^{\otimes 2})$,\ $B[A] := \mathrm{tr}(B A^\top)$, and we set $S(x,\alpha) := \left(a(x, \alpha)\right)^{\otimes 2 }$. 
For $l \ge 1$ and $m\ge 1$, let $C_{\uparrow}^{l,m}(\mathbb{R}^{d} \times \Theta)$ be the set of functions $f$ satisfying the following conditions: $(\mathrm{i})$\ $f(x,\theta)$ is $l$ times continuously differentiable with respect to $x$;  $(\mathrm{ii})$\ $f(x,\theta )$ and all its $x$-derivatives up to order $l$ are $m$ times continuously differentiable with respect to $\theta$;\ $(\mathrm{iii})$ $f(x,\theta )$ and all its derivatives are of polynomial growth in $x$, uniformly in $\theta$. Next, for any positive sequence $u_n$, $R:\Theta\times\mathbb{R}\times\mathbb{R}^d\to\mathbb{R}$ denotes a function with a constant $C>0$ such that for all $\theta\in\Theta$ and $x\in\mathbb{R}^d$, $|R(\theta,u_n,x)|\leq u_nC(1+|x|)^C$. $\overset{P}{\to}$ and $\overset{d}{\to}$ indicate convergence in probability and convergence in distribution, respectively. 
Let $I(\theta;\theta_0)$ be the $(p_1+p_2) \times (p_1+p_2)$-matrix defined as 
\begin{align*}
I(\theta;\theta_0)&:=\begin{pmatrix}
I_{a}(\alpha;\alpha_0)&0\\
0&I_b(\theta;\theta_0)
\end{pmatrix}, 
\end{align*}
where
$I_a(\alpha;\alpha_0)=(I_a^{(ij)}(\alpha;\alpha_0))_{1\leq i,j\leq p_1}$, 
$I_b(\theta;\theta_0)=(I_b^{(ij)}(\theta;\theta_0))_{1\leq i,j\leq p_2}$,
\begin{align*}
I_a^{(ij)}(\alpha;\alpha_0)&=\frac{1}{2}\int\mathrm{tr}\left(S^{-1}(\partial_{\alpha_i}S)S^{-1}(\partial_{\alpha_j}S)S^{-1}+S^{-1}(\partial_{\alpha_j}S)S^{-1}(\partial_{\alpha_i}S)S^{-1}\right.\\
&\quad\left. -S^{-1}(\partial_{\alpha_i}\partial_{\alpha_j}S)S^{-1}\right)(x,\alpha)S(x,\alpha_0)\mu_{\theta_0}(dx)\\
&\quad+\frac{1}{2}\int\mathrm{tr}\left(S^{-1}(\partial_{\alpha_i}\partial_{\alpha_j}S)-S^{-1}(\partial_{\alpha_i}S)S^{-1}(\partial_{\alpha_j}S)\right)(x,\alpha) \mu_{\theta_0}(dx),\\
I_b^{(ij)}(\theta;\theta_0)&=\int{(b(x,\beta)-b(x,\beta_0))^\top S^{-1}(x,\alpha)(\partial_{\beta_i}\partial_{\beta_j}b(x,\beta))}\mu_{\theta_0}(dx)\\
&\quad+\int(\partial_{\beta_i}b(x,\beta))^\top S^{-1}(x,\alpha)(\partial_{\beta_j}b(x,\beta))\mu_{\theta_0}(dx).
\end{align*}

We make the following assumptions.
\begin{enumerate}
\item[\bf A1] There exists $C>0$ such that for all $x,y \in\mathbb{R}^d$,

$\displaystyle{\sup_{\alpha \in \Theta_{\alpha}} \|a(x, {\alpha})-a(y, {\alpha} )\|+\sup_{\beta \in \Theta_{\beta}}|b (x,{\beta})-b (y, {\beta})|\leq C|x-y|}.$
\item[\bf A2] The diffusion process $X$ is ergodic with its invariant measure $\mu_{\theta_0}(dx)$.
\item[\bf A3] $\displaystyle \inf_{x, \alpha} \det(S(x, \alpha ))>0$.
\item[\bf A4] For all $p \ge 0,$ $\displaystyle  \sup_{t} E_{\theta}[|X_{t}|^p]<\infty$.
\item[\bf A5] $a \in C_{\uparrow}^{4,3}( \mathbb{R}^{d} \times \Theta_{\alpha})$,\ $b \in C_{\uparrow}^{4,3}( \mathbb{R}^{d} \times \Theta_{\beta})$.
\item[\bf A6] $b(x,\beta) = b(x, \beta_{0} ) \ \ \text{for} \ \mu_{\theta_{0}} \ \text{a.s. all} \ x \Rightarrow \beta = \beta_{0}.$

\hspace{-0.45cm}$S(x,\alpha) = S(x, \alpha_{0} ) \ \text{for} \ \mu_{\theta_{0}} \ \text{a.s. all} \ x \Rightarrow \alpha = \alpha_{0}.$
\item[\bf A7]  $I (\theta_{0};\theta_0)$ is non-singular.
\end{enumerate}

\section{Main results}\label{sec3}
\subsection{Asymptotic distribution under null hypothesis}\label{sec3.1}
\ 

For $1\leq r_1\leq p_1$ and $1\leq r_2\leq p_2$, we consider the following set of tests:
\begin{align*}
\begin{cases}
H_0^{(1)}:\ \alpha_1=\cdots=\alpha_{r_1}=0,\\
H_1^{(1)}:\ \text{not}\ H_0^{(1)},
\end{cases}
\quad
\begin{cases}
H_0^{(2)}:\ \beta_1=\cdots=\beta_{r_2}=0,\\
H_1^{(2)}:\ \text{not}\ H_0^{(2)}.
\end{cases}
\end{align*}
Let $\tilde{\Theta}_\alpha:=\{\alpha\in\Theta_\alpha\ |\ \alpha\  \text{holds}\  H_0^{(1)}\}$, $\tilde{\Theta}_\beta:=\{\beta\in\Theta_\beta\ |\ \beta\  \text{holds}\ H_0^{(2)}\}$, $\tilde{\Theta}:=\tilde{\Theta}_\alpha\times\tilde{\Theta}_\beta$. 

The quasi log likelihood functions are defined as follows:
 \begin{align*}
U_n^{(1)}(\alpha)&:=-\frac{1}{2}\sum_{i=1}^{n}\left\{h_n^{-1}S^{-1}(X_{t_{i-1}^n},\alpha)\left[(X_{t_{i}^n}-X_{t_{i-1}^n})^{\otimes 2}\right]+\log\det S(X_{t_{i-1}^n},\alpha)\right\},\\
U_n^{(2)}(\beta|\bar{\alpha})&:=-\frac{1}{2}\sum_{i=1}^{n}\left\{h_n^{-1}S^{-1}(X_{t_{i-1}^n},\bar{\alpha})\left[(X_{t_{i}^n}-X_{t_{i-1}^n}-h_nb(X_{t_{i-1}^n},\beta))^{\otimes 2}\right]\right\}.
\end{align*}
The adaptive ML type estimator under $\Theta$,
$\hat{\theta}_n=(\hat{\alpha}_n,\hat{\beta}_n)$
and the adaptive ML type estimator under $\tilde{\Theta}$,
$\tilde{\theta}_n=(\tilde{\alpha}_n,\tilde{\beta}_n)$ 
 are defined as
\begin{align*}
\hat{\alpha}_n :=\argsup_{\alpha\in{\Theta}_\alpha}U_n^{(1)}(\alpha),\quad&\tilde{\alpha}_n :=\argsup_{\alpha\in\tilde{\Theta}_\alpha}U_n^{(1)}(\alpha),\\
\hat{\beta}_n :=\argsup_{\beta\in{\Theta}_\beta}U_n^{(2)}(\beta|\hat{\alpha}),\quad&\tilde{\beta}_n :=\argsup_{\beta\in\tilde{\Theta}_\beta}U_n^{(2)}(\beta|\hat{\alpha}).
\end{align*}
Next we set
\begin{align*}
&I_{a,n}(\alpha):=-\frac{1}{n}\partial_\alpha^2U_n^{(1)}(\alpha),\quad I_{b,n}(\beta|\alpha):=-\frac{1}{nh_n}\partial_\beta^2U_n^{(2)}(\beta|\alpha),
\end{align*}
$$J_{a,n}(\alpha):=\{I_{a,n}(\alpha)\ \text{is non-singular}\}, \ J_{b,n}(\bar{\alpha},\beta):=\{I_{b,n}(\beta|\bar{\alpha})\ \text{is non-singular}\},$$
$$
\bar{I}_{a,n}(\alpha):=
\begin{cases}
I_{a,n}^{-1}(\alpha)&(\omega\in J_{a,n}(\alpha)),\\
E_{p_1}&(\omega\in J_{a,n}^C(\alpha)),
\end{cases}
\quad
\bar{I}_{b,n}(\beta|\bar{\alpha}):=
\begin{cases}
I_{b,n}^{-1}(\beta|\bar{\alpha})&(\omega\in J_{b,n}(\bar{\alpha},\beta)),\\
E_{p_2}&(\omega\in J_{b,n}^C(\bar{\alpha},\beta)).
\end{cases}
$$
We define the likelihood ratio, Wald and Rao type test statistics $\Lambda_n^{(i)}, W_n^{(i)},\ R_n^{(i)}\ (i=1,2)$ as follows:
\begin{align*}
\Lambda_n^{(1)}&:=-2(U_n^{(1)}(\tilde{\alpha}_n)-U_n^{(1)}(\hat{\alpha}_n)),\quad
\Lambda_n^{(2)}:=-2(U_n^{(2)}(\tilde{\beta}_n|\hat{\alpha}_n)-U_n^{(2)}(\hat{\beta}_n|\hat{\alpha}_n)),\\
W_n^{(1)}&:=n(\hat{\alpha}_n-\tilde{\alpha}_n)^\top I_{a,n}(\hat{\alpha}_n)(\hat{\alpha}_n-\tilde{\alpha}_n),\quad
W_n^{(2)}:=nh_n(\hat{\beta}_n-\tilde{\beta}_n)^\top I_{b,n}(\hat{\beta}_n|\hat{\alpha}_n)(\hat{\beta}_n-\tilde{\beta}_n),\\
R_n^{(1)}&:=n^{-1}(\partial_{\alpha}U_n^{(1)}(\tilde{\alpha}_n))^\top \bar{I}_{a,n}(\hat{\alpha}_n)\partial_{\alpha}U_n^{(1)}(\tilde{\alpha}_n),\\
R_n^{(2)}&:=(nh_n)^{-1}(\partial_{\beta}U_n^{(2)}(\tilde{\beta}_n|\hat{\alpha}_n))^\top \bar{I}_{b,n}(\hat{\beta}_n|\hat{\alpha}_n)\partial_{\beta}U_n^{(2)}(\tilde{\beta}_n|\hat{\alpha}_n),
\end{align*}
where $\Lambda_n^{(1)}, W_n^{(1)}$ and $R_n^{(1)}$ are used in testing $\alpha$ and $\Lambda_n^{(2)}, W_n^{(2)}$ and $R_n^{(2)}$ are used in testing $\beta$.

The following gives asymptotic distributions of these test statistics under null hypothesis.
\begin{theorem}\label{thm:1}
Assume $\mathbf{A1\mathchar`- A7}$. If \ $h_n\to0,\ nh_n\to\infty$  and\ $nh_n^2\to0$, then
\begin{enumerate}
\item[$(\mathrm{i})$] $\Lambda_n^{(1)}\overset{d}{\to}\chi_{r_1}^2\quad(\text{under}\  H_0^{(1)}),\quad \Lambda_n^{(2)}\overset{d}{\to}\chi_{r_2}^2\quad(\text{under}\  H_0^{(2)}),$
\item[$(\mathrm{ii})$] $W_n^{(1)}\overset{d}{\to}\chi_{r_1}^2\quad(\text{under}\  H_0^{(1)}),\quad W_n^{(2)}\overset{d}{\to}\chi_{r_2}^2\quad(\text{under}\  H_0^{(2)}),$
\item[$(\mathrm{iii})$] $R_n^{(1)}\overset{d}{\to}\chi_{r_1}^2\quad(\text{under}\  H_0^{(1)}),\quad R_n^{(2)}\overset{d}{\to}\chi_{r_2}^2\quad(\text{under}\  H_0^{(2)}).$
\end{enumerate}
\end{theorem}
\begin{remark}\label{test.method}
When we perform the parametric tests for $\theta=(\alpha,\beta)$, the testing procedure is as follows:
\begin{enumerate}
\item[$(\mathrm{1})$] test $H_0^{(1)}$ v.s. $H_1^{(1)}$;
\item[$(\mathrm{2})$] test $H_0^{(2)}$ v.s. $H_1^{(2)}$ regardless of the result of $(\mathrm{1})$.
\end{enumerate}
This testing method provides four interpretations, which gives more information about the parameters than the joint test by Kitagawa and Uchida \cite{Kitagawa_2014}.   
\end{remark}
\subsection{Consistency of tests}\label{sec3.2}
\ 

Next we consider under alternatives $H_1^{(1)}$ or $H_1^{(2)}$. To emphasize that we consider under alternatives, $\theta_1^*=(\alpha_1^*,\beta_1^*)$ denotes the true parameter under alternatives. We define an optimal parameter $\theta^*=(\alpha^*,\beta^*)\in\tilde{\Theta}$ as
$$
\alpha^*:=\argsup_{\alpha\in\tilde{\Theta}_\alpha}\bar{U}_1(\alpha;\alpha_1^*),\quad
\beta^*:=\argsup_{\alpha\in\tilde{\Theta}_\beta}\bar{U}_2(\alpha^*_1,\beta;\beta_1^*),
$$
where
\begin{align*}
\bar{U}_1(\alpha;\alpha_1^*)&:=-\frac{1}{2}\int\left(\mathrm{tr}\left(S(x,\alpha_1^*)S^{-1}(x,\alpha)\right)+\log\det S(x,\alpha)\right)\mu_{\theta_1^*}(dx),\\
\bar{U}_2(\alpha,\beta;\beta_1^*)&:=-\frac{1}{2}\int S^{-1}(x,\alpha)\left[\left(b(x,\beta_1^*)-b(x,\beta)\right)^{\otimes 2}\right]\mu_{\theta_1^*}(dx).
\end{align*}
\begin{remark}\label{rem:1}
It always true that $\alpha_1^*\neq \alpha^*$. Since $\alpha_1^*$ is the true value under $H_1^{(1)}$, there exists $i \in\{1,\ldots,r_1\}$
such that $\alpha_i\neq 0$. On the other hand,  $\alpha^*$ is the optimal value under $H_1^{(0)}$, hence $\alpha_i= 0\quad(i=1,\ldots, r_1)$. Similarly, it always holds $\beta_1^*\neq \beta^*$.
\end{remark}
We assume the following conditions.
\begin{enumerate}
\item[\bf B1] \begin{enumerate}[(a)]
\item For all $\varepsilon>0$, 
$$
\sup_{\{\alpha\in\tilde{\Theta}_\alpha;|\alpha-\alpha^*|\geq\varepsilon\}}
\left(\bar{U}_1(\alpha;\alpha_1^*)-\bar{U}_1(\alpha^*;\alpha_1^*)\right)<0.
$$

\item For all $\varepsilon>0$,
$$
\sup_{\{\beta \in\tilde{\Theta}_\beta;|\beta-\beta^*|\geq\varepsilon\}}
\left(\bar{U}_2(\alpha^*_1, \beta;\beta_1^*)-\bar{U}_2(\alpha^*_1, \beta^*;\beta_1^*)\right)
<0.
$$
\end{enumerate}
\item[\bf B2]  \begin{enumerate}[(a)]
\item For all $\alpha\in\Theta_\alpha$,\ $I_a(\alpha;\alpha_1^*)$ is non-singular.

\item For all $\beta\in\Theta_\beta$,\ $I_b(\alpha_1^*,\beta;\theta_1^*)$ is non-singular.
\end{enumerate}
\end{enumerate}

The following gives the consistency of tests. 
\begin{theorem} \label{thm:2}
Assume $\mathbf{A1\mathchar`-A7}$ and $\mathbf{B1}$. If $h_n\to0$ and$\ nh_n\to\infty$, then for all $\varepsilon\in (0,1)$,
\begin{enumerate}

\item[$(\mathrm{i})$] $P(\Lambda_n^{(1)}\geq \chi_{r_1,\varepsilon}^2 ){\to}\ 1\quad(\text{under}\ H_1^{(1)}),\quad P(\Lambda_n^{(2)}\geq \chi_{r_2,\varepsilon}^2 ){\to}\ 1\quad(\text{under}\ H_1^{(2)}),$
\item[$(\mathrm{ii})$] $P(W_n^{(1)}\geq \chi_{r_1,\varepsilon}^2 ){\to}\ 1\quad(\text{under}\ H_1^{(1)}),\quad P(W_n^{(2)}\geq \chi_{r_2,\varepsilon}^2 ){\to}\ 1\quad(\text{under}\ H_1^{(2)}),$
\item[$(\mathrm{iii})$] $P(R_n^{(1)}\geq \chi_{r_1,\varepsilon}^2 ){\to}\ 1\quad(\text{under}\ H_1^{(1)}\ \text{and}\ \mathbf{B2}$-$(\mathrm{a})),$

\hspace{-0.45cm}$P(R_n^{(2)}\geq \chi_{r_2,\varepsilon}^2 ){\to}\ 1\quad(\text{under}\ H_1^{(2)}\  \text{and}\  \mathbf{B2}$-$(\mathrm{b})),$
\end{enumerate}
where $\chi_{p,\varepsilon}^2$ denotes  the upper $\varepsilon$ point of chi-squared distribution with $p$ degrees of freedom.
\end{theorem}
\subsection{Asymptotic distribution under local alternatives}
\ 

For $u_\alpha\in\mathbb{R}^{p_1}$ and $u_\beta\in\mathbb{R}^{p_2}$, we consider the following tests.
$$
\begin{cases}
H_{0,n}^{(1)}:\ \alpha=\alpha_0,\\
H_{1,n}^{(1)}:\ \alpha=\alpha_{1,n}^*:=\alpha_0+\frac{u_\alpha}{\sqrt{n}},
\end{cases}
\quad
\begin{cases}
H_{0,n}^{(2)}:\ \beta=\beta_0,\\
H_{1,n}^{(2)}:\ \beta=\beta_{1,n}^*:=\beta_0+\frac{u_\beta}{\sqrt{nh_n}}.
\end{cases}
$$
Let $\theta_{1,n}^*$ denote the true parameter under local alternatives. We assume the following condition.
\begin{enumerate}

\item[\bf C1] $P_{\theta_{1,n}^*}$ is contiguous with respect to $P_{\theta_0}$. That is,\ for sequence of sets $A_n$,
$$
\lim_{n\to\infty}P_{\theta_0}(A_n)=0\ \Rightarrow  \lim_{n\to\infty}P_{\theta_{1,n}^*}(A_n)=0
$$
\end{enumerate}
\begin{remark}\label{rem:gobet}
$(\mathrm{i})$ Let $\bar{\alpha}$ denote the true parameter of $\alpha$ under $H_{0,n}^{(2)}$ or $H_{1,n}^{(2)}$, and $\bar{\beta}$ denote the true parameter of $\beta$ under $H_{0,n}^{(1)}$ or $H_{1,n}^{(1)}$. 
By using these symbols, {\bf C1} is applied as follows:
\begin{enumerate}
\item[$(\mathrm{I})$]\  
If we test for $\alpha$, then  $\theta_0=(\alpha_0,\bar{\beta}),\ \theta_{1,n}^*=(\alpha_{1,n}^*,\bar{\beta}).$
\item[$(\mathrm{II})$]\ 
If we test for $\beta$, then  $\theta_0=(\bar{\alpha},\beta_0),\ \theta_{1,n}^*=(\bar{\alpha},\beta_{1,n}^*).$
\end{enumerate}
$(\mathrm{ii})$\ A sufficient condition of {\bf C1} is local asymptotic normality (LAN), and a sufficient condition of LAN is 
as follows: 
$\mathbf {A1\mathchar`-A5}$ and
\begin{enumerate}

\item[$(\mathrm{I})$]\  There exists $C >0$ such that for all $x\in\mathbb{R}^d$ and $(\alpha,\beta)\in \Theta$,\\
$|b(x,\beta)| \leq C(1+|x|),\quad|\partial_{i}b(x,\beta)|+|a(x,\alpha)|+|\partial_{i}a(x,\alpha)| \leq C.$
\item[$(\mathrm{II})$]\ There exists $c_0>0$ and $K >0$ such that for all $(x,\beta) \in \mathbb{R}^d\times \Theta_\beta$,\\
$b(x,\beta)^\top x \leq -c_0|x|^2 + K.$
\end{enumerate}
For details of the relation between LAN and {\bf C1},  see,  for example, van der Vaart\cite{van_der_Vaart}. For sufficient conditions of LAN, see Gobet\cite{Gobet_2002}.
\end{remark}

\begin{theorem} \label{thm:3} 
Assume $\mathbf{A1\mathchar`-A7}$ and $\mathbf{C1}$. If $h_n\to0,\ nh_n\to\infty$ and $nh_n^2\to0$, then
\begin{enumerate}

\item [$(\mathrm{i})$] $\Lambda_n^{(1)}\overset{d}{\to}\chi_{p_1}^2(u_\alpha^\top I_a(\alpha_0;\alpha_0)u_\alpha)\quad(\text{under}\ H_{1,n}^{(1)}),\quad
\Lambda_n^{(2)}\overset{d}{\to}\chi_{p_2}^2(u_\beta^\top I_b(\theta_0;\theta_0)u_\beta)\quad(\text{under}\ H_{1,n}^{(2)}),$
\item [$(\mathrm{ii})$] $W_n^{(1)}\overset{d}{\to}\chi_{p_1}^2(u_\alpha^\top I_a(\alpha_0;\alpha_0)u_\alpha)\quad(\text{under}\ H_{1,n}^{(1)}),\quad
W_n^{(2)}\overset{d}{\to}\chi_{p_2}^2(u_\beta^\top I_b(\theta_0;\theta_0)u_\beta)\quad(\text{under}\ H_{1,n}^{(2)}),$
\item [$(\mathrm{iii})$] $R_n^{(1)}\overset{d}{\to}\chi_{p_1}^2(u_\alpha^\top I_a(\alpha_0;\alpha_0)u_\alpha)\quad(\text{under}\ H_{1,n}^{(1)}),\quad
R_n^{(2)}\overset{d}{\to}\chi_{p_2}^2(u_\beta^\top I_b(\theta_0;\theta_0)u_\beta)\quad(\text{under}\ H_{1,n}^{(2)}),$
\end{enumerate}
where $\chi_{p}^2(c)$ denotes the 
non-central chi-squared distribution with $p$ degrees of freedom and noncentrality parameter $c$.
\end{theorem}


\section{Examples and simulations}\label{sec4}
\subsection{Model 1}\label{sec:model1}
\ 

We consider the following 1-dimensional Ornstein-Uhlenbeck process:
\begin{equation}\label{simu_model1}
\begin{cases}
dX_t = -(X_t-\beta)dt + \alpha dW_t,\\
X_0 = 1.0.
\end{cases}
\end{equation}
We simulate the asymptotic performance of the three types of test statistics: likelihood ratio type, Wald type and Rao's score type. In model \eqref{simu_model1}, we deal with 
the following hypothesis tests:
\begin{equation}\label{simu_test1}
\begin{cases}
H_{0}^{(1)}:\ \alpha=1.0,\\
H_{1}^{(1)}:\ \alpha\neq 1.0,
\end{cases}
\quad
\begin{cases}
H_{0}^{(2)}:\ \beta=2.0,\\
H_{1}^{(2)}:\ \beta\neq 2.0.
\end{cases}
\end{equation}
These tests derive the four kinds of results as follows:
\begin{enumerate}[\bf{Case} 1.]
\item Neither $\alpha$ nor $\beta$ is rejected;
\item $\alpha$ is not rejected, but $\beta$ is rejected;
\item $\alpha$ is rejected, but $\beta$ is not rejected;
\item Both $\alpha$ and $\beta$ are rejected;
\end{enumerate}
We choice the true parameter $\theta^*=(\alpha^*,\beta^*)$ from $\{(1.0,2.0),\ (1.0,2.5),\ (1.1,2.0),\ (1.1,2.5)\}$, 
where 
$\theta^*$ corresponds to the true parameter of Cases 1-4.
Let $n$ be fix and $h_n=n^{-2/3}$, which satisfies the conditions $nh_n=n^{1/3}\to \infty$ and $nh_n^2=n^{-1/3}\to 0$ as $n\to\infty$. In this simulation, we consider the cases of $n=10^4, 10^5$ and $10^6$. 
Let the significance level 
denote $\varepsilon=0.05$ and each test is rejected when the realization of each test statistic is greater than $\chi_{1,0.05}^2$. The simulation is repeated 1000 times.

Tables \ref{model1_table1-1}-\ref{model1_table1-4} 
show 
the number of counts of 
Cases 1-4 selected by the tests \eqref{simu_test1}, 
where the true parameter $\theta^*$ 
corresponds to each case. 
In all of the tables, 
the true case is most often selected 
as $n$ increases. In Table 1, the percent of misidentification is about 10 percent, which is 
the sum of the significance levels of the considering tests. 
Figures \ref{model1_hist1_1}-\ref{model1_ecdf3} 
are simulation results of 
the histograms and the empirical distributions 
of the three types of test statistics 
in Theorem $1$. Theoretically, these test statistics 
converge in distribution to $\chi_1^2$, and 
we see  from all of the figures that 
these test statistics have good behavior.
Tables \ref{model1_size_table1-1}-\ref{model1_size_table3} 
show 
the empirical sizes of the three test statistics when the null hypothesis is true.
\begin{small}
\begin{table}[htbp]
\begin{center}
\caption{Results of test statistics in {\bf{Case 1}}:\ $(\alpha^*,\beta^*)=(1.0,2.0).$}
\label{model1_table1-1}
\begin{tabular}{ccc|c|cccc}
$n$&$h_n$&$nh_n$&Test type&{\bf{Case 1}}&Case 2&Case 3&Case 4\\
\hline\hline
&&&Likelihood					&890&53&53&4\\
$10^4$&$2.15\times10^{-3}$&21.5&Wald&891&54&52&3\\
&&&Rao							&891&52&52&5\\
\hline
&&&Likelihood					&897&50&50&3\\
$10^5$&$4.64\times10^{-4}$&46.4&Wald&897&50&50&3\\
&&&Rao							&897&50&50&3\\
\hline
&&&Likelihood					&900&49&47&4\\
$10^6$&$1.00\times10^{-4}$&100&Wald&900&49&47&4\\
&&&Rao							&901&49&46&4\\
\end{tabular}
\end{center}
\end{table}

\begin{table}[htbp]
\begin{center}
\caption{Results of test statistics in {\bf{Case 2}}:\ $(\alpha^*,\beta^*)=(1.0,2.5).$}
\label{model1_table1-2}
\begin{tabular}{ccc|c|cccc}
$n$&$h_n$&$nh_n$&Test type&Case 1&{\bf{Case 2}}&Case 3&Case 4\\
\hline\hline
&&&Likelihood					&352&590&23&35\\
$10^4$&$2.15\times10^{-3}$&21.5&Wald&353&591&22&34\\
&&&Rao							&352&590&23&35\\
\hline
&&&Likelihood					&74&881&3&42\\
$10^5$&$4.64\times10^{-4}$&46.4&Wald&74&883&3&40\\
&&&Rao							&74&880&3&43\\
\hline
&&&Likelihood					&1&950&0&49\\
$10^6$&$1.00\times10^{-4}$&100&Wald&1&950&0&49\\
&&&Rao							&1&950&0&49\\
\end{tabular}
\end{center}
\end{table}

\begin{table}[htbp]
\begin{center}
\caption{Results of test statistics in {\bf{Case 3}}:\ $(\alpha^*,\beta^*)=(1.1,2.0).$}
\label{model1_table1-3}
\begin{tabular}{ccc|c|cccc}
$n$&$h_n$&$nh_n$&Test type&Case 1&Case 2&{\bf{Case 3}}&Case 4\\
\hline\hline
&&&Likelihood					&0&0&956&44\\
$10^4$&$2.15\times10^{-3}$&21.5&Wald&0&0&956&44\\
&&&Rao							&0&0&956&44\\
\hline
&&&Likelihood					&0&0&954&46\\
$10^5$&$4.64\times10^{-4}$&46.4&Wald&0&0&954&46\\
&&&Rao							&0&0&954&46\\
\hline
&&&Likelihood					&0&0&955&45\\
$10^6$&$1.00\times10^{-4}$&100&Wald&0&0&955&45\\
&&&Rao							&0&0&955&45\\
\end{tabular}
\end{center}
\end{table}

\begin{table}[htbp]
\begin{center}
\caption{Results of test statistics in {\bf{Case 4}}:\ $(\alpha^*,\beta^*)=(1.1,2.5).$}
\label{model1_table1-4}
\begin{tabular}{ccc|c|cccc}
$n$&$h_n$&$nh_n$&Test type&Case 1&Case 2&Case 3&{\bf{Case 4}}\\
\hline\hline
&&&Likelihood					&0&0&425&575\\
$10^4$&$2.15\times10^{-3}$&21.5&Wald&0&0&425&575\\
&&&Rao							&0&0&425&575\\
\hline
&&&Likelihood					&0&0&114&886\\
$10^5$&$4.64\times10^{-4}$&46.4&Wald&0&0&114&886\\
&&&Rao							&0&0&114&886\\
\hline
&&&Likelihood					&0&0&2&998\\
$10^6$&$1.00\times10^{-4}$&100&Wald&0&0&2&998\\
&&&Rao							&0&0&2&998\\
\end{tabular}
\end{center}
\end{table}

\end{small}

\begin{figure}[htbp]
\centering
\includegraphics[height=0.39\vsize]{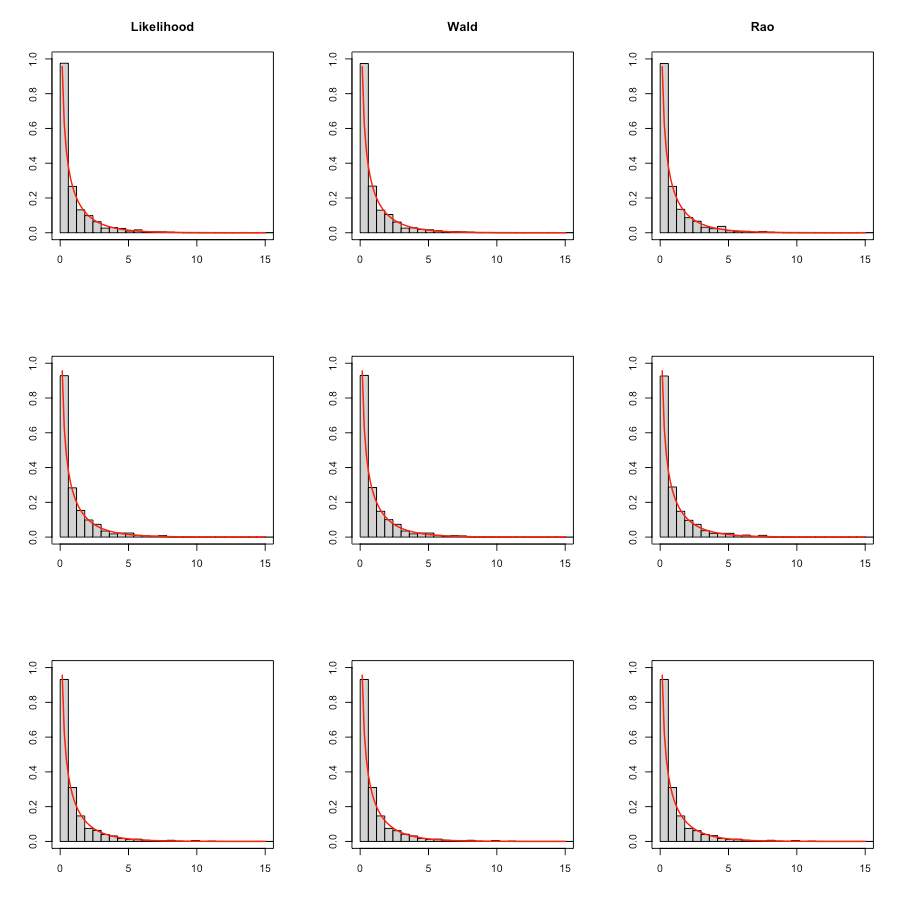}
\caption{Histograms of 
the three types of 
test statistics for $\alpha$ in {\bf{Case 1}}. 
Each row of figures corresponds to the case of 
$n=10^4$ (upper), $10^5$ (middle) and $10^6$ (bottom) 
and each column of figures corresponds to Likelihood ratio (left), Wald (middle) and Rao (right) type test statistics. 
The red line is the probability density function of $\chi_1^2$.}
\label{model1_hist1_1}
\end{figure}

\begin{figure}[htbp]
\centering
\includegraphics[height=0.39\vsize]{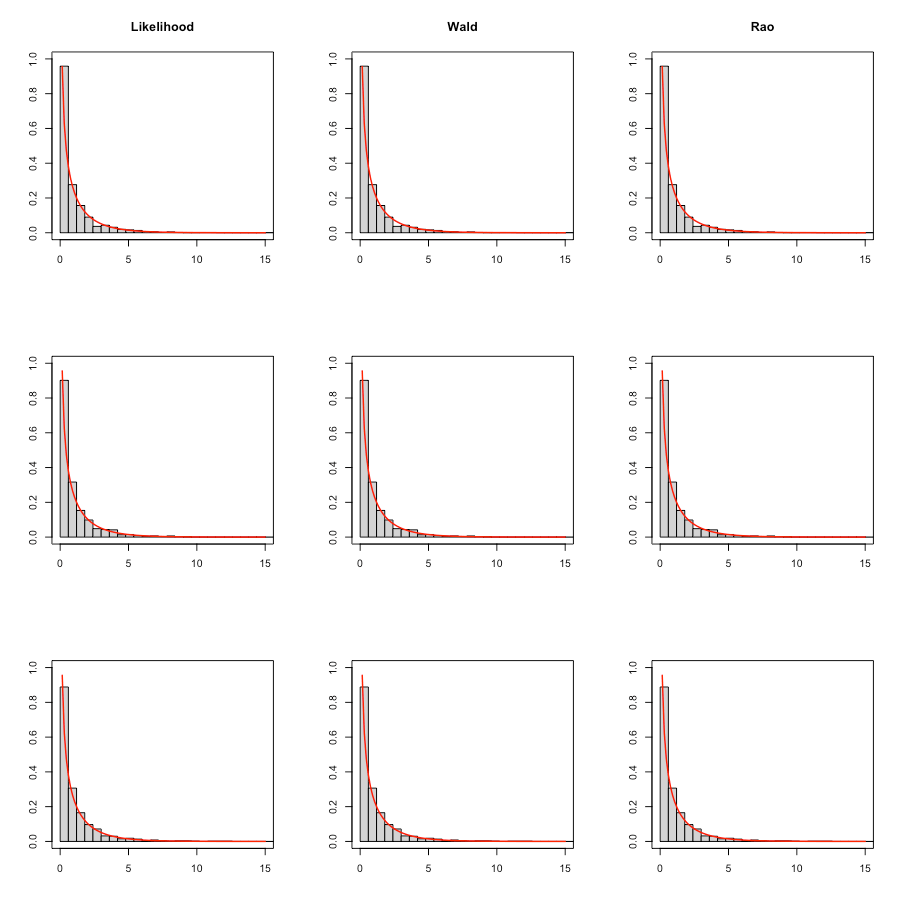}
\caption{Histograms of the three types of test statistics for $\beta$ in {\bf{Case 1}}. 
Each row of figures corresponds to the case of 
$n=10^4$ (upper), $10^5$ (middle) and $10^6$ (bottom) 
and each column of figures corresponds to Likelihood ratio (left), Wald (middle) and Rao (right) type test statistics. 
The red line is the probability density function of $\chi_1^2$.}
\label{model1_hist1_2}
\end{figure}

\begin{figure}[htbp]
\centering
\includegraphics[height=0.39\vsize]{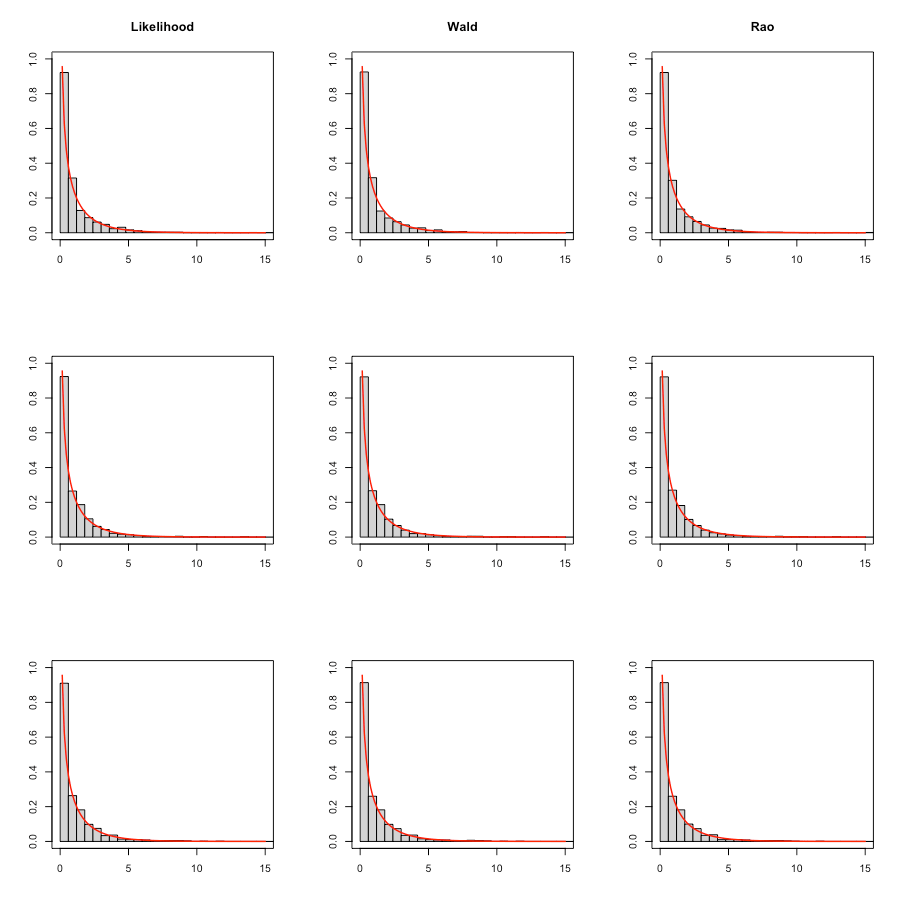}
\caption{Histograms of the three types of test statistics for $\alpha$ in {\bf{Case 2}}.
Each row of figures corresponds to the case of 
$n=10^4$ (upper), $10^5$ (middle) and $10^6$ (bottom) 
and each column of figures corresponds to Likelihood ratio (left), Wald (middle) and Rao (right) type test statistics. 
The red line is the probability density function of $\chi_1^2$.}
\label{model1_hist2}
\end{figure}

\begin{figure}[htbp]
\centering
\includegraphics[height=0.39\vsize]{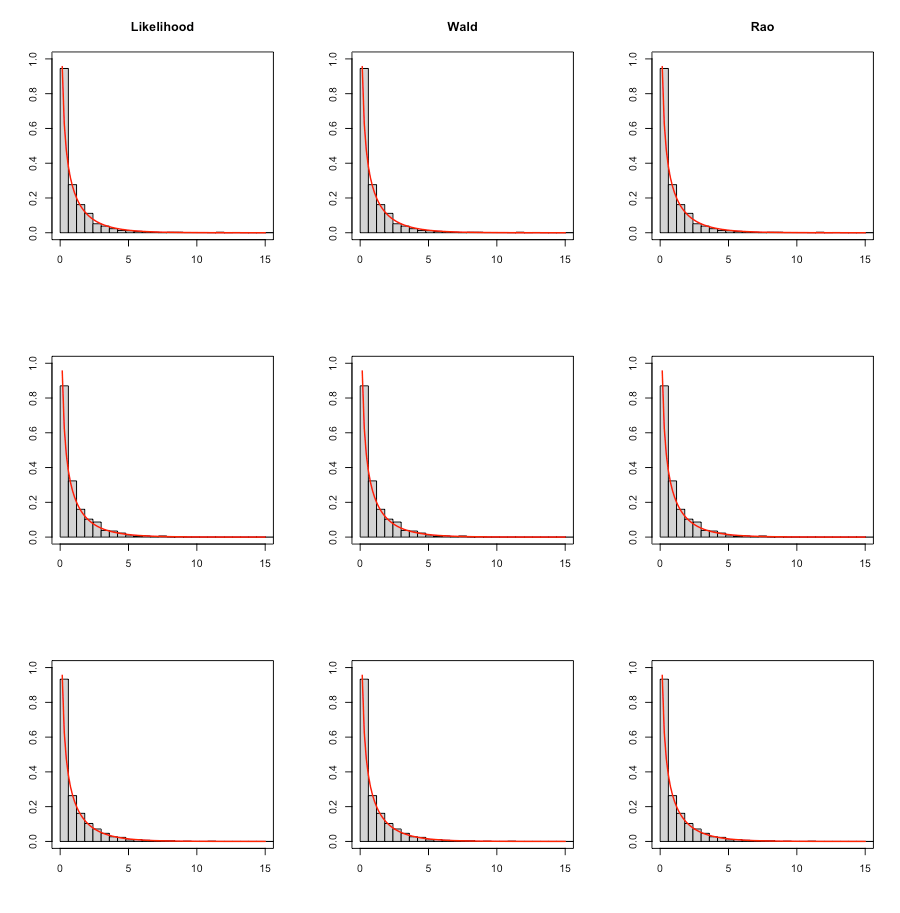}
\caption{Histograms of the three types of test statistics for $\beta$ in {\bf{Case 3}}.
Each row of figures corresponds to the case of 
$n=10^4$ (upper), $10^5$ (middle) and $10^6$ (bottom) 
and each column of figures corresponds to Likelihood ratio (left), Wald (middle) and Rao (right) type test statistics. 
The red line is the probability density function of $\chi_1^2$.}
\label{model1_hist3}
\end{figure}

\begin{figure}[htbp]
\centering
\includegraphics[height=0.39\vsize]{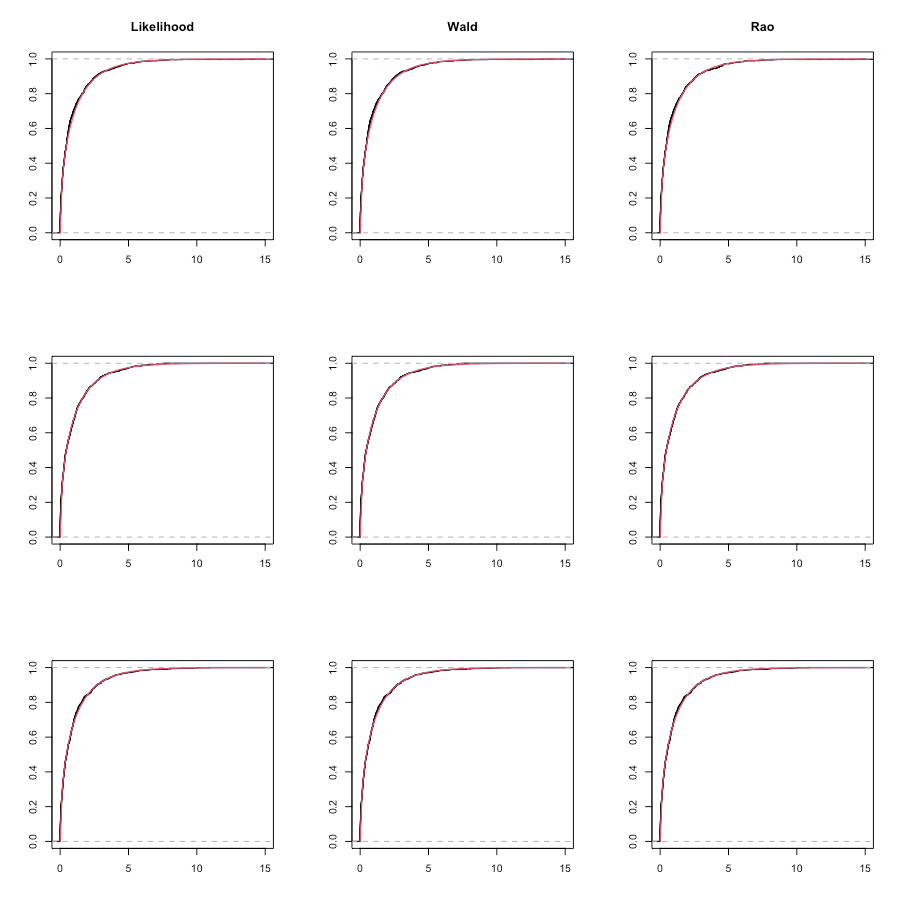}
\caption{Empirical distributions of the three types of test statistics for $\alpha$ in {\bf{Case 1}}. 
Each row of figures corresponds to the case of 
$n=10^4$ (upper), $10^5$ (middle) and $10^6$ (bottom) 
and each column of figures corresponds to Likelihood ratio (left), Wald (middle) and Rao (right) type test statistics. The red line is the cumulative distribution function of $\chi_1^2$.}
\label{model1_ecdf1_1}
\end{figure}

\begin{figure}[htbp]
\centering
\includegraphics[height=0.39\vsize]{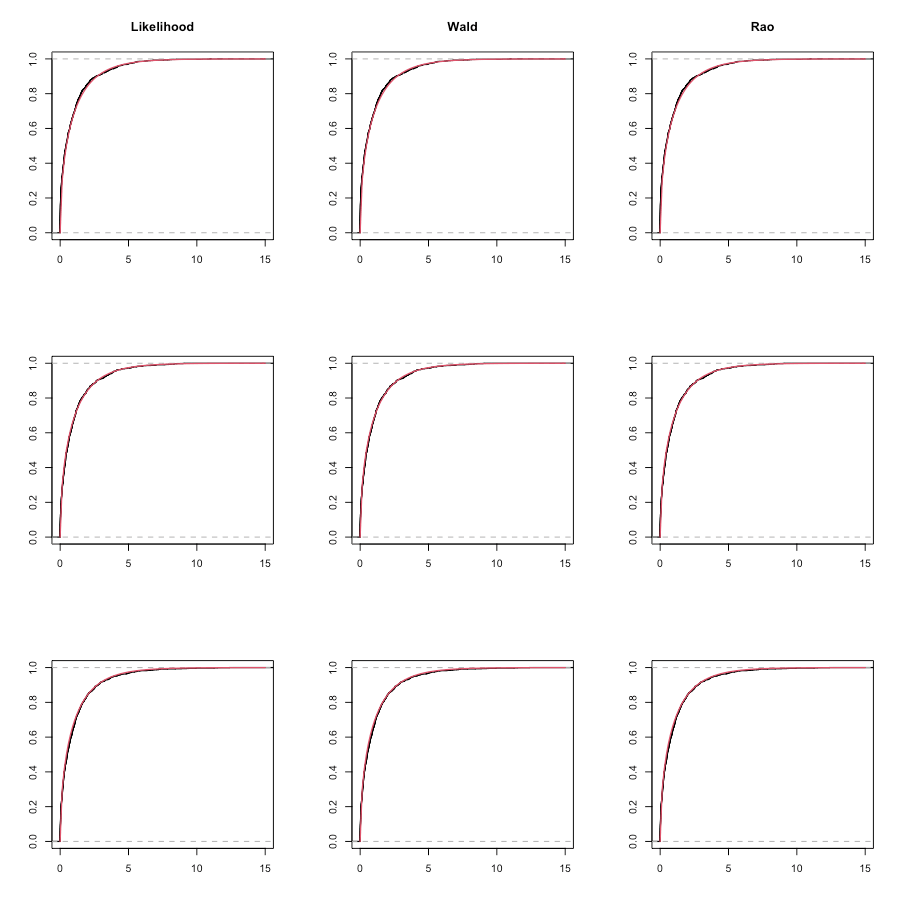}
\caption{Empirical distributions of the three types of test statistics for $\beta$ in {\bf{Case 1}}.  Each row of figures corresponds to the case of 
$n=10^4$ (upper), $10^5$ (middle) and $10^6$ (bottom) 
and each column of figures corresponds to Likelihood ratio (left), Wald (middle) and Rao (right) type test statistics. The red line is the cumulative distribution function of $\chi_1^2$.}
\label{model1_ecdf1_2}
\end{figure}

\begin{figure}[htbp]
\centering
\includegraphics[height=0.39\vsize]{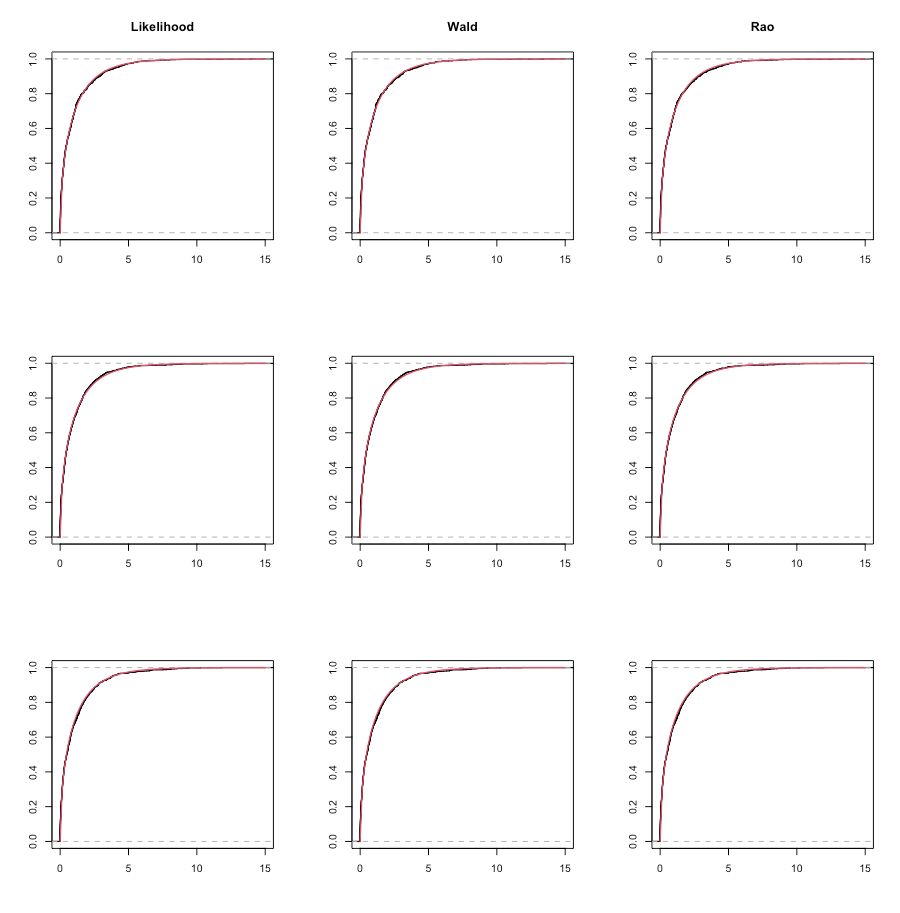}
\caption{Empirical distributions of the three types of test statistics for $\alpha$ in {\bf{Case 2}}.  Each row of figures corresponds to the case of 
$n=10^4$ (upper), $10^5$ (middle) and $10^6$ (bottom) 
and each column of figures corresponds to Likelihood ratio (left), Wald (middle) and Rao (right) type test statistics. The red line is the cumulative distribution function of $\chi_1^2$.}
\label{model1_ecdf2}
\end{figure}

\begin{figure}[htbp]
\centering
\includegraphics[height=0.39\vsize]{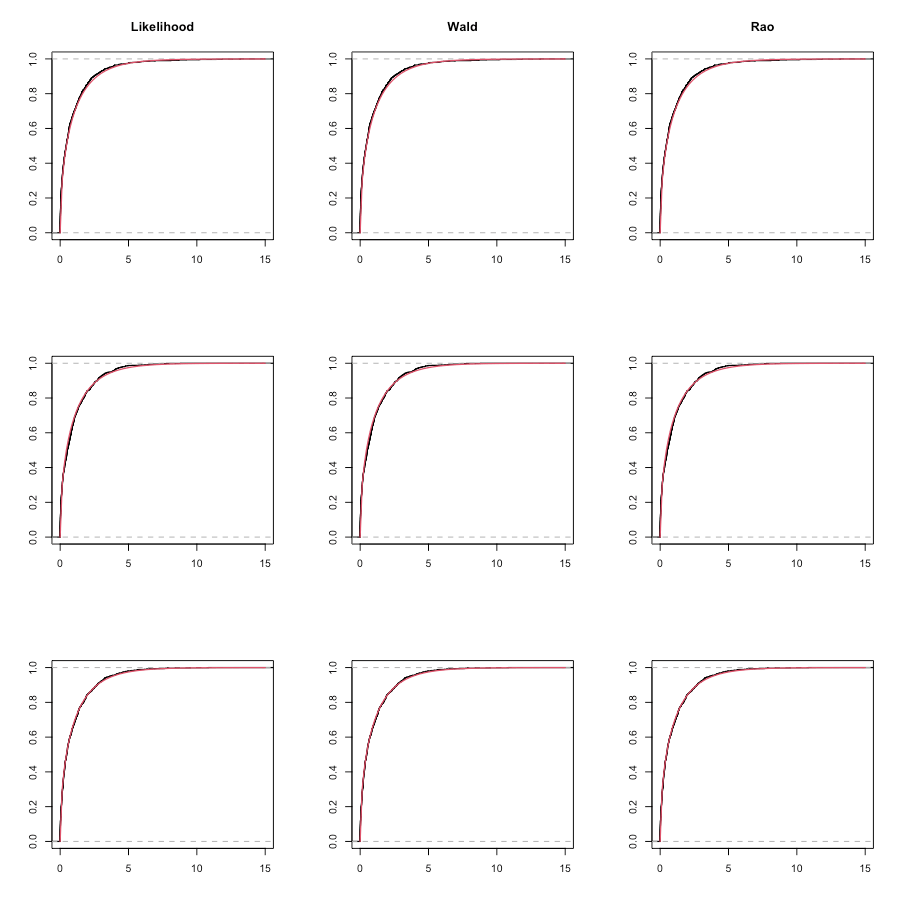}
\caption{Empirical distributions of the three types of test statistics for $\beta$ in {\bf{Case 3}}. Each row of figures corresponds to the case of 
$n=10^4$ (upper), $10^5$ (middle) and $10^6$ (bottom) 
and each column of figures corresponds to Likelihood ratio (left), Wald (middle) and Rao (right) type test statistics. The red line is the cumulative distribution function of $\chi_1^2$.}
\label{model1_ecdf3}
\end{figure}

\begin{table}[htbp]
\begin{minipage}[htbp]{.49\hsize}
\begin{center}
\caption{Empirical sizes of 
the test statistics
for $\alpha$ in $\mathbf{Case\ 1}.$}
\label{model1_size_table1-1}
\begin{tabular}{cc|ccc}
$n$&$nh_n$&Likelihood&Wald&Rao\\
\hline\hline
$10^4$&21.5&0.057&0.055&0.057\\
$10^5$&46.4&0.053&0.053&0.053\\
$10^6$&100&0.051&0.051&0.050
\end{tabular}
\end{center} 
\end{minipage}
\hfill
\begin{minipage}[htbp]{.49\hsize}
\begin{center}
\caption{Empirical sizes of 
the test statistics  for $\beta$ in $\mathbf{Case\ 1}.$}
\label{model1_size_table1-2}
\begin{tabular}{cc|ccc}
$n$&$nh_n$&Likelihood&Wald&Rao\\
\hline\hline
$10^4$&21.5&0.057&0.057&0.057\\
$10^5$&46.4&0.053&0.053&0.053\\
$10^6$&100&0.053&0.053&0.053
\end{tabular}
\end{center}
\end{minipage}
\end{table}

\begin{table}[htbp]
\begin{minipage}[htbp]{.49\hsize}
\begin{center}
\caption{Empirical sizes of 
the test statistics  for $\alpha$ in $\mathbf{Case\ 2}.$}
\label{model1_size_table2}
\begin{tabular}{cc|ccc}
$n$&$nh_n$&Likelihood&Wald&Rao\\
\hline\hline
$10^4$&21.5&0.058&0.056&0.058\\
$10^5$&46.4&0.045&0.043&0.046\\
$10^6$&100&0.049&0.049&0.049
\end{tabular}
\end{center} 
\end{minipage}
\hfill
\begin{minipage}[htbp]{.49\hsize}
\begin{center}
\caption{Empirical sizes of 
the test statistics for $\beta$ in $\mathbf{Case\ 3}.$}
\label{model1_size_table3}
\begin{tabular}{cc|ccc}
$n$&$nh_n$&Likelihood&Wald&Rao\\
\hline\hline
$10^4$&21.5&0.044&0.044&0.044\\
$10^5$&46.4&0.046&0.046&0.046\\
$10^6$&100&0.045&0.045&0.045
\end{tabular}
\end{center}
\end{minipage}
\end{table}

In order to check consistency of the tests, 
we treat the three kinds of $\alpha^*$ ($1.1, 1.01$ and $1.001$) in $\mathbf{Case\ 3}$
and the three kinds of $\beta^*$ ($3.0, 2.5$ and $2.1$) in $\mathbf{Case\ 2}$.
Tables \ref{model1_power_table3} and \ref{model1_power_table2} 
are simulation results of
the empirical powers of the three test statistics. 
For 
$\alpha^* = 1.1$, $1.01$ and $\beta^* = 3.0$, $2.5$, 
the empirical powers increase and tend to 1 as $n$ increases.

\begin{table}[bp]
\begin{minipage}[htbp]{.49\hsize}
\begin{center}
\caption{Empirical sizes of 
the test statistics  for $\alpha$ in $\mathbf{Case\ 3}.$}
\label{model1_power_table3}
\begin{tabular}{c|c|ccc}
$\alpha^*$&$n$&Likelihood&Wald&Rao\\
\hline\hline
&$10^4$&1.000&1.000&1.000\\
1.1&$10^5$&1.000&1.000&1.000\\
&$10^6$&1.000&1.000&1.000\\
\hline
&$10^4$&0.326&0.316&0.349\\
1.01&$10^5$&0.994&0.994&0.994\\
&$10^6$&1.000&1.000&1.000\\
\hline
&$10^4$&0.050&0.047&0.048\\
1.001&$10^5$&0.095&0.093&0.098\\
&$10^6$&0.295&0.294&0.297
\end{tabular}
\end{center} 
\end{minipage}
\hfill
\begin{minipage}[htbp]{.49\hsize}
\begin{center}
\caption{
Empirical sizes of 
the test statistics 
for $\beta$ in $\mathbf{Case\ 2}.$}
\label{model1_power_table2}
\begin{tabular}{c|c|ccc}
$\beta^*$&$nh_n$&Likelihood&Wald&Rao\\
\hline\hline
&$21.5$&0.994&0.994&0.994\\
3.0&$46.4$&1.000&1.000&1.000\\
&$100$&1.000&1.000&1.000\\
\hline
&$21.5$&0.625&0.625&0.625\\
2.5&$46.4$&0.923&0.923&0.923\\
&$100$&0.999&0.999&0.999\\
\hline
&$21.5$&0.091&0.091&0.091\\
2.1&$46.4$&0.098&0.098&0.098\\
&$100$&0.154&0.154&0.154\\
\end{tabular}
\end{center}
\end{minipage}
\end{table}

Next, we consider the asymptotic distributions of the three test statistics under local alternatives. 
The hypothesis tests are
defined as
\begin{equation}\notag
\begin{cases}
H_{0,n}^{(1)}:\ \alpha=\alpha_0=1.0,\\
H_{1,n}^{(1)}:\ \alpha=\alpha_{1,n}^*=\alpha_0+\frac{u_\alpha}{\sqrt{n}},
\end{cases}
\quad
\begin{cases}
H_{0,n}^{(2)}:\ \beta=\beta_0=2.0,\\
H_{1,n}^{(2)}:\ \beta=\beta_{1,n}^*=\beta_0+\frac{u_\beta}{\sqrt{nh_n}},
\end{cases}
\end{equation}
where we set $u_\alpha=5.0$ and $u_\beta=2.0$. In the Ornstein-Uhlenbeck model \eqref{simu_model1}, the invariant distribution is the normal distribution with mean $\beta$ and variance $\frac{\alpha^2}{2}$. Therefore we calculate
$$
I_a(\alpha_0,\alpha_0)=\frac{2.0}{\alpha_0^2}=2.0,\quad I_b(\theta_0,\theta_0)=\frac{1.0}{\alpha^2_0}=1.0,
$$
$$I_a(\alpha_0,\alpha_0)u_\alpha^2=50,\quad I_b(\theta_0,\theta_0)u_\beta^2=4.0.$$

Figures \ref{model1_local_hist1}-\ref{model1_local_ecdf2} show 
the histograms and the empirical distributions 
of 
the three test statistics
under local alternatives. Theoretically, 
the asymptotic distributions of the three test statistics for $\alpha$ are $\chi_1^2(50)$ 
and those for $\beta$ are $\chi_1^2(4)$. 
We see from Figures \ref{model1_local_hist1}-\ref{model1_local_ecdf2}  that
as $n$ increases, the three test statistics for $\alpha$ and $\beta$
are approximately  distributed as $\chi_1^2(50)$  and $\chi_1^2(4)$, respectively.
In particular, the Likelihood ratio type test statistic for $\alpha$ has good performance.
\newpage
\begin{figure}[htbp]
\centering
\includegraphics[height=0.35\vsize]{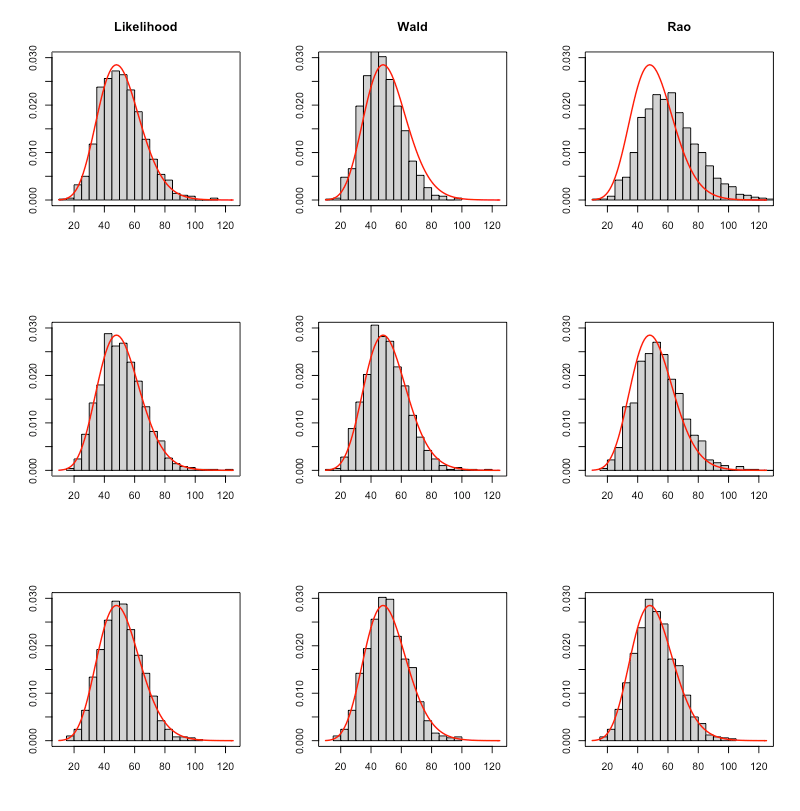}
\caption{Histograms of the three types of test statistics for $\alpha$ under $H_{1,n}^{(1)}$. Each row of figures corresponds to the case of 
$n=10^4$ (upper), $10^5$ (middle) and $10^6$ (bottom) 
and each column of figures corresponds to Likelihood ratio (left), Wald (middle) and Rao (right) type test statistics. The red line is the probability density function of $\chi_2^2(50)$.}
\label{model1_local_hist1}
\end{figure}

\begin{figure}[htbp]
\centering
\includegraphics[height=0.35\vsize]{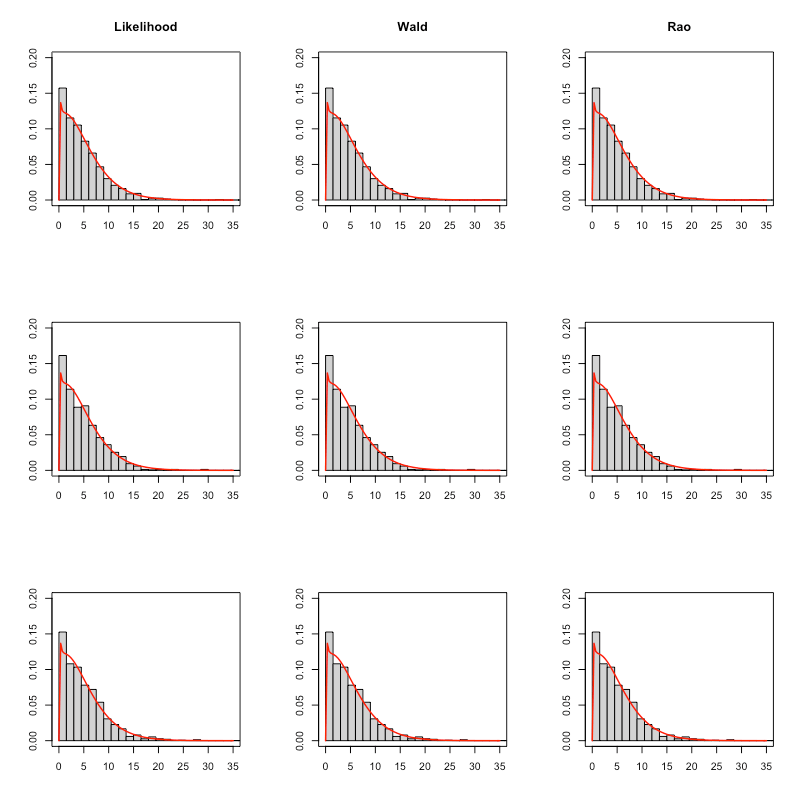}
\caption{Histograms of the three types of test statistics for $\beta$ under $H_{1,n}^{(2)}$. Each row of figures corresponds to the case of 
$n=10^4$ (upper), $10^5$ (middle) and $10^6$ (bottom) 
and each column of figures corresponds to Likelihood ratio (left), Wald (middle) and Rao (right) type test statistics. The red line is the probability density function of $\chi_2^2(4)$.}
\label{model1_local_hist2}
\end{figure}

\begin{figure}[htbp]
\centering
\includegraphics[height=0.36\vsize]{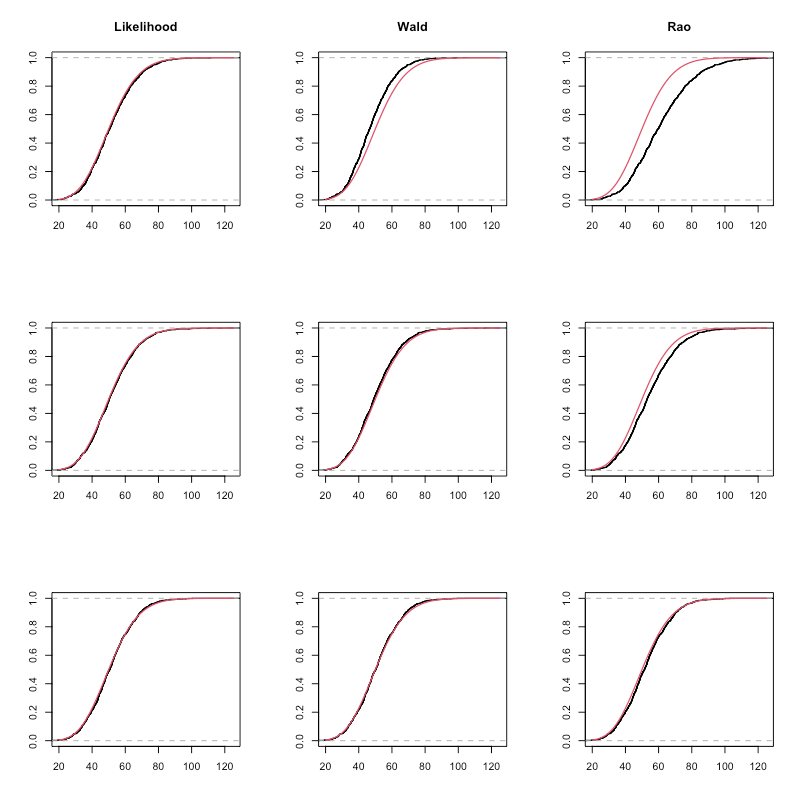}
\caption{Empirical distributions of the three types of test statistics for $\alpha$ under $H_{1,n}^{(1)}$. Each row of figures corresponds to the case of 
$n=10^4$ (upper), $10^5$ (middle) and $10^6$ (bottom) 
and each column of figures corresponds to Likelihood ratio (left), Wald (middle) and Rao (right) type test statistics. The red line is the cumulative distribution function of $\chi_2^2(50)$.}
\label{model1_local_ecdf1}
\end{figure}

\begin{figure}[htbp]
\centering
\includegraphics[height=0.36\vsize]{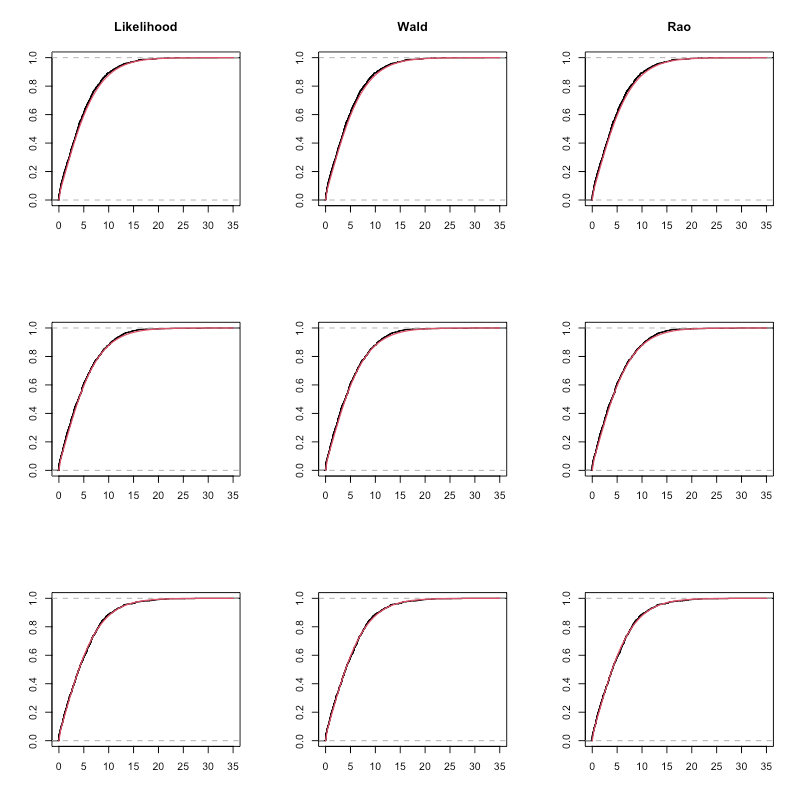}
\caption{Empirical distributions of the three types of test statistics for $\beta$ under $H_{1,n}^{(1)}$. Each row of figures corresponds to the case of 
$n=10^4$ (upper), $10^5$ (middle) and $10^6$ (bottom) 
and each column of figures corresponds to Likelihood ratio (left), Wald (middle) and Rao (right) type test statistics. The red line is the cumulative distribution function of $\chi_2^2(4)$.}
\label{model1_local_ecdf2}
\end{figure}

\newpage
\subsection{Model 2}
\ 

Next, we consider a 1-dimensional diffusion process satisfying the following stochastic differential equation which is more complex than model 1:
\begin{equation}\label{simu_model2}
\begin{cases}
dX_t = -\beta_1(X_t-\beta_2)dt + \left(\alpha_1+\frac{\alpha_2}{1+X_t^2}+\alpha_3\cos^2{X_t}\right)dW_t,\\
X_0 = 1.0.
\end{cases}
\end{equation}
We deal with the following tests:
\begin{equation}\label{simu_test2}
\begin{cases}
H_{0}^{(1)}:\ (\alpha_1,\alpha_2)=(1.0,1.0),\\
H_{1}^{(1)}:\ \text{not}\  H_0^{(1)},
\end{cases}
\quad
\begin{cases}
H_{0}^{(2)}:\ (\beta_1,\beta_2)=(2.0,2.0),\\
H_{1}^{(2)}:\ \text{not}\  H_0^{(2)}.
\end{cases}
\end{equation}
We set the true parameter $\alpha^*_3=0.5$ and choice the true parameter $(\alpha^*_1,\alpha^*_2,\beta^*_1,\beta^*_2)$ from $\{(1.0,1.0, 2.0, 2.0),$ $(1.0, 1.0, 2.5, 2.5),$ $(1.05, 1.05, 2.0, 2.0),$ $(1.05, 1.05, 2.5, 2.5)\}$, which corresponds to the true parameter of Cases 1-4 described in Section \ref{sec:model1}, respectively. Each test is rejected when the realization of each test statistics is greater than $\chi_{2,0.05}^2$, and the other simulation settings are the same 
as in the model (\ref{simu_model1}) of Section 4.1.

Tables \ref{model2_table1-1}-\ref{model2_table1-4} show the number of counts of Cases 1-4 selected by the tests \eqref{simu_test2}, 
where the true parameter $(\alpha^*_1,\alpha^*_2,\beta^*_1,\beta^*_2)$ corresponds to each case. In all of the tables, 
the true case is most often selected as $n$ increases. However, the Rao type test statistic does not have 
good performance in Cases 2 and 4 when $n=10^4$.
Figures \ref{model2_hist1_1}-\ref{model2_ecdf3} are simulation results of the histograms and empirical distributions of the three types of test statistics in Theorem 1. Theoretically, these test statistics converge in distribution to $\chi_2^2$.
Tables \ref{model2_size_table1-1}-\ref{model2_size_table3} show the empirical sizes of the three types of test statistics when null hypothesis is true. From these figures and tables, 
the Likelihood ratio type test statistic has the best behavior of the three types of test statistics.  
\begin{small}
\begin{table}[htbp]
\begin{center}
\caption{Results of test statistics in {\bf{Case 1}}:\ $(\alpha^*_1,\alpha^*_2,\beta^*_1,\beta^*_2)=(1.0,1.0, 2.0, 2.0).$}
\label{model2_table1-1}
\begin{tabular}{ccc|c|cccc}
$n$&$h_n$&$nh_n$&Test type&{\bf{Case 1}}&Case 2&Case 3&Case 4\\
\hline\hline
&&&Likelihood					&896&53&47&4\\
$10^4$&$2.15\times10^{-3}$&21.5&Wald&878&68&48&6\\
&&&Rao							&894&53&49&4\\
\hline
&&&Likelihood					&896&55&47&2\\
$10^5$&$4.64\times10^{-4}$&46.4&Wald&884&68&46&2\\
&&&Rao							&895&56&48&1\\
\hline
&&&Likelihood					&905&41&51&3\\
$10^6$&$1.00\times10^{-4}$&100&Wald&902&44&51&3\\
&&&Rao							&914&32&54&0\\
\end{tabular}
\end{center}
\end{table}

\begin{table}[htbp]
\begin{center}
\caption{Results of test statistics in {\bf{Case 2}}:\ $(\alpha^*_1,\alpha^*_2,\beta^*_1,\beta^*_2)=(1.0, 1.0, 2.5, 2.5).$}
\label{model2_table1-2}
\begin{tabular}{ccc|c|cccc}
$n$&$h_n$&$nh_n$&Test type&Case 1&{\bf{Case 2}}&Case 3&Case 4\\
\hline\hline
&&&Likelihood					&55&888&5&52\\
$10^4$&$2.15\times10^{-3}$&21.5&Wald&37&904&4&55\\
&&&Rao							&325&613&24&38\\
\hline
&&&Likelihood					&0&947&0&53\\
$10^5$&$4.64\times10^{-4}$&46.4&Wald&0&944&0&56\\
&&&Rao							&11&936&0&53\\
\hline
&&&Likelihood					&0&949&0&51\\
$10^6$&$1.00\times10^{-4}$&100&Wald&0&949&0&51\\
&&&Rao							&0&946&0&54\\
\end{tabular}
\end{center}
\end{table}

\begin{table}[htbp]
\begin{center}
\caption{Results of test statistics in {\bf{Case 3}}:\ $(\alpha^*_1,\alpha^*_2,\beta^*_1,\beta^*_2)=(1.05, 1.05, 2.0, 2.0).$}
\label{model2_table1-3}
\begin{tabular}{ccc|c|cccc}
$n$&$h_n$&$nh_n$&Test type&Case 1&Case 2&{\bf{Case 3}}&Case 4\\
\hline\hline
&&&Likelihood					&12&2&926&60\\
$10^4$&$2.15\times10^{-3}$&21.5&Wald&13&2&908&77\\
&&&Rao							&9&4&923&64\\
\hline
&&&Likelihood					&0&0&955&45\\
$10^5$&$4.64\times10^{-4}$&46.4&Wald&0&0&946&54\\
&&&Rao							&0&0&954&46\\
\hline
&&&Likelihood					&0&0&951&49\\
$10^6$&$1.00\times10^{-4}$&100&Wald&0&0&946&54\\
&&&Rao							&0&0&956&44\\
\end{tabular}
\end{center}
\end{table}

\begin{table}[htbp]
\begin{center}
\caption{Results of test statistics in {\bf{Case 4}}:\ $(\alpha^*_1,\alpha^*_2,\beta^*_1,\beta^*_2)=(1.05, 1.05, 2.5, 2.5).$}
\label{model2_table1-4}
\begin{tabular}{ccc|c|cccc}
$n$&$h_n$&$nh_n$&Test type&Case 1&Case 2&Case 3&{\bf{Case 4}}\\
\hline\hline
&&&Likelihood					&2&76&71&851\\
$10^4$&$2.15\times10^{-3}$&21.5&Wald&1&91&44&864\\
&&&Rao							&21&47&412&520\\
\hline
&&&Likelihood					&0&0&0&1000\\
$10^5$&$4.64\times10^{-4}$&46.4&Wald&0&0&0&1000\\
&&&Rao							&0&0&12&988\\
\hline
&&&Likelihood					&0&0&0&1000\\
$10^6$&$1.00\times10^{-4}$&100&Wald&0&0&0&1000\\
&&&Rao							&0&0&0&1000\\
\end{tabular}
\end{center}
\end{table}
\end{small}

\begin{figure}[htbp]
\centering
\includegraphics[height=0.39\vsize]{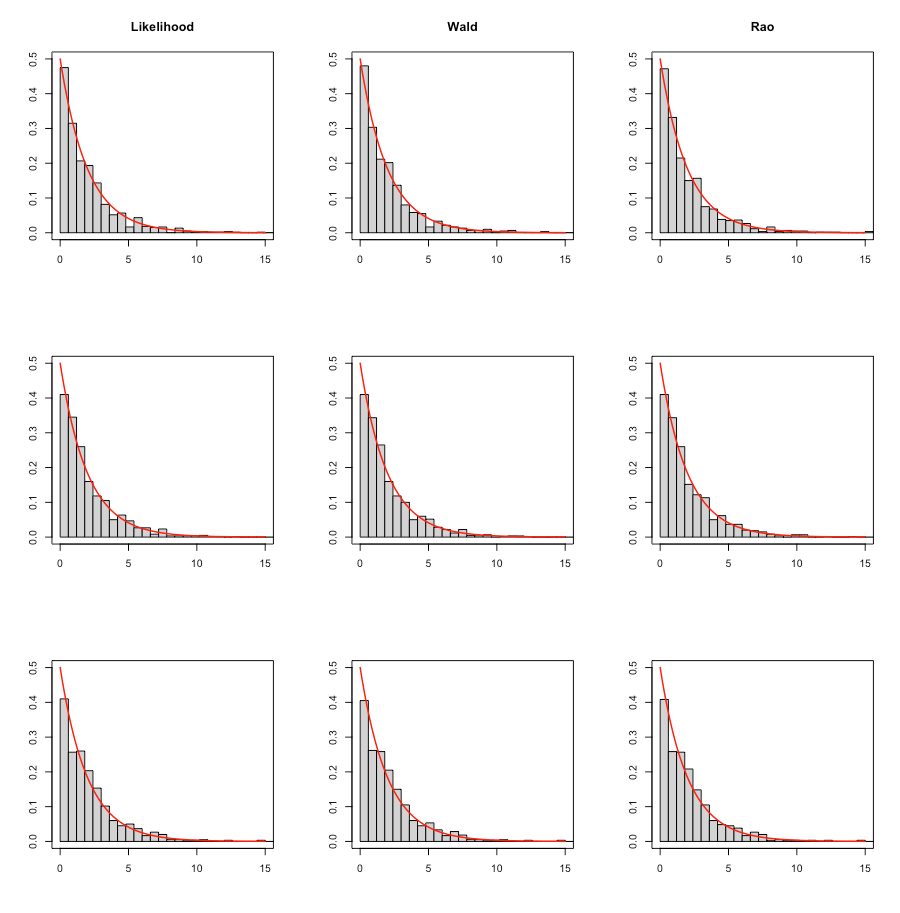}
\caption{Histograms of the three types of test statistics for $\alpha$ in {\bf{Case 1}}.   Each row of figures corresponds to the case of 
$n=10^4$ (upper), $10^5$ (middle) and $10^6$ (bottom) 
and each column of figures corresponds to Likelihood ratio (left), Wald (middle) and Rao (right) type test statistics. The red line is the probability density function of $\chi_2^2$.}
\label{model2_hist1_1}
\end{figure}

\begin{figure}[htbp]
\centering
\includegraphics[height=0.39\vsize]{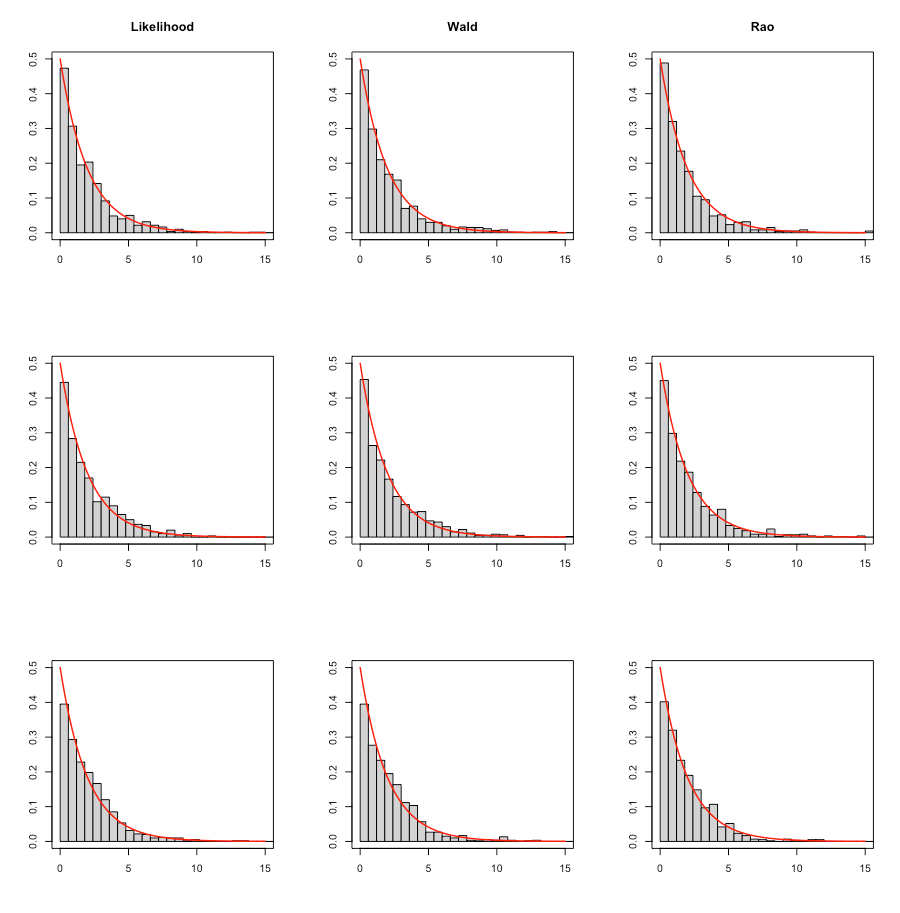}
\caption{Histograms of the three types of test statistics for $\beta$ in {\bf{Case 1}}.  Each row of figures corresponds to the case of 
$n=10^4$ (upper), $10^5$ (middle) and $10^6$ (bottom) 
and each column of figures corresponds to Likelihood ratio (left), Wald (middle) and Rao (right) type test statistics. The red line is the probability density function of $\chi_2^2$.}
\label{model2_hist1_2}
\end{figure}

\begin{figure}[htbp]
\centering
\includegraphics[height=0.39\vsize]{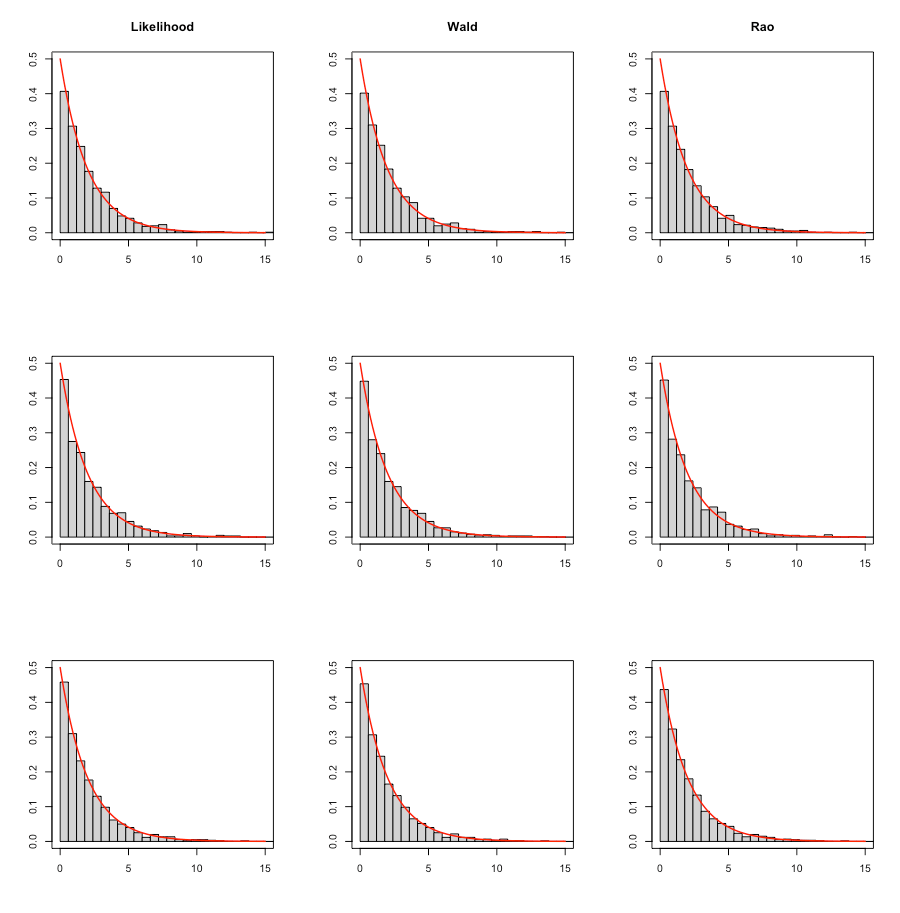}
\caption{Histograms of the three types of test statistics for $\alpha$ in {\bf{Case 2}}. Each row of figures corresponds to the case of 
$n=10^4$ (upper), $10^5$ (middle) and $10^6$ (bottom) 
and each column of figures corresponds to Likelihood ratio (left), Wald (middle) and Rao (right) type test statistics. The red line is the probability density function of $\chi_2^2$.}
\label{model2_hist2}
\end{figure}

\begin{figure}[htbp]
\centering
\includegraphics[height=0.39\vsize]{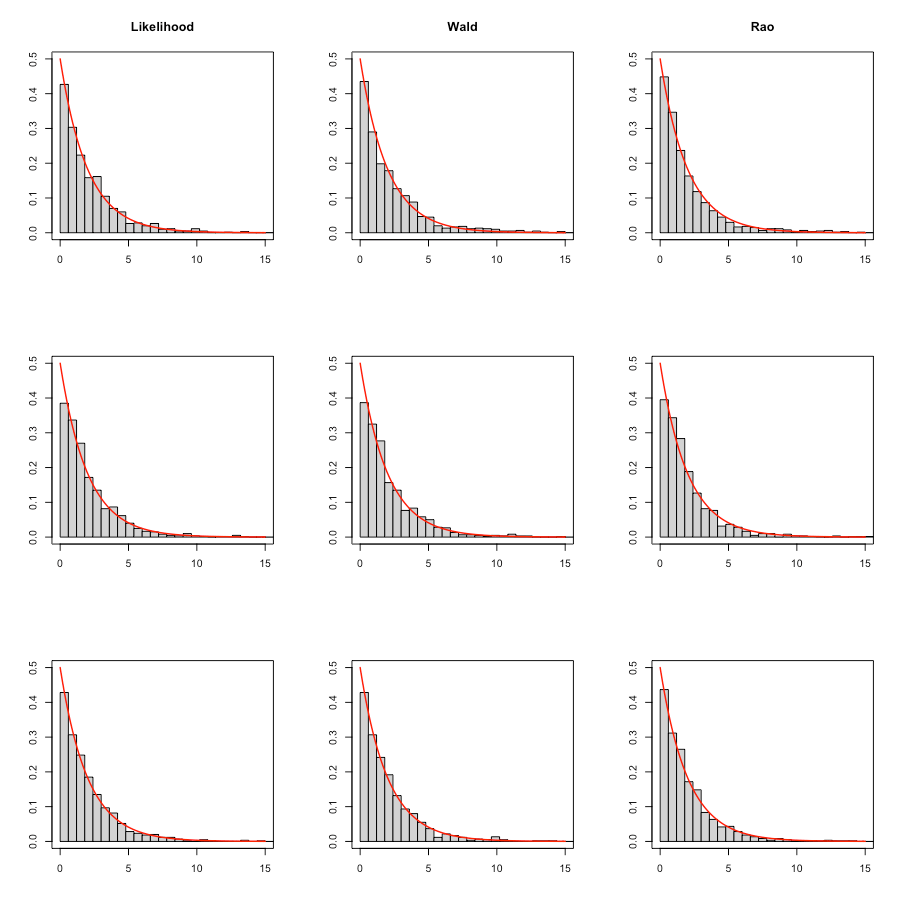}
\caption{Histograms of the three types of test statistics for $\beta$ in {\bf{Case 3}}.  Each row of figures corresponds to the case of 
$n=10^4$ (upper), $10^5$ (middle) and $10^6$ (bottom) 
and each column of figures corresponds to Likelihood ratio (left), Wald (middle) and Rao (right) type test statistics. The red line is the probability density function of $\chi_2^2$.}
\label{model2_hist3}
\end{figure}

\begin{figure}[htbp]
\centering
\includegraphics[height=0.39\vsize]{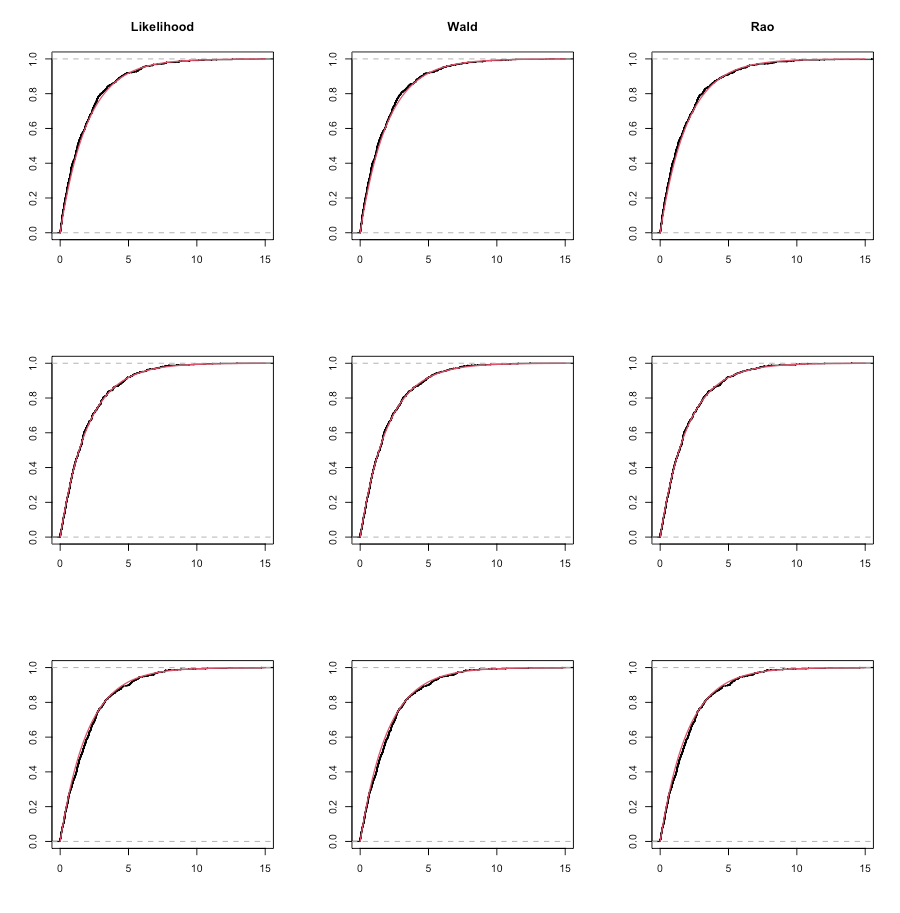}
\caption{Empirical distributions of the three types of test statistics for $\alpha$ in {\bf{Case 1}}.   Each row of figures corresponds to the case of 
$n=10^4$ (upper), $10^5$ (middle) and $10^6$ (bottom) 
and each column of figures corresponds to Likelihood ratio (left), Wald (middle) and Rao (right) type test statistics. The red line is the cumulative distribution function of $\chi_2^2$.}
\label{model2_ecdf1_1}
\end{figure}

\begin{figure}[htbp]
\centering
\includegraphics[height=0.39\vsize]{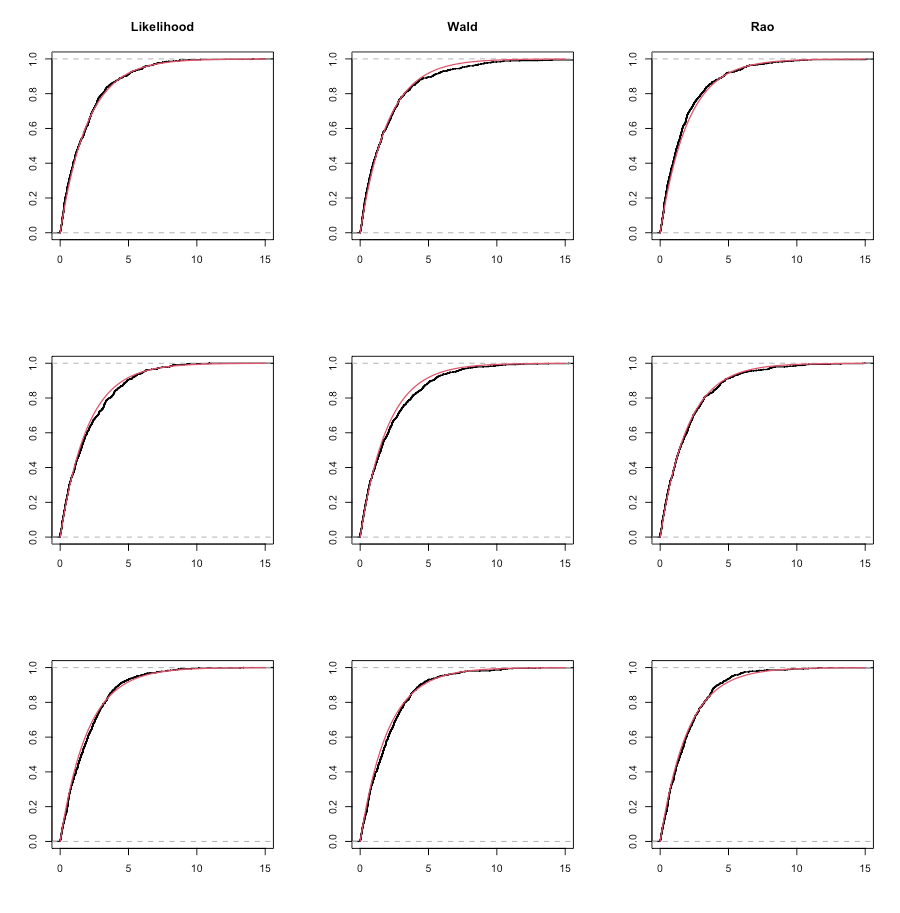}
\caption{Empirical distributions of the three types of test statistics for $\beta$ in {\bf{Case 1}}.   Each row of figures corresponds to the case of 
$n=10^4$ (upper), $10^5$ (middle) and $10^6$ (bottom) 
and each column of figures corresponds to Likelihood ratio (left), Wald (middle) and Rao (right) type test statistics. The red line is the cumulative distribution function of $\chi_2^2$.}
\label{model2_ecdf1_2}
\end{figure}

\begin{figure}[htbp]
\centering
\includegraphics[height=0.39\vsize]{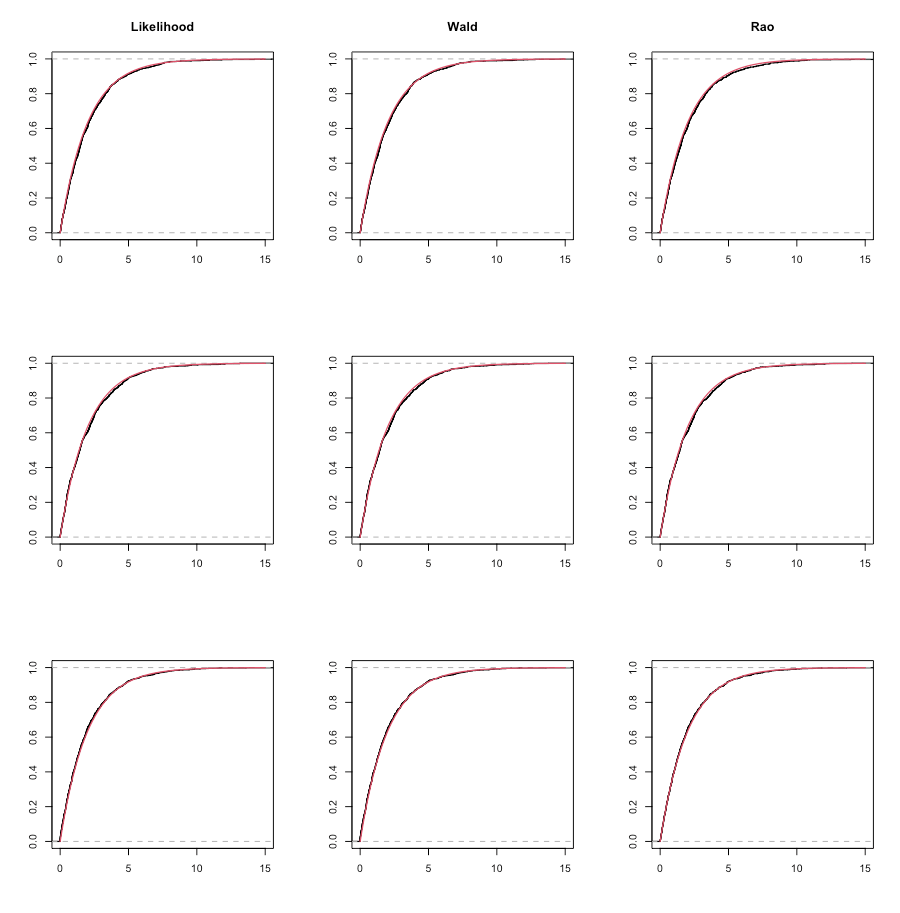}
\caption{Empirical distributions of the three types of test statistics for $\alpha$ in {\bf{Case 2}}.   Each row of figures corresponds to the case of 
$n=10^4$ (upper), $10^5$ (middle) and $10^6$ (bottom) 
and each column of figures corresponds to Likelihood ratio (left), Wald (middle) and Rao (right) type test statistics. The red line is the cumulative distribution function of $\chi_2^2$.}
\label{model2_ecdf2}
\end{figure}

\begin{figure}[htbp]
\centering
\includegraphics[height=0.39\vsize]{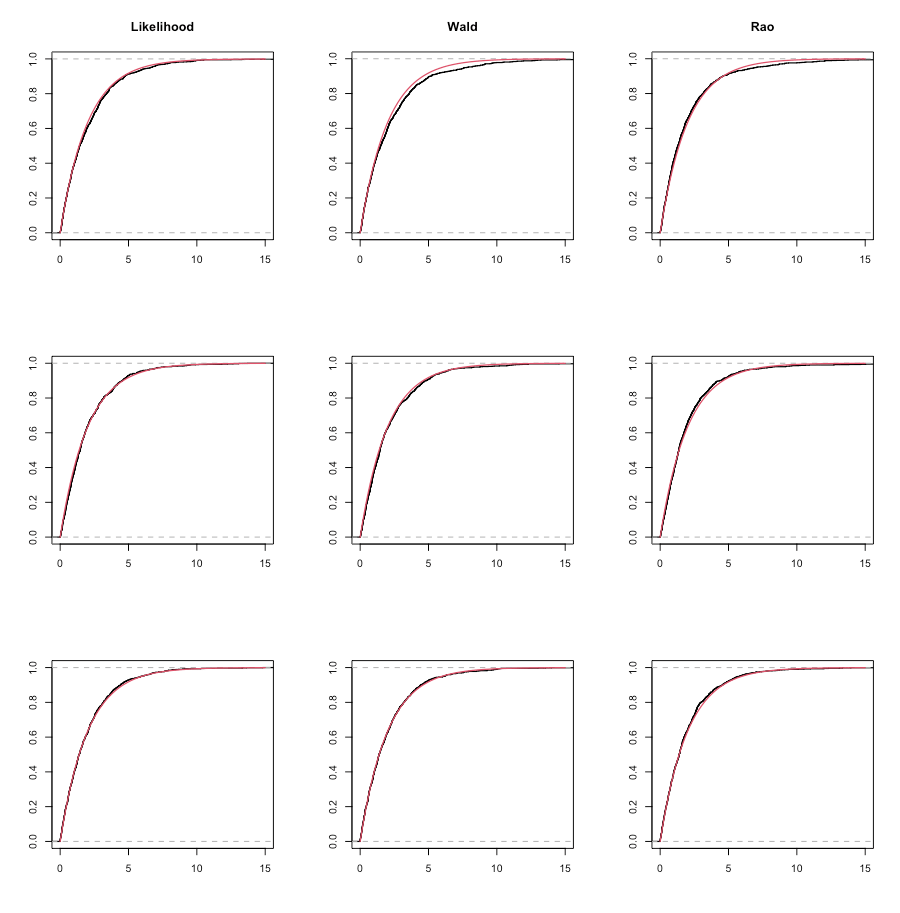}
\caption{Empirical distributions of the three types of test statistics for $\beta$ in {\bf{Case 3}}.   Each row of figures corresponds to the case of 
$n=10^4$ (upper), $10^5$ (middle) and $10^6$ (bottom) 
and each column of figures corresponds to Likelihood ratio (left), Wald (middle) and Rao (right) type test statistics. The red line is the cumulative distribution function of $\chi_2^2$.}
\label{model2_ecdf3}
\end{figure}

\begin{table}[htbp]
\begin{minipage}[htbp]{.49\hsize}
\begin{center}
\caption{Empirical sizes of 
 the test statistics  for $\alpha$ in $\mathbf{Case\ 1}.$}
\label{model2_size_table1-1}
\begin{tabular}{cc|ccc}
$n$&$nh_n$&Likelihood&Wald&Rao\\
\hline\hline
$10^4$&21.5&0.051&0.054&0.053\\
$10^5$&46.4&0.049&0.048&0.049\\
$10^6$&100&0.054&0.054&0.054
\end{tabular}
\end{center} 
\end{minipage}
\hfill
\begin{minipage}[htbp]{.49\hsize}
\begin{center}
\caption{Empirical sizes of 
 the test statistics   for $\beta$ in $\mathbf{Case\ 1}.$}
\label{model2_size_table1-2}
\begin{tabular}{cc|ccc}
$n$&$nh_n$&Likelihood&Wald&Rao\\
\hline\hline
$10^4$&21.5&0.057&0.074&0.057\\
$10^5$&46.4&0.057&0.070&0.057\\
$10^6$&100&0.044&0.047&0.032
\end{tabular}
\end{center}
\end{minipage}
\end{table}

\newpage

\begin{table}[htbp]
\begin{minipage}[htbp]{.49\hsize}
\begin{center}
\caption{Empirical sizes of 
 the test statistics   for $\alpha$ in $\mathbf{Case\ 2}.$}
\label{model2_size_table2}
\begin{tabular}{cc|ccc}
$n$&$nh_n$&Likelihood&Wald&Rao\\
\hline\hline
$10^4$&21.5&0.057&0.059&0.062\\
$10^5$&46.4&0.053&0.056&0.053\\
$10^6$&100&0.051&0.051&0.054
\end{tabular}
\end{center} 
\end{minipage}
\hfill
\begin{minipage}[htbp]{.49\hsize}
\begin{center}
\caption{Empirical sizes of 
the test statistics   for $\beta$ in $\mathbf{Case\ 3}.$}
\label{model2_size_table3}
\begin{tabular}{cc|ccc}
$n$&$nh_n$&Likelihood&Wald&Rao\\
\hline\hline
$10^4$&21.5&0.062&0.079&0.068\\
$10^5$&46.4&0.045&0.054&0.046\\
$10^6$&100&0.049&0.054&0.044
\end{tabular}
\end{center}
\end{minipage}
\end{table}

Next, in order to check consistency of the tests, 
we treat the three kinds of $(\alpha^*_1,\alpha^*_2)$  $((1.05, 1.0), (1.0, 1.05)$ and $(1.05, 1.05))$ in $\mathbf{Case\ 3}$ and the three kinds of $(\beta^*_1,\beta^*_2)$ $((2.5, 2.0), (2.0, 2.5)$ and $(2.5, 2.5))$ in $\mathbf{Case\ 2}$.
Tables \ref{model2_power_table3} and \ref{model2_power_table2} show the empirical powers of the three types of test statistics. 	In table \ref{model2_power_table3}, when $n \leq10^5$, 
since 
the asymptotic behavior of the estimator of $\alpha_1$ is stable compared with that of $\alpha_2$, 
the empirical power in the case of $(\alpha^*_1,\alpha^*_2)=(1.05, 1.0)$ is greater than 
that in the case of $(\alpha^*_1,\alpha^*_2)=(1.0, 1.05)$. 
In Table \ref{model2_power_table2}, 
when $n h_n \leq100$,
the empirical power in the case of $(\beta^*_1,\beta^*_2)=(2.0, 2.5)$ 
is greater than that in the case of $(\beta^*_1,\beta^*_2)=(2.5, 2.0)$ 
because  the asymptotic performance of the estimator of $\beta_2$ is better than that of $\beta_1$.

\begin{table}[htbp]
\begin{minipage}[htbp]{.49\hsize}
\begin{center}
\caption{Empirical powers of the three types of test statistics for $\alpha$ in $\mathbf{Case\ 3}.$}
\label{model2_power_table3}
\begin{tabular}{c|c|ccc}
$(\alpha^*_1,\alpha^*_2)$&$n$&Likelihood&Wald&Rao\\
\hline\hline
&$10^4$&0.898&0.892&0.907\\
(1.05, 1.0)&$10^5$&1.000&1.000&1.000\\
&$10^6$&1.000&1.000&1.000\\
\hline
&$10^4$&0.164&0.149&0.185\\
(1.0, 1.05)&$10^5$&0.918&0.915&0.921\\
&$10^6$&1.000&1.000&1.000\\
\hline
&$10^4$&0.986&0.985&0.987\\
(1.05, 1.05)&$10^5$&1.000&1.000&1.000\\
&$10^6$&1.000&1.000&1.000
\end{tabular}
\end{center} 
\end{minipage}
\hfill
\begin{minipage}[htbp]{.49\hsize}
\begin{center}
\caption{Empirical powers of the three types of test statistics for $\beta$ in $\mathbf{Case\ 2}.$}
\label{model2_power_table2}
\begin{tabular}{c|c|ccc}
$(\beta^*_1,\beta^*_2)$&$nh_n$&Likelihood&Wald&Rao\\
\hline\hline
&$21.5$&0.142&0.200&0.075\\
(2.5, 2.0)&$46.4$&0.207&0.261&0.141\\
&$100$&0.401&0.427&0.342\\
\hline
&$21.5$&0.861&0.848&0.780\\
(2.0, 2.5)&$46.4$&0.997&0.994&0.984\\
&$100$&1.000&1.000&1.000\\
\hline
&$21.5$&0.940&0.959&0.651\\
(2.5, 2.5)&$46.4$&1.000&1.000&0.989\\
&$100$&1.000&1.000&1.000\\
\end{tabular}
\end{center}
\end{minipage}
\end{table}

\section{Proofs}\label{sec5}
Let
\begin{align*}
I_a(\alpha_0;\alpha_0)
=\begin{pmatrix}
I_{a,1}(\alpha_0;\alpha_0)&I_{a,2}(\alpha_0;\alpha_0)\\
I_{a,2}^\top(\alpha_0;\alpha_0)&I_{a,3}(\alpha_0;\alpha_0)
\end{pmatrix},
\quad
H_1=\begin{pmatrix}
0&0\\
0&I_{a,3}^{-1}(\alpha_0;\alpha_0)
\end{pmatrix},
\end{align*}
where
$I_{a,1}(\alpha_0;\alpha_0)$,
$I_{a,2}(\alpha_0;\alpha_0)$,
$I_{a,3}(\alpha_0;\alpha_0)$
are the $(r_1\times r_1)$, $(r_1\times (p_1-r_1))$ 
and $((p_1-r_1)\times (p_1-r_1))$
matrices, respectively.
Set
\begin{align*}
I_b(\theta_0;\theta_0)
=\begin{pmatrix}
I_{b,1}(\theta_0;\theta_0)&I_{b,2}(\theta_0;\theta_0)\\
I_{b,2}^\top(\theta_0;\theta_0)&I_{b,3}(\theta_0;\theta_0)
\end{pmatrix}, 
\quad 
H_2=\begin{pmatrix}
0&0\\
0&I_{b,3}^{-1}(\theta_0;\theta_0)
\end{pmatrix},
\end{align*}
where
$I_{b,1}(\theta_0;\theta_0)$,
$I_{b,2}(\theta_0;\theta_0)$,
$I_{b,3}(\theta_0;\theta_0)$
are the $(r_2\times r_2)$, $(r_2\times (p_2-r_2))$ and $((p_2-r_2)\times (p_2-r_2))$
matrices, respectively.
Moreover, $Y_1$ and $Y_2$ are the normal random variables
with mean $0$ and covariance matrices
$I_a(\alpha_0;\alpha_0)$ and $I_b(\theta_0;\theta_0)$,
respectively, 
that is, 
$Y_1 \sim N(0,I_a(\alpha_0;\alpha_0))$ and  $Y_2\sim N(0,I_b(\theta_0;\theta_0))$.

In order to proof Theorem $\ref{thm:1}$, we need the following lemma.
 \begin{lemma}\label{lem:1}
Assume $\mathbf{A1\mathchar`-A7}$. If $h_n\to0,\ nh_n\to\infty$ and $nh_n^2\to0$, then
\begin{enumerate}
\item [$(\mathrm{i})$] $\sqrt{n}(\hat{\alpha}_n-\tilde{\alpha}_n)\overset{d}{\to}(I_{a}^{-1}(\alpha_0;\alpha_0)-H_1)Y_1\quad(\text{under}\ H_0^{(1)}),$
\item [$(\mathrm{ii})$]  $\sqrt{nh_n}(\hat{\beta}_n-\tilde{\beta}_n)\overset{d}{\to}(I_{b}^{-1}(\theta_0;\theta_0)-H_2)Y_2\quad(\text{under}\ H_0^{(2)})$. 
\end{enumerate}
\end{lemma}
\begin{proof}[Proof]\ $(\mathrm{i})$\ 
First of all, 
we will show that

\begin{align}\label{goal:1}
H_1\frac{1}{\sqrt{n}}\partial_{\alpha}U_n^{(1)}(\alpha_0)=\sqrt{n}(\tilde{\alpha}_n-\alpha_0)+o_P(1)\quad(\text{under}\ H_0^{(1)}).
\end{align}
From Taylor's theorem,
\begin{align}\label{equ:lm1:a}
\frac{1}{\sqrt{n}}\partial_{\alpha}U_n^{(1)}(\tilde{\alpha}_n)&-\frac{1}{\sqrt{n}}\partial_{\alpha}U_n^{(1)}(\alpha_0)=\tilde{I}_{a,n}(\tilde{\alpha}_n,\alpha_0)\sqrt{n}(\tilde{\alpha}_n-\alpha_0),
\end{align}
where 
\begin{align*}
\tilde{I}_{a,n}(\tilde{\alpha}_n,\alpha_0):=\int_{0}^{1}{\frac{1}{n}\partial_\alpha^2U_n^{(1)}(\alpha_0+u(\tilde{\alpha}_n-\alpha_0))}du.
\end{align*}
Noting that $H_1\partial_\alpha U_n^{(1)}(\tilde{\alpha}_n)=0$, we obtain
\begin{equation}\label{equ:lm3:a}
H_1\frac{1}{\sqrt{n}}\partial_{\alpha}U_n^{(1)}(\alpha_0)=-H_1\tilde{I}_{a,n}(\tilde{\alpha}_n,\alpha_0)\sqrt{n}(\tilde{\alpha}_n-\alpha_0).
\end{equation}
On the other hand, we can show
\begin{align}
-\tilde{I}_{a,n}(\tilde{\alpha}_n,\alpha_0)\overset{P}{\to}I_a(\alpha_0;\alpha_0)\quad(\text{under}\ H_0^{(1)}),\label{equ:uchida-1} \\
\frac{1}{\sqrt{n}}\partial_\alpha U_n^{(1)}(\alpha_0)\overset{d}{\to}
Y_1\quad(\text{under}\ H_0^{(1)}),  \label{equ:uchida-6}
\\
H_1\frac{1}{\sqrt{n}}\partial_\alpha U_n^{(1)}(\alpha_0)\overset{d}{\to}H_1Y_1\quad(\text{under}\ H_0^{(1)}). \label{equ:uchida-3}
\end{align}
We set $\tilde{\alpha}_n^{(p_1-r_1)}=(\tilde{\alpha}_{n,r_1+1}-\alpha_{0,r_1+1},\ldots,\tilde{\alpha}_{n,p_1}-\alpha_{0,p_1})$.
Noting that the first $r_1$ components 
of $\sqrt{n}(\tilde{\alpha}_n-\alpha_0)$
are all $0$,  one has that under $H_0^{(1)}$, 
\begin{align}
H_1 \tilde{I}_{a,n}(\alpha_0,\alpha_0)
 \sqrt{n}(\tilde{\alpha}_n-\alpha_0)
 =& 
 \begin{pmatrix}
0 & 0 \\
I_{a,3}^{-1} \tilde{I}_{a, n, 2}^T(\alpha_0, \alpha_0) 
& 
I_{a,3}^{-1} \tilde{I}_{a, n, 3}(\alpha_0, \alpha_0) 
\end{pmatrix}
\begin{pmatrix}
0\\
\sqrt{n}\tilde{\alpha}_n^{(p_1-r_1)}
\end{pmatrix}
\nonumber 
\\
=&
\begin{pmatrix}
0\\
I_{a,3}^{-1} \tilde{I}_{a, n, 3}(\alpha_0, \alpha_0)  
\sqrt{n}\tilde{\alpha}_n^{(p_1-r_1)}
\end{pmatrix}, 
\label{equ:uchida-4}
\\
H_1 { I_{a}(\alpha_0;\alpha_0)}
 \sqrt{n}(\tilde{\alpha}_n-\alpha_0)
 =& \begin{pmatrix}
0 & 0 \\
I_{a,3}^{-1} I_{a,2}^T(\alpha_0; \alpha_0)
& 
E_{p_1-r_1}
\end{pmatrix}
\begin{pmatrix}
0\\
\sqrt{n}\tilde{\alpha}_n^{(p_1-r_1)}
\end{pmatrix}
=\sqrt{n}(\tilde{\alpha}_n-\alpha_0),
\label{equ:uchida-5}
\end{align}
where $E_{p}$ is the $(p\times p)$ identity matrix. 
The above equations 
\eqref{equ:lm3:a}, \eqref{equ:uchida-1}, \eqref{equ:uchida-3}  and
\eqref{equ:uchida-4}
implies that under $H_0^{(1)}$, 
\begin{align*}
\sqrt{n}\tilde{\alpha}_n^{(p_1-r_1)} =O_P(1).
\end{align*}
Since the first $r_1$ components 
of $\sqrt{n}(\tilde{\alpha}_n-\alpha_0)$
are all $0$, 
it follows that
\begin{align}
\sqrt{n}(\tilde{\alpha}_n-\alpha_0)=O_P(1)\quad(\text{under}\ H_0^{(1)}). 
\label{equ:uchida-2} 
\end{align}
Hence, 
by \eqref{equ:lm3:a}, \eqref{equ:uchida-1},   \eqref{equ:uchida-5} and
\eqref{equ:uchida-2},  
we have that under $H_0^{(1)}$, 
\begin{align*}
H_1\frac{1}{\sqrt{n}}\partial_{\alpha}U_n^{(1)}(\alpha_0)&=H_1\left({I}_{a}(\alpha_0;\alpha_0)+o_P(1)\right)\sqrt{n}(\tilde{\alpha}_n-\alpha_0)\\
&=H_1I_a(\alpha_0;\alpha_0)\sqrt{n}(\tilde{\alpha}_n-\alpha_0)+o_P(1)\\
&=\sqrt{n}(\tilde{\alpha}_n-\alpha_0)+o_P(1).
\end{align*}
This accords with \eqref{goal:1}. 
Next, 
it follows from \eqref{equ:lm1:a},  \eqref{equ:uchida-1},
\eqref{equ:uchida-2}, \eqref{goal:1} and  \eqref{equ:uchida-6},
that
\begin{align}
\frac{1}{\sqrt{n}}\partial_{\alpha}U_n^{(1)}(\tilde{\alpha}_n)&=\frac{1}{\sqrt{n}}\partial_{\alpha}U_n^{(1)}(\alpha_0)+\tilde{I}_{a,n}(\tilde{\alpha}_n,\alpha_0)\sqrt{n}(\tilde{\alpha}_n-\alpha_0)
\nonumber \\
&=\frac{1}{\sqrt{n}}\partial_{\alpha}U_n^{(1)}(\alpha_0)-{I}_{a}(\alpha_0;\alpha_0)\sqrt{n}(\tilde{\alpha}_n-\alpha_0)+o_P(1)
\nonumber \\
&=\frac{1}{\sqrt{n}}\partial_{\beta}U_n^{(1)}(\alpha_0)-{I}_{a}(\alpha_0;\alpha_0)H_1\frac{1}{\sqrt{n}}\partial_{\alpha}U_n^{(1)}(\alpha_0)+o_P(1)
\nonumber \\
&=(E_{p_1}-I_a(\alpha_0;\alpha_0)H_1)\frac{1}{\sqrt{n}}\partial_{\alpha}U_n^{(1)}(\alpha_0)+o_P(1)
\nonumber \\
&\overset{d}{\to}(E_{p_1}-I_a(\alpha_0;\alpha_0)H_1)Y_1\quad(\text{under}\ H_0^{(1)}).
\label{equ:uchida-7}
\end{align}
By noting that $\partial_{\alpha}U_n^{(1)}(\hat{\alpha}_n)=0$, 
it follows from Taylor's theorem that
\begin{align*}
-\frac{1}{\sqrt{n}}\partial_{\alpha}U_n^{(1)}(\tilde{\alpha}_n)&=\tilde{I}_{a,n}(\hat{\alpha}_n,\tilde{\alpha}_n)\sqrt{n}(\hat{\alpha}_n-\tilde{\alpha}_n)\\
&=-I_a(\alpha_0;\alpha_0)\sqrt{n}(\hat{\alpha}_n-\tilde{\alpha}_n)+o_P(1).
\end{align*}
Since $I_a(\alpha_0;\alpha_0)$ is non-singular, by \eqref{equ:uchida-7} and Slutsky's theorem,
\begin{align*}
\sqrt{n}(\hat{\alpha}_n-\tilde{\alpha}_n)&\overset{d}{\to}I_a^{-1}(\alpha_0;\alpha_0)(E_{p_1}-I_a(\alpha_0;\alpha_0)H_1)Y_1\\
&=(I_{a}^{-1}(\alpha_0;\alpha_0)-H_1)Y_1\quad(\text{under}\ H_0^{(1)}).
\end{align*}
This completes the proof of $(\mathrm{i})$.

$(\mathrm{ii})$ First, we will show that 
\begin{align}\label{goal:2}
H_2\frac{1}{\sqrt{nh_n}}\partial_{\beta}U_n^{(2)}(\beta_0|\alpha_0)=\sqrt{nh_n}(\tilde{\beta}_n-\beta_0)+o_P(1)\quad(\text{under}\ H_0^{(2)}). 
\end{align}
From Taylor's theorem with respect to $\beta$, 
\begin{align}\label{equ:lm1}
\frac{1}{\sqrt{nh_n}}\partial_{\beta}U_n^{(2)}(\tilde{\beta}_n|\hat{\alpha}_n)&=\frac{1}{\sqrt{nh_n}}\partial_{\beta}U_n^{(2)}(\beta_0|\hat{\alpha}_n)+\tilde{I}_{b,n}(\tilde{\beta}_n,\beta_0)\sqrt{nh_n}(\tilde{\beta}_n-\beta_0),
\end{align}
where
\begin{align*}
\tilde{I}_{b,n}(\tilde{\beta}_n,\beta_0):=\int_{0}^{1}{\frac{1}{nh_n}\partial_\beta^2U_n^{(2)}(\beta_0+u(\tilde{\beta}_n-\beta_0)|\hat{\alpha}_n)}du.
\end{align*}
Moreover, by Taylor's theorem with respect to $\alpha$,
\begin{align}\label{equ:lm2}
\frac{1}{\sqrt{nh_n}}\partial_{\beta}U_n^{(2)}(\beta_0|\hat{\alpha}_n)&=\frac{1}{\sqrt{nh_n}}\partial_{\beta}U_n^{(2)}(\beta_0|\alpha_0)+\tilde{I}_{ab,n}(\hat{\alpha}_n,\alpha_0)\sqrt{n}(\hat{\alpha}_n-\alpha_0),
\end{align}
where
\begin{align*}
\tilde{I}_{ab,n}(\hat{\alpha}_n,\alpha_0):=\int_{0}^{1}{\frac{1}{n\sqrt{h_n}}\partial_\alpha\partial_\beta U_n^{(2)}(\beta_0|\alpha_0+u(\hat{\alpha}_n-\alpha_0))}du.
\end{align*}
We can show $\tilde{I}_{ab,n}(\hat{\alpha}_n,\alpha_0)\overset{P}{\to}0\ (\text{under}\ H_0^{(2)})$.
It follows from \eqref{equ:lm1} and \eqref{equ:lm2} that
\begin{align}\label{equ:lm4}
\frac{1}{\sqrt{nh_n}}\partial_{\beta}U_n^{(2)}(\tilde{\beta}_n|\hat{\alpha}_n)=\frac{1}{\sqrt{nh_n}}\partial_{\beta}U_n^{(2)}(\beta_0|\alpha_0)+\tilde{I}_{b,n}(\tilde{\beta}_n,\beta_0)\sqrt{nh_n}(\tilde{\beta}_n-\beta_0)+o_P(1).
\end{align}
Since $H_2\partial_\beta U_n^{(2)}(\tilde{\beta}_n|\hat{\alpha}_n)=0$, we obtain
\begin{equation}\label{equ:lm3}
H_2\frac{1}{\sqrt{nh_n}}\partial_{\beta}U_n^{(2)}(\beta_0|\alpha_0)=-H_2\tilde{I}_{b,n}(\tilde{\beta}_n,\beta_0)\sqrt{nh_n}(\tilde{\beta}_n-\beta_0)+o_P(1).
\end{equation}
On the other hand, we have
\begin{align}
-\tilde{I}_{b,n}(\tilde{\beta}_n,\beta_0)\overset{P}{\to}I_b(\theta_0;\theta_0)\quad(\text{under}\ H_0^{(2)}),
\label{beta:uchida-1} \\
\frac{1}{\sqrt{nh_n}}\partial_\beta U_n^{(2)}(\beta_0|\alpha_0)\overset{d}{\to}
Y_2\quad(\text{under}\ H_0^{(2)}), 
\label{beta:uchida-2} \\
H_2\frac{1}{\sqrt{nh_n}}\partial_\beta U_n^{(2)}(\beta_0|\alpha_0)\overset{d}{\to}H_2Y_2\quad(\text{under}\ H_0^{(2)}).
\label{beta:uchida-3} 
\end{align}
If we set $\tilde{\beta}_n^{(p_2-r_2)}=(\tilde{\beta}_{n,r_2+1}-\beta_{0,r_2+1},\ldots,\tilde{\beta}_{n,p_2}-\beta_{0,p_2})$, then
\begin{align}
H_2 \tilde{I}_{b,n}(\beta_0,\beta_0) \sqrt{nh_n}(\tilde{\beta}_n-\beta_0)
=& 
\begin{pmatrix}
0 & 0 \\
I_{b,3}^{-1} \tilde{I}_{b, n, 2}^T(\beta_0, \beta_0) 
& 
I_{b,3}^{-1} \tilde{I}_{b, n, 3}(\beta_0, \beta_0) 
\end{pmatrix}
\begin{pmatrix}
0\\
\sqrt{nh_n}\tilde{\beta}_n^{(p_2-r_2)}
\end{pmatrix}
\nonumber 
\\
=&
\begin{pmatrix}
0\\
I_{b,3}^{-1} \tilde{I}_{b, n, 3}(\beta_0, \beta_0)  
\sqrt{nh_n}\tilde{\beta}_n^{(p_2-r_2)}
\end{pmatrix}, 
\label{beta:uchida-4}
\\
H_2 I_{b}(\theta_0;\theta_0)\sqrt{nh_n}(\tilde{\beta}_n-\beta_0)
=&
\begin{pmatrix}
0&0\\
I_{b,3}^{-1} I_{b, 2}^T(\theta_0; \theta_0)  &E_{p_2-r_2}
\end{pmatrix}
\begin{pmatrix}
0\\
\sqrt{nh_n}\tilde{\beta}_n^{(p_2-r_2)}
\end{pmatrix}=\sqrt{nh_n}(\tilde{\beta}_n-\beta_0).
\label{beta:uchida-5}
\end{align}
In an analogous manner to the proof of
\eqref{equ:uchida-2},
one has that 
\begin{align}
\sqrt{nh_n}(\tilde{\beta}_n-\beta_0)=O_P(1)\quad(\text{under}\ H_0^{(2)}).
\label{beta:uchida-6}
\end{align}
Hence,  
\eqref{equ:lm3}, \eqref{beta:uchida-1}, 
\eqref{beta:uchida-6} 
and \eqref{beta:uchida-5} 
implies that
\begin{align}
H_2\frac{1}{\sqrt{nh_n}}\partial_{\beta}U_n^{(2)}(\beta_0|\alpha_0)&=H_2\left({I}_{b}(\theta_0;\theta_0)+o_P(1)\right)\sqrt{nh_n}(\tilde{\beta}_n-\beta_0)+o_P(1)
\nonumber \\
&=H_2I_b(\theta_0;\theta_0)\sqrt{nh_n}(\tilde{\beta}_n-\beta_0)+o_P(1)
\nonumber \\
&=\sqrt{nh_n}(\tilde{\beta}_n-\beta_0)+o_P(1). 
\label{beta:uchida-7} 
\end{align}
This accords with \eqref{goal:2}. 
It follows from \eqref{equ:lm4}, \eqref{beta:uchida-1}, 
\eqref{beta:uchida-6}, \eqref{beta:uchida-7} 
and \eqref{beta:uchida-2} that
\begin{align}
\frac{1}{\sqrt{nh_n}}\partial_{\beta}U_n^{(2)}(\tilde{\beta}_n|\hat{\alpha}_n)&=\frac{1}{\sqrt{nh_n}}\partial_{\beta}U_n^{(2)}(\beta_0|\alpha_0)+\tilde{I}_{b,n}(\tilde{\beta}_n,\beta_0)\sqrt{nh_n}(\tilde{\beta}_n-\beta_0)+o_P(1)
\nonumber \\
&=\frac{1}{\sqrt{nh_n}}\partial_{\beta}U_n^{(2)}(\beta_0|\alpha_0)-{I}_{b}(\theta_0;\theta_0)\sqrt{nh_n}(\tilde{\beta}_n-\beta_0)+o_P(1)
\nonumber \\
&=\frac{1}{\sqrt{nh_n}}\partial_{\beta}U_n^{(2)}(\beta_0|\alpha_0)-{I}_{b}(\theta_0;\theta_0)H_2\frac{1}{\sqrt{nh_n}}\partial_{\beta}U_n^{(2)}(\beta_0|\alpha_0)+o_P(1)
\nonumber \\
&=(E_{p_2}-I_b(\theta_0;\theta_0)H_2)\frac{1}{\sqrt{nh_n}}\partial_{\beta}U_n^{(2)}(\beta_0|\alpha_0)+o_P(1)
\nonumber \\
&\overset{d}{\to}(E_{p_2}-I_b(\theta_0;\theta_0)H_2)Y_2\quad(\text{under}\ H_0^{(2)}).
\label{beta:uchida-8} 
\end{align}
Since we have that under $H_0^{(2)}$,
\begin{eqnarray*}
& & \sqrt{nh_n}(\hat{\beta}_n-\tilde{\beta}_n) =O_P(1), \\
& & \tilde{I}_{b,n}(\hat{\beta}_n,\tilde{\beta}_n) = - I_b(\theta_0;\theta_0) +o_P(1),
\end{eqnarray*}
it follows from Taylor's theorem that under $H_0^{(2)}$,  
\begin{align*}
-\frac{1}{\sqrt{nh_n}}\partial_{\beta}U_n^{(2)}(\tilde{\beta}_n|\hat{\alpha}_n)&
=\tilde{I}_{b,n}(\hat{\beta}_n,\tilde{\beta}_n)\sqrt{nh_n}(\hat{\beta}_n-\tilde{\beta}_n)\\
&=-I_b(\theta_0;\theta_0)\sqrt{nh_n}(\hat{\beta}_n-\tilde{\beta}_n)+o_P(1).
\end{align*}
Since $I_b(\theta_0;\theta_0)$ is non-singular, 
by \eqref{beta:uchida-8} and Slutsky's theorem,
\begin{align*}
\sqrt{nh_n}(\hat{\beta}_n-\tilde{\beta}_n)&\overset{d}{\to}I_b^{-1}(\theta_0;\theta_0)(E_{p_2}-I_b(\theta_0,\theta_0)H_2)Y_2\\
&=(I_{b}^{-1}(\theta_0;\theta_0)-H_2)Y_2\quad(\text{under}\ H_0^{(2)}).
\end{align*}
This completes the proof of $(\mathrm{ii})$.
\end{proof}

\begin{proof}[Proof of Theorem 1]\ $(\mathrm{i})$\ By Taylor's theorem, we obtain
\begin{align*}
U_n^{(1)}(\tilde{\alpha}_n)-U_n^{(1)}(\hat{\alpha}_n)=\left(\sqrt{n}(\hat{\alpha}_n-\tilde{\alpha}_n)\right)^\top J_{a,n}(\tilde{\alpha}_n,\hat{\alpha}_n)\sqrt{n}(\hat{\alpha}_n-\tilde{\alpha}_n),
\end{align*}
where
\begin{align*}
 J_{a,n}(\tilde{\alpha}_n,\hat{\alpha}_n)=\int_{0}^{1}{(1-u)\frac{1}{n}\partial_\alpha^2 U_n^{(1)}(\hat{\alpha}_n+u(\tilde{\alpha}_n-\hat{\alpha}_n))}du.
\end{align*}
It follows from  Lemma $\ref{lem:1}$ that
\begin{align*}
\sqrt{n}(\hat{\alpha}_n-\tilde{\alpha}_n)\overset{d}{\to}(I_a^{-1}(\alpha_0;\alpha_0)-H_1)Y_1\quad(\text{under}\ H_0^{(1)}).
\end{align*}
Moreover, we can check that
under $H_0^{(1)}$, 
\begin{align*}
\hat{\alpha}_n\overset{P}{\to}\alpha_0,&\quad\tilde{\alpha}_n\overset{P}{\to}\alpha_0,\\
J_{a,n}(\tilde{\alpha}_n,\hat{\alpha}_n)&\overset{P}{\to}-\frac{1}{2}I_a(\alpha_0;\alpha_0).
\end{align*}
We set $Y_1=I_a^{\frac{1}{2}}(\alpha_0;\alpha_0)Z_1$,
$Z_1\sim N(0,E_{p_1})$
and
\begin{align*}
P_1=I_a^{\frac{1}{2}}(\alpha_0;\alpha_0)(I_a^{-1}(\alpha_0;\alpha_0)-H_1)I_a^{\frac{1}{2}}(\alpha_0;\alpha_0).
\end{align*}
It follows from the continuous mapping theorem that
\begin{align*}
\Lambda_n^{(1)}&=-2(U_n^{(1)}(\tilde{\alpha}_n)-U_n^{(1)}(\hat{\alpha}_n))\\
&\overset{d}{\to}Y_1^\top (I_a^{-1}(\alpha_0;\alpha_0)-H_1)I_a(\alpha_0;\alpha_0)(I_a^{-1}(\alpha_0;\alpha_0)-H_1)Y_1\\
&=Z_1^\top I_a^{\frac{1}{2}}(\alpha_0;\alpha_0)(I_a^{-1}(\alpha_0;\alpha_0)-H_1)I_a^{\frac{1}{2}}(\alpha_0;\alpha_0)Z_1
= Z_1^\top  P_1 Z_1.
\end{align*}
Since $P_1$ is the projection matrix with 
$$
\mathrm{rank}\ P_1=\mathrm{trace}\ P_1=r_1,
$$ 
we obtain
\begin{align*}
\Lambda_n^{(1)}\overset{d}{\to}\chi_{r_1}^2\quad(\text{under}\ H_0^{(1)}),
\end{align*}
which implies the first statement of $(\mathrm{i})$.

Next, 
by Taylor's theorem with respect to $\beta$, 
one has that under $H_0^{(2)}$,
\begin{align*}
U_n^{(2)}(\tilde{\beta}_n|\hat{\alpha}_n)-U_n^{(2)}(\hat{\beta}_n|\hat{\alpha}_n)=\left(\sqrt{nh_n}(\hat{\beta}_n-\tilde{\beta}_n)\right)^\top J_{b,n}(\tilde{\beta}_n,\hat{\beta}_n)\sqrt{nh_n}(\hat{\beta}_n-\tilde{\beta}_n),
\end{align*}
where
\begin{align*}
 J_{b,n}(\tilde{\beta}_n,\hat{\beta}_n):=\int_{0}^{1}{(1-u)\frac{1}{nh_n}\partial_\beta^2 U_n^{(2)}(\hat{\beta}_n+u(\tilde{\beta}_n-\hat{\beta}_n)|\hat{\alpha}_n)}du.
\end{align*}
It follows from Lemma $\ref{lem:1}$ that
\begin{align*}
\sqrt{nh_n}(\hat{\beta}_n-\tilde{\beta}_n)\overset{d}{\to}(I_b^{-1}(\theta_0;\theta_0)-H_2)Y_2\quad(\text{under}\ H_0^{(2)}),
\end{align*}
and we can check
\begin{align*}
\hat{\beta}_n\overset{P}{\to}\beta_0,&\quad\tilde{\beta}_n\overset{P}{\to}\beta_0,\\
J_{b,n}(\tilde{\beta}_n,\hat{\beta}_n)&\overset{P}{\to}-\frac{1}{2}I_b(\theta_0;\theta_0).
\end{align*}
Let $Y_2=I_b^{\frac{1}{2}}(\theta_0;\theta_0)Z_2$, 
$Z_2\sim N(0,E_{p_2})$
and
\begin{align*}
P_2=I_b^{\frac{1}{2}}(\theta_0;\theta_0)(I_b^{-1}(\theta_0;\theta_0)-H_2)
I_b^{\frac{1}{2}}(\theta_0;\theta_0).
\end{align*}
Note that $P_2$ is the projection matrix with 
$$
\mathrm{rank}\ P_2=\mathrm{trace}\ P_2=r_2.
$$
The continuous mapping theorem yields that
\begin{align*}
\Lambda_n^{(2)}&=-2(U_n^{(2)}(\tilde{\beta}_n|\hat{\alpha}_n)-U_n^{(2)}(\hat{\beta}_n|\hat{\alpha}_n))\\
&\overset{d}{\to}Y_2^\top (I_b^{-1}(\theta_0;\theta_0)-H_2)I_b(\theta_0;\theta_0)(I_b^{-1}(\theta_0;\theta_0)-H_2)Y_2\\
&=Z_2^\top I_b^{\frac{1}{2}}(\theta_0;\theta_0)(I_b^{-1}(\theta_0;\theta_0)-H_2)I_b^{\frac{1}{2}}(\theta_0;\theta_0)Z_2 = Z_2^\top   P_2 Z_2 \\
&\sim \chi_{r_2}^2\quad(\text{under}\ H_0^{(2)}).
\end{align*}
This completes the proof of the Likelihood type.

$(\mathrm{ii})$ By noting that 
$$I_{a,n}(\hat{\alpha}_n)\overset{P}{\to}I_a(\alpha_0;\alpha_0)\quad (\text{under}\ H_0^{(1)}),$$
it follows from Lemma $\ref{lem:1}$ and the continuous mapping theorem that
\begin{align*}
W_n^{(1)}&=\left(\sqrt{n}(\hat{\alpha}_n-\tilde{\alpha}_n)\right)^\top I_{a,n}(\hat{\alpha}_n)\sqrt{n}(\hat{\alpha}_n-\tilde{\alpha}_n)\\
&\overset{d}{\to}Y_1^\top(I_{a}^{-1}(\alpha_0;\alpha_0)-H_1) I_a(\alpha_0;\alpha_0)(I_{a}^{-1}(\alpha_0;\alpha_0)-H_1)Y_1\\
&\sim \chi_{r_1}^2\quad(\text{under}\ H_0^{(1)}).
\end{align*}
Moreover, we check
$$I_{b,n}(\hat{\beta}_n|\hat{\alpha}_n)\overset{P}{\to}I_b(\theta_0;\theta_0)\quad (\text{under}\ H_0^{(2)}),$$
and it holds
\begin{align*}
W_n^{(2)}&=\left(\sqrt{nh_n}(\hat{\beta}_n-\tilde{\beta}_n)\right)^\top I_{b,n}(\hat{\beta}_n|\hat{\alpha}_n)\sqrt{nh_n}(\hat{\beta}_n-\tilde{\beta}_n)\\
&\overset{d}{\to}Y_2^\top(I_{b}^{-1}(\theta_0;\theta_0)-H_2)I_b(\theta_0;\theta_0)(I_{b}^{-1}(\theta_0;\theta_0)-H_2)Y_2\\
&\sim \chi_{r_2}^2\quad(\text{under}\ H_0^{(2)}).
\end{align*}
This completes the proof of the Wald type.

$(\mathrm{iii})$ From the proof of Lemma $\ref{lem:1}$, 
it holds
$$
\frac{1}{\sqrt{n}}\partial_\alpha U_n^{(1)}(\tilde{\alpha}_n)\overset{d}{\to}(E_{p_1}-I_a(\alpha_0;\alpha_0)H_1)Y_1\quad(\text{under}\ H_0^{(1)}).
$$
It follows that
$$
 \bar{I}_{a,n}(\hat{\alpha}_n)\overset{P}{\to}I_{a}^{-1}(\alpha_0;\alpha_0)\quad(\text{under}\ H_0^{(1)}).
$$
Hence we obtain from the continuous mapping theorem that
\begin{align*}
R_n^{(1)}&=\left(\frac{1}{\sqrt{n}}\partial_{\alpha}U_n^{(1)}(\tilde{\alpha}_n)\right)^\top \bar{I}_{a,n}(\hat{\alpha}_n)\frac{1}{\sqrt{n}}\partial_{\alpha}U_n^{(1)}(\tilde{\alpha}_n)\\
&\overset{d}{\to}Y_1^\top(E_{p_1}-I_a(\alpha_0;\alpha_0)H_1)^\top I_a^{-1}(\alpha_0;\alpha_0)(E_{p_1}-I_a(\alpha_0;\alpha_0)H_1)Y_1\\
&=Y_1^\top(E_{p_1}-I_a(\alpha_0;\alpha_0)H_1)^\top I_a^{-1}(\alpha_0;\alpha_0)I_a(\alpha_0;\alpha_0)I_a^{-1}(\alpha_0;\alpha_0)(E_{p_1}-I_a(\alpha_0;\alpha_0)H_1)Y_1\\
&=Y_1^\top(I_{a}^{-1}(\alpha_0;\alpha_0)-H_1)I_a(\alpha_0;\alpha_0)(I_{a}^{-1}(\alpha_0;\alpha_0)-H_1)Y_1\\
&\sim \chi_{r_1}^2 \quad(\text{under}\ H_0^{(1)}).
\end{align*}
Furthermore, it follows from the proof of Lemma $\ref{lem:1}$ that
$$
\frac{1}{\sqrt{nh_n}}\partial_{\beta}U_n^{(2)}(\tilde{\beta}_n|\hat{\alpha}_n)\overset{d}{\to}(E_{p_2}-I_b(\theta_0,\theta_0)H_2)Y_2\quad(\text{under}\ H_0^{(2)}),
$$
and we can check
$$
 \bar{I}_{b,n}(\hat{\beta}_n|\hat{\alpha}_n)\overset{P}{\to}I_{b}^{-1}(\theta_0;\theta_0)\quad(\text{under}\ H_0^{(2)}).
$$
Thus, one has
\begin{align*}
R_n^{(2)}&=\left(\frac{1}{\sqrt{nh_n}}\partial_{\beta}U_n^{(2)}(\tilde{\beta}_n|\hat{\alpha}_n)\right)^\top \bar{I}_{b,n}(\hat{\beta}_n|\hat{\alpha}_n)\frac{1}{\sqrt{nh_n}}\partial_{\beta}U_n^{(2)}(\tilde{\beta}_n|\hat{\alpha}_n)\\
&\overset{d}{\to}Y_2^\top(E_{p_2}-I_b(\theta_0;\theta_0)H_2)^\top I_b^{-1}(\theta_0;\theta_0)(E_{p_2}-I_b(\theta_0;\theta_0)H_2)Y_2\\
&=Y_2^\top(E_{p_2}-I_b(\theta_0;\theta_0)H_2)^\top I_b^{-1}(\theta_0;\theta_0)I_b(\theta_0;\theta_0)I_b^{-1}(\theta_0;\theta_0)(E_{p_2}-I_b(\theta_0;\theta_0)H_2)Y_2\\
&=Y_2^\top(I_{b}^{-1}(\theta_0;\theta_0)-H_2)I_b(\theta_0;\theta_0)(I_{b}^{-1}(\theta_0;\theta_0)-H_2)Y_2\\
&\sim \chi_{r_2}^2\quad(\text{under}\ H_0^{(2)}).
\end{align*}
This completes the proof of the Rao type.
\end{proof}

Next,  
we will prove the following lemma
to show Theorem $\ref{thm:2}$. 

\begin{lemma}\label{lem:2}
Assume $\mathbf{A1\mathchar`-A7}$ and $\mathbf{B1}$.\ If $h_n\to0$ and $nh_n\to\infty$, then
\begin{enumerate}
\item[$(\mathrm{i})$] $\tilde{\alpha}_n\overset{P}{\to}\alpha^*\quad(\text{under}\ H_1^{(1)}),$
\item[$(\mathrm{ii})$] $\tilde{\beta}_n\overset{P}{\to}\beta^*\quad(\text{under}\ H_1^{(2)}).$
\end{enumerate}
\end{lemma}
\begin{proof}[Proof]\ $(\mathrm{i})$ From the definition of $\bar{U}_1(\alpha;\alpha_1^*)$, it holds
$$
\sup_{\alpha\in\Theta_\alpha}\left|\frac{1}{n}U_n^{(1)}(\alpha)-\bar{U}_1(\alpha;\alpha_1^*)\right|\overset{P}{\to}0\quad(\text{under}\ H_1^{(1)}).
$$
$\mathbf{B1}$-(a) implies the following : for all $\varepsilon>0$, there exists $\delta>0$ such that for all $\alpha\in\tilde{\Theta}_\alpha$,
$$|\alpha-\alpha^*|\geq \varepsilon\  \Longrightarrow\  \bar{U}_1(\alpha^*;\alpha_1^*)-\bar{U}_1(\alpha;\alpha_1^*)>\delta.$$
Therefore we obtain
\begin{align*}
0&\leq P\left(|\tilde{\alpha}_n-\alpha^*|\geq \varepsilon\right)\\
&\leq P\left(\bar{U}_1(\alpha^*;\alpha_1^*)-\bar{U}_1(\tilde{\alpha}_n;\alpha_1^*)>\delta\right)\\
&=P\left(\bar{U}_1(\alpha^*;\alpha_1^*)-\frac{1}{n}U_n^{(1)}(\alpha^*)+\frac{1}{n}U_n^{(1)}(\alpha^*)-\frac{1}{n}U_n^{(1)}(\tilde{\alpha}_n)+\frac{1}{n}U_n^{(1)}(\tilde{\alpha}_n)-\bar{U}_1(\tilde{\alpha}_n;\alpha_1^*)>\delta\right)\\
&\leq 2P\left(\sup_{\alpha\in\Theta_\alpha}\left|\frac{1}{n}U_n^{(1)}(\alpha)-\bar{U}_1(\alpha;\alpha_1^*)\right|>\frac{\delta}{3}\right)+P\left(\frac{1}{n}U_n^{(1)}(\alpha^*)-\frac{1}{n}U_n^{(1)}(\tilde{\alpha}_n)>\frac{\delta}{3}\right)\\
&\to 0\ (n\to\infty),
\end{align*}
which implies the statement of $(\mathrm{i})$.

$(\mathrm{ii})$ For sake of simplicity, $\alpha_1^*$ denotes the true value of $\alpha$ under $H_1^{(2)}$.
From the definition of $\bar{U}_2(\alpha,\beta;\beta_1^*)$, it holds
$$
\sup_{\theta\in\Theta}\left|L_n(\alpha,\beta)-\bar{U}_2(\alpha,\beta;\beta_1^*)\right|\overset{P}{\to}0\quad(\text{under}\ H_1^{(2)}),
$$
where
$$
L_n(\alpha,\beta)=\frac{1}{nh_n}U_n^{(2)}(\beta|\alpha)-\frac{1}{nh_n}U_n^{(2)}(\beta_1^*|\alpha). 
$$
$\mathbf{B1}$-(b) implies the following : for all $\varepsilon>0$, there exists $\delta>0$ such that for all $\beta\in\tilde{\Theta}_\beta$,
$$|\beta-\beta^*|\geq \varepsilon\  \Longrightarrow\  \bar{U}_2(\alpha_1^*,\beta^*;\beta_1^*)-\bar{U}_2(\alpha_1^*,\beta;\beta_1^*)>\delta.$$
Hence it holds
\begin{align*}
0&\leq P\left(|\tilde{\beta}_n-\beta^*|\geq \varepsilon\right)\\
&\leq P\left(\bar{U}_2(\alpha_1^*,\beta^*;\beta_1^*)-\bar{U}_2(\alpha_1^*,\tilde{\beta}_n;\beta_1^*)>\delta\right)\\
&=P\left(\bar{U}_2(\alpha_1^*,\beta^*;\beta_1^*)-L_n(\hat{\alpha}_n,\beta^*)+L_n(\hat{\alpha}_n,\beta^*)-L_n(\hat{\alpha}_n,\tilde{\beta}_n)+L_n(\hat{\alpha}_n,\tilde{\beta}_n)-\bar{U}_2(\alpha_1^*,\tilde{\beta}_n;\beta_1^*)>\delta\right)\\
&\leq 2P\left(\sup_{\beta\in\Theta_\beta}\left|L_n(\hat{\alpha}_n,\beta)-\bar{U}_2(\alpha_1^*,\beta;\beta_1^*)\right|>\frac{\delta}{3}\right)+P\left(L_n(\hat{\alpha}_n,\beta^*)-L_n(\hat{\alpha}_n,\tilde{\beta}_n)>\frac{\delta}{3}\right)\\
&=2P\left(\sup_{\beta\in\Theta_\beta}\left|L_n(\hat{\alpha}_n,\beta)-\bar{U}_2(\alpha_1^*,\beta;\beta_1^*)\right|>\frac{\delta}{3}\right),
\end{align*}
where
\begin{align}
P\left(\sup_{\beta\in\Theta_\beta}\left|L_n(\hat{\alpha}_n,\beta)-\bar{U}_2(\alpha_1^*,\beta;\beta_1^*)\right|>\frac{\delta}{3}\right)&\leq
P\left(\sup_{\beta\in\Theta_\beta}\left|L_n(\hat{\alpha}_n,\beta)-L_n(\alpha_1^*,\beta)\right|>\frac{\delta}{6}\right)
\nonumber \\
&\quad+P\left(\sup_{\beta\in\Theta_\beta}\left|L_n(\alpha_1^*,\beta)-\bar{U}_2(\alpha_1^*,\beta;\beta_1^*)\right|>\frac{\delta}{6}\right)
\nonumber \\
&\leq P\left(\sup_{\beta\in\Theta_\beta}\left|L_n(\hat{\alpha}_n,\beta)-L_n(\alpha_1^*,\beta)\right|>\frac{\delta}{6}\right)
\nonumber \\
&\quad+P\left(\sup_{\theta\in\Theta}\left|L_n(\alpha,\beta)-\bar{U}_2(\alpha,\beta;\beta_1^*)\right|>\frac{\delta}{6}\right).
\label{uchida-proof-1}
\end{align}
The second term on the right hand side 
of the inequality \eqref{uchida-proof-1}
converges to 0 under $H_0^{(2)}$. 
It follows from Taylor's theorem that
\begin{align*}
\left|L_n(\hat{\alpha}_n,\beta)-L_n(\alpha_1^*,\beta)\right|&=\left|\int_{0}^{1}\partial_\alpha L_n(\alpha_1^*+u(\hat{\alpha}_n-\alpha_1^*),\beta)du(\hat{\alpha}_n-\alpha_1^*)\right|\\
&\leq \int_{0}^{1}{\left|\partial_\alpha L_n(\alpha_1^*+u(\hat{\alpha}_n-\alpha_1^*),\beta)\right|}du\left|\hat{\alpha}_n-\alpha_1^*\right|\\
&\leq \sup_{\theta\in\Theta}\left|\partial_\alpha L_n(\alpha,\beta)\right|\left|\hat{\alpha}_n-\alpha_1^*\right|,
\end{align*}
and it is easy to check
$$
\left|\hat{\alpha}_n-\alpha_1^*\right|=o_P(1),\ \quad \sup_{\theta\in\Theta}\left|\partial_\alpha L_n(\alpha,\beta)\right|=O_P(1).
$$
Therefore, 
the first term on the right hand side 
of the inequality \eqref{uchida-proof-1}
converges 
to 0, which implies
$
\tilde{\beta}_n\overset{P}{\to}\beta^*.
$
\end{proof}
\begin{proof}[Proof of Theorem 2]\ $(\mathrm{i})$ 
Since  Lemma \ref{lem:2} implies
that under $H_1^{(1)}$, 
$$
\hat{\alpha}_n\overset{P}{\to}\alpha_1^*,\quad \tilde{\alpha}_n\overset{P}{\to}\alpha^*\neq\alpha_1^*,
$$
$$
\sup_{\alpha\in\Theta_\alpha}\left|\frac{1}{n}U_n^{(1)}(\alpha)-\bar{U}_1(\alpha;\alpha_1^*)\right|\overset{P}{\to}0,
$$
we obtain
\begin{align*}
\frac{1}{n}\Lambda_n^{(1)}&=\frac{2}{n}(U_n^{(1)}(\hat{\alpha}_n)-U_n^{(1)}(\tilde{\alpha}_n))\\
&\overset{P}{\to}2(\bar{U}_1(\alpha_1^*;\alpha_1^*)-\bar{U}_1(\alpha^*;\alpha_1^*))>0.
\end{align*}
Therefore, for all $\varepsilon\in (0,1)$,
$$
0\leq P\left(\Lambda_n^{(1)}\leq\chi_{r_1,\varepsilon}^2\right)=P\left(\frac{1}{n}\Lambda_n^{(1)}\leq\frac{\chi_{r_1,\varepsilon}^2}{n}\right)\to0\quad(\text{under}\ H_1^{(1)},\ n\to\infty),
$$
which implies the first statement of $(\mathrm{i})$.

Next, we consider $\Lambda_n^{(2)}$ under $H_1^{(2)}$. It holds
\begin{align*}
\frac{1}{nh_n}\Lambda_n^{(2)}&=\frac{2}{nh_n}(U_n^{(2)}(\hat{\beta}_n|\hat{\alpha}_n)-U_n^{(2)}(\tilde{\beta}_n|\hat{\alpha}_n))\\
&=\frac{2}{nh_n}(U_n^{(2)}(\hat{\beta}_n|\hat{\alpha}_n)-U_n^{(2)}(\beta_1^*|\hat{\alpha}_n))+\frac{2}{nh_n}(U_n^{(2)}(\beta_1^*|\hat{\alpha}_n)-U_n^{(2)}(\tilde{\beta}_n|\hat{\alpha}_n))\\
&\geq \frac{2}{nh_n}(U_n^{(2)}(\beta_1^*|\hat{\alpha}_n)-U_n^{(2)}(\tilde{\beta}_n|\hat{\alpha}_n)).
\end{align*}
Lemma \ref{lem:2} implies that under $H_1^{(2)}$,
$$
\hat{\alpha}_n\overset{P}{\to}\alpha_1^*,\quad \tilde{\beta}_n\overset{P}{\to}\beta^*\neq\beta_1^*,
$$
$$
\sup_{\theta\in\Theta}\left|L_n(\alpha,\beta)-\bar{U}_2(\alpha,\beta;\beta_1^*)\right|\overset{P}{\to}0.
$$
Noting that $\bar{U}_2(\alpha_1^*,\beta_1^*;\beta_1^*)=0$ 
and $L_n(\alpha,\beta)=\frac{1}{nh_n}(U_n^{(2)}(\beta|\alpha)-U_n^{(2)}(\beta_1^*|\alpha))
$, 
one has
$$
\frac{2}{nh_n}(U_n^{(2)}(\beta_1^*|\hat{\alpha}_n)-U_n^{(2)}(\tilde{\beta}_n|\hat{\alpha}_n))
\overset{P}{\to}-2\bar{U}_2(\alpha_1^*,\beta^*;\beta_1^*)>0. 
$$
Therefore, for all $\varepsilon\in (0,1)$,
\begin{align*}
0\leq P\left(\Lambda_n^{(2)}\leq\chi_{r_2,\varepsilon}^2\right)&=P\left(\frac{1}{nh_n}\Lambda_n^{(2)}\leq\frac{\chi_{r_2,\varepsilon}^2}{nh_n}\right)\\
&\leq P\left(\frac{2}{nh_n}(U_n^{(2)}(\beta_1^*|\hat{\alpha}_n)-U_n^{(2)}(\tilde{\beta}_n|\hat{\alpha}_n))\leq\frac{\chi_{r_2,\varepsilon}^2}{nh_n}\right)\\
&\to0\quad(\text{under}\ H_1^{(2)},\ n\to\infty).
\end{align*}
This completes the proof of the Likelihood type.

$(\mathrm{ii})$ 
We have that under $H_1^{(1)}$,
\begin{align*}
&\hat{\alpha}_n-\tilde{\alpha}_n\overset{P}{\to}\alpha_1^*-\alpha^*\neq0,\\
&I_{a,n}(\hat{\alpha}_n)\overset{P}\to{}I_a(\alpha_1^*;\alpha_1^*).
\end{align*}
Hence 
\begin{align*}
\frac{1}{n}W_n^{(1)}&=(\hat{\alpha}_n-\tilde{\alpha}_n)^\top I_{a,n}(\hat{\alpha}_n)(\hat{\alpha}_n-\tilde{\alpha}_n)\\
&\overset{P}{\to}(\alpha_1^*-\alpha^*)^\top I_a(\alpha_1^*;\alpha_1^*)(\alpha_1^*-\alpha^*)>0
\end{align*}
because of $\mathbf{A7}$. Therefore, for all $\varepsilon\in (0,1)$,
$$
P(W_n^{(1)}\geq\chi^2_{{r_1,\varepsilon}})=P\left(\frac{1}{n}W_n^{(1)}\geq\frac{\chi^2_{{r_1,\varepsilon}}}{n}\right)\to 1\quad(\text{under}\ H_1^{(1)}).
$$
Moreover,  Lemma \ref{lem:2} implies that under $H_1^{(2)}$,
\begin{align*}
&\hat{\beta}_n-\tilde{\beta}_n\overset{P}{\to}\beta_1^*-\beta^*\neq0\\
&I_{b,n}(\hat{\beta}_n|\hat{\alpha}_n)\overset{P}\to{}I_b(\theta_1^*;\theta_1^*).
\end{align*}
Hence
\begin{align*}
\frac{1}{nh_n}W_n^{(2)}&=(\hat{\beta}_n-\tilde{\beta}_n)^\top I_{b,n}(\hat{\beta}_n|\hat{\alpha}_n)(\hat{\beta}_n-\tilde{\beta}_n)\\
&\overset{P}{\to}(\beta_1^*-\beta^*)^\top I_b(\theta_1^*;\theta_1^*)(\beta_1^*-\beta^*)>0.
\end{align*}
Therefore, for all $\varepsilon\in (0,1)$,
$$
P(W_n^{(2)}\geq\chi^2_{{r_2,\varepsilon}})=P\left(\frac{1}{nh_n}W_n^{(2)}\geq\frac{\chi^2_{{r_2,\varepsilon}}}{nh_n}\right)\to 1\quad(\text{under}\ H_1^{(2)}).
$$
This completes the proof of the Wald type.

$(\mathrm{iii})$ From Taylor's theorem, it holds
$$
\frac{1}{n}\partial_\alpha U_n^{(1)}(\tilde{\alpha}_n)=-\int_{0}^{1}{\frac{1}{n}\partial_\alpha^2U_n^{(1)}(\hat{\alpha}_n+u(\tilde{\alpha}_n-\hat{\alpha}_n))}du(\hat{\alpha}_n-\tilde{\alpha}_n).
$$
It is shown that under $H_1^{(1)}$,
\begin{align*}
\hat{\alpha}_n-\tilde{\alpha}_n&\overset{P}{\to}\alpha_1^*-\alpha^*\neq0,\\
-\int_{0}^{1}{\frac{1}{n}\partial_\alpha^2U_n^{(1)}(\hat{\alpha}_n+u(\tilde{\alpha}_n-\hat{\alpha}_n))}du&\overset{P}{\to}
\int_{0}^{1}{I_a(\alpha_1^*+u(\alpha^*-\alpha_1^*);\alpha_1^*)}du.
\end{align*}
Hence it follows from $\mathbf{B2}$-(a) that
$$
\frac{1}{n}\partial_\alpha U_n^{(1)}(\tilde{\alpha}_n)\overset{P}{\to}\int_{0}^{1}{I_a(\alpha_1^*+u(\alpha^*-\alpha_1^*);\alpha_1^*)}du(\alpha_1^*-\alpha^*)=:c_1\neq0,
$$
and
\begin{align*}
\frac{1}{n}R_n^{(1)}&=\left(\frac{1}{n}\partial_{\alpha}U_n^{(1)}(\tilde{\alpha}_n)\right)^\top \bar{I}_{a,n}(\hat{\alpha}_n)\frac{1}{n}\partial_{\alpha}U_n^{(1)}(\tilde{\alpha}_n)\\
&\overset{P}{\to}c_1^\top I_a^{-1}(\alpha_1^*,\alpha_1^*)c_1>0.
\end{align*}
Therefore, for all $\varepsilon\in (0,1)$,
$$
P(R_n^{(1)}\geq\chi^2_{{r_1,\varepsilon}})=P\left(\frac{1}{n}R_n^{(1)}\geq\frac{\chi^2_{{r_1,\varepsilon}}}{n}\right)\to 1\quad(\text{under}\ H_1^{(1)}).
$$
Furthermore, 
it follows from Taylor's theorem with respect to $\beta$ that
under $H_1^{(2)}$, 
$$
\frac{1}{nh_n}\partial_\beta U_n^{(2)}(\tilde{\beta}_n|\hat{\alpha}_n)=-\int_{0}^{1}{\frac{1}{nh_n}\partial_\beta^2U_n^{(2)}(\hat{\beta}_n+u(\tilde{\beta}_n-\hat{\beta}_n)|\hat{\alpha}_n)}du(\hat{\beta}_n-\tilde{\beta}_n),
$$
and we have
\begin{align*}
\hat{\beta}_n-\tilde{\beta}_n&\overset{P}{\to}\beta_1^*-\beta^*\neq0,\\
-\int_{0}^{1}{\frac{1}{nh_n}\partial_\beta^2U_n^{(2)}(\hat{\beta}_n+u(\tilde{\beta}_n-\hat{\beta}_n)|\hat{\alpha}_n))}du&\overset{P}{\to}
\int_{0}^{1}{I_b((\alpha_1^*,\beta_1^*+u(\beta^*-\beta_1^*));\theta_1^*)}du.
\end{align*}
Hence it follows from $\mathbf{B2}$-(b) that
$$
\frac{1}{nh_n}\partial_\beta U_n^{(2)}(\tilde{\beta}_n|\hat{\alpha}_n)\overset{P}{\to}\int_{0}^{1}{I_b((\alpha_1^*,\beta_1^*+u(\beta^*-\beta_1^*));\theta_1^*)}du(\beta_1^*-\beta^*)=:c_2\neq0,
$$
and
\begin{align*}
\frac{1}{nh_n}R_n^{(2)}&=\left(\frac{1}{nh_n}\partial_{\beta}U_n^{(2)}(\tilde{\beta}_n|\hat{\alpha}_n)\right)^\top \bar{I}_{b,n}(\hat{\beta}_n|\hat{\alpha}_n)\frac{1}{nh_n}\partial_{\beta}U_n^{(2)}(\tilde{\beta}_n|\hat{\alpha}_n)\\
&\overset{P}{\to}c_2^\top I_b^{-1}(\theta_1^*,\theta_1^*)c_2>0.
\end{align*}
Therefore, for all $\varepsilon\in (0,1)$,
$$
P(R_n^{(2)}\geq\chi^2_{{r_2,\varepsilon}})=P\left(\frac{1}{nh_n}R_n^{(2)}\geq\frac{\chi^2_{{r_2,\varepsilon}}}{nh_n}\right)\to 1\quad(\text{under}\ H_1^{(2)}).
$$
This completes the proof of the Rao type.
\end{proof}

Finally, we will prove Theorem \ref{thm:3}. 
In order to proof this theorem, we use the following lemma.
\begin{lemma}\label{lem:3}
Assume $\mathbf{A1\mathchar`-A7}$ and $\mathbf{C1}$. If $h_n\to0,\ nh_n\to\infty$ and $nh_n^2\to0$, then
\begin{enumerate}
\item[$(\mathrm{i})$] $\sqrt{n}(\hat{\alpha}_n-\alpha_0)\overset{d}{\to}Y_1'\sim N(u_\alpha,I_a^{-1}(\alpha_0;\alpha_0))\quad(\text{under}\ H_{1,n}^{(1)}),$
\item[$(\mathrm{ii})$] $\sqrt{nh_n}(\hat{\beta}_n-\beta_0)\overset{d}{\to}Y_2'\sim N(u_\beta,I_b^{-1}(\theta_0;\theta_0))\quad(\text{under}\ H_{1,n}^{(2)}).$
\end{enumerate}
\end{lemma}
\begin{proof}[Proof]\ $(\mathrm{i})$ 
First of all, 
it is shown that
\begin{align}\label{ggoal:1}
\frac{1}{\sqrt{n}}\partial_\alpha U_n^{(1)}(\alpha_0)\overset{d}{\to}L_1\sim N(I_a(\alpha_0;\alpha_0)u_\alpha,I_a(\alpha_0;\alpha_0))\quad(\text{under}\ H_{1,n}^{(1)}).
\end{align}
For sake of simplicity, we use the following symbols:
\begin{align*}
&X_{{t_i^n}}=(X_{i,1},\ldots,X_{i,d})^\top,\ S(X_{t_{i-1}^n},\alpha)=S_{i-1}(\alpha)=(S_{i-1,pq}(\alpha))_{1\leq p,q\leq d},\ u_\alpha=(u_{\alpha,1},\ldots,u_{\alpha,p_1})^\top,\\
&R_{i-1}(h_n):=R(\theta,h_n,X_{t_{i-1}^n}),\ \mathcal{G}_{i}^n:=\sigma(W_s,s\leq t_{i}^n).
\end{align*}
It follows that
\begin{align*}
U_n^{(1)}(\alpha)&=-\frac{1}{2}\sum_{i=1}^{n}\left\{h_n^{-1}S^{-1}(X_{t_{i-1}^n},\alpha)\left[(X_{t_{i}^n}-X_{t_{i-1}^n})^{\otimes 2}\right]+\log\det S(X_{t_{i-1}^n},\alpha)\right\}\\
&=-\frac{1}{2h_n}\sum_{i=1}^{n}\sum_{p,q}\left\{S^{-1}_{i-1,pq}(\alpha)(X_{i,p}-X_{i-1,p})(X_{i,q}-X_{i-1,q})\right\}-\frac{1}{2}\sum_{i=1}^{n}\log\det S_{i-1}(\alpha).
\end{align*}
For $l_1=1,\ldots, p_1$, we have
$$
(\partial_{\alpha_{l_1}}S^{-1})_{i-1}(\alpha)=-(S^{-1}(\partial_{\alpha_{l_1}}S)S^{-1})_{i-1}(\alpha),\ 
(\partial_{\alpha_{l_1}}\log\det S)_{i-1}(\alpha)=\mathrm{tr}\left((S^{-1}\partial_{\alpha_{l_{1}}}S)_{i-1}(\alpha)\right),
$$
and we set
\begin{align*}
&\partial_{\alpha_{l_1}}U_n^{(1)}(\alpha_0):=\sum_{i=1}^{n}{(\xi_{i,1}^{l_1}+\xi_{i,2}^{l_1})},\\
&\xi_{i,1}^{l_1}:=\frac{1}{2h_n}\sum_{p,q}(S^{-1}(\partial_{\alpha_{l_1}}S)S^{-1})_{i-1,pq}(\alpha_0)(X_{i,p}-X_{i-1,p})(X_{i,q}-X_{i-1,q}),\\
&\xi_{i,2}^{l_1}:=-\frac{1}{2}\mathrm{tr}\left((S^{-1}\partial_{\alpha_{l_{1}}}S)_{i-1}(\alpha_0)\right).
\end{align*}
From Theorems $3.2$ and $3.4$ of Hall and Heyde \cite{hall} , 
it is sufficient to show that under $H_{1,n}^{(1)}$,
\begin{align}
&\sum_{i=1}^{n}\mathbb{E}_{\theta_{1,n}^*}\left[\frac{1}{\sqrt{n}}(\xi_{i,1}^{l_1}+\xi_{i,2}^{l_1})| \mathcal{G}_{i-1}^n\right]\overset{P}{\to}\sum_{j=1}^{p_1}I_a^{(l_1j)}(\alpha_0;\alpha_0)u_{\alpha,j}, \label{MCLT:1}\\
&\sum_{i=1}^{n}\mathbb{E}_{\theta_{1,n}^*}\left[\frac{1}{n}(\xi_{i,1}^{l_1}+\xi_{i,2}^{l_1})(\xi_{i,1}^{l_2}+\xi_{i,2}^{l_2})| \mathcal{G}_{i-1}^n\right]\overset{P}{\to}I_a^{(l_1l_2)}(\alpha_0;\alpha_0), \label{MCLT:2}\\
&\sum_{i=1}^{n}\mathbb{E}_{\theta_{1,n}^*}\left[\frac{1}{\sqrt{n}}(\xi_{i,1}^{l_1}+\xi_{i,2}^{l_1})| \mathcal{G}_{i-1}^n\right]\mathbb{E}_{\theta_{1,n}^*}\left[\frac{1}{\sqrt{n}}(\xi_{i,1}^{l_2}+\xi_{i,2}^{l_2})| 
\mathcal{G}_{i-1}^n\right]\overset{P}{\to}0, \label{MCLT:3}\\
&\sum_{i=1}^{n}\mathbb{E}_{\theta_{1,n}^*}\left[\frac{1}{n^2}(\xi_{i,1}^{l_1}+\xi_{i,2}^{l_1})^4| \mathcal{G}_{i-1}^n\right]\overset{P}{\to}0,\label{MCLT:4}
\end{align}
where $l_1,l_2=1,\ldots,p_1$.

Proof of \eqref{MCLT:1}. 
From the It\^{o}-Taylor expansion and 
Lemma $7$ of Kessler \cite{Kessler_1997}, we have
$$
\mathbb{E}_{\theta_{1,n}^*}[(X_{i,p}-X_{i-1,p})(X_{i,q}-X_{i-1,q})|\mathcal{G}_{i-1}^n]=h_nS_{i-1,pq}(\alpha_{1,n}^*)+R_{i-1}(h_n^2),
$$
and it holds
\begin{align}
\mathbb{E}_{\theta_{1,n}^*}[\xi_{i,1}^{l_1}|\mathcal{G}_{i-1}^n]&=\frac{1}{2h_n}\sum_{p,q}(S^{-1}(\partial_{\alpha_{l_1}}S)S^{-1})_{i-1,pq}(\alpha_0)\mathbb{E}_{\theta_{1,n}^*}[(X_{i,p}-X_{i-1,p})(X_{i,q}-X_{i-1,q})|\mathcal{G}_{i-1}^n]\notag\\
&=\frac{1}{2}\sum_{p,q}(S^{-1}(\partial_{\alpha_{l_1}}S)S^{-1})_{i-1,pq}(\alpha_0)\left(S_{i-1,pq}(\alpha_{1,n}^*)+R_{i-1}(h_n)\right)\notag\\
&=\frac{1}{2}\sum_{p,q}(S^{-1}(\partial_{\alpha_{l_1}}S)S^{-1})_{i-1,pq}(\alpha_0)S_{i-1,pq}(\alpha_{1,n}^*)+R_{i-1}(h_n).\label{ee:1}
\end{align}
By Taylor's theorem, we obtain
\begin{align*}
S_{i-1,pq}(\alpha_{1,n}^*)&=S_{i-1,pq}(\alpha_0)+(\partial_\alpha S)_{i-1,pq}(\alpha_0)(\alpha_{1,n}^*-\alpha_0)\\
&\quad+(\alpha_{1,n}^*-\alpha_0)^\top\left(\int_{0}^{1}{u(\partial_{\alpha}^2S)_{i-1,pq}(\alpha_0+u(\alpha_{1,n}^*-\alpha_0))}du\right)(\alpha_{1,n}^*-\alpha_0)\\
&=S_{i-1, pq}(\alpha_0)+\frac{1}{\sqrt{n}}(\partial_\alpha S)_{i-1,pq}(\alpha_0)u_\alpha\\
&\quad+\frac{1}{n}u_\alpha^\top\left(\int_{0}^{1}{u(\partial_{\alpha}^2S)_{i-1,pq}(\alpha_0+u(\alpha_{1,n}^*-\alpha_0))}du\right) u_\alpha\\
&=S_{i-1, pq}(\alpha_0)+\frac{1}{\sqrt{n}}(\partial_\alpha S)_{i-1, pq}(\alpha_0)u_\alpha+\frac{1}{n}R_{i-1}(1).
\end{align*}
Note that $\displaystyle{\frac{h_n}{\frac{1}{n}}=nh_n>1}$ and $\displaystyle{R_{i-1}(h_n)+\frac{1}{n}R_{i-1}(1)=R_{i-1}(h_n)}$, it follows from \eqref{ee:1} that
\begin{align*}
\mathbb{E}_{\theta_{1,n}^*}[\xi_{i,1}^{l_1}|\mathcal{G}_{i-1}^n]&=\frac{1}{2}\sum_{p,q}(S^{-1}(\partial_{\alpha_{l_1}}S)S^{-1})_{i-1,pq}(\alpha_0)S_{i-1,pq}(\alpha_0)\\
&\quad+\frac{1}{2\sqrt{n}}\sum_{p,q}(S^{-1}(\partial_{\alpha_{l_1}}S)S^{-1})_{i-1,pq}(\alpha_0)(\partial_\alpha S)_{i-1,pq}(\alpha_0)u_\alpha+R_{i-1}(h_n)\\
&=\frac{1}{2}\mathrm{tr}\left((S^{-1}(\partial_{\alpha_{l_1}}S)S^{-1}S)_{i-1}(\alpha_0)\right)\\
&\quad+\frac{1}{2\sqrt{n}}\sum_{p,q}(S^{-1}(\partial_{\alpha_{l_1}}S)S^{-1})_{i-1,pq}(\alpha_0)\left(\sum_{j=1}^{p_1}(\partial_{\alpha_{j}} S)_{i-1,pq}(\alpha_0)\right)u_{\alpha,j}\\
&\quad+R_{i-1}(h_n)\\
&=\frac{1}{2}\mathrm{tr}\left((S^{-1}(\partial_{\alpha_{l_1}}S))_{i-1}(\alpha_0)\right)\\
&\quad+\frac{1}{2\sqrt{n}}\sum_{j=1}^{p_1}\mathrm{tr}\left((S^{-1}(\partial_{\alpha_{l_1}}S)S^{-1}(\partial_{\alpha_{j}}S))_{i-1}(\alpha_0)\right)u_{\alpha,j}+R_{i-1}(h_n).
\end{align*}
On the other hand, we have
\begin{equation}\label{ee:2}
\mathbb{E}_{\theta_{1,n}^*}[\xi_{i,2}^{l_1}|\mathcal{G}_{i-1}^n]=-\frac{1}{2}\mathrm{tr}\left((S^{-1}\partial_{\alpha_{l_{1}}}S)_{i-1}(\alpha_0)\right).\notag
\end{equation}
Hence, it follows from Lemma $8$ of Kessler \cite{Kessler_1997} that under $H_{0,n}^{(1)}$,
\begin{align*}
\sum_{i=1}^{n}\mathbb{E}_{\theta_{1,n}^*}\left[\frac{1}{\sqrt{n}}(\xi_{i,1}^{l_1}+\xi_{i,2}^{l_1})| \mathcal{G}_{i-1}^n\right]&=\frac{1}{n}\sum_{i=1}^{n}\left\{\frac{1}{2}\sum_{j=1}^{p_1}\mathrm{tr}\left((S^{-1}(\partial_{\alpha_{l_1}}S)S^{-1}(\partial_{\alpha_{j}}S))_{i-1}(\alpha_0)\right)u_{\alpha,j}\right\}\\
&\quad+\frac{1}{n}\sum_{i=1}^{n}R_{i-1}(\sqrt{nh_n^2})\\
&\overset{P}{\to}\frac{1}{2}\sum_{i=1}^{p_1}\int\mathrm{tr}\left((S^{-1}(\partial_{\alpha_{l_1}}S)S^{-1}(\partial_{\alpha_{j}}S))(x,\alpha_0)\right) \mu_{\theta_0}(dx)u_{\alpha,j}\\
&=\sum_{i=1}^{p_1}I_a^{(l_1j)}(\alpha_0;\alpha_0)u_{\alpha,j}.
\end{align*}
By setting
$$
Q_n:=\sum_{i=1}^{n}\mathbb{E}_{\theta_{1,n}^*}\left[\frac{1}{\sqrt{n}}(\xi_{i,1}^{l_1}+\xi_{i,2}^{l_1})| \mathcal{G}_{i-1}^n\right]-\sum_{i=1}^{p_1}I_a^{(l_1j)}(\alpha_0;\alpha_0)u_{\alpha,j},
$$
the above convergence implies that for all $\varepsilon>0$,
$$
\lim_{n\to\infty}P_{\theta_0}(|Q_n|\geq \varepsilon)=0.
$$
Hence it follows from $\mathbf{C1}$ that
$$
\lim_{n\to\infty}P_{\theta_{1,n}^*}(|Q_n|\geq \varepsilon)=0,
$$
which implies \eqref{MCLT:1}. 

Proof of \eqref{MCLT:2}. From 
the It\^{o}-Taylor expansion, it holds that for $1\leq p,q,r,s\leq p_1$,
\begin{align*}
&\mathbb{E}_{\theta_{1,n}^*}[(X_{i,p}-X_{i-1,p})(X_{i,q}-X_{i-1,q})(X_{i,r}-X_{i-1,r})(X_{i,s}-X_{i-1,s})|\mathcal{G}_{i-1}^n]\\
&\quad=h_n^2(S_{i-1,pq}(\alpha_{1,n}^*)S_{i-1,rs}(\alpha_{1,n}^*)+S_{i-1,pr}(\alpha_{1,n}^*)S_{i-1,qs}(\alpha_{1,n}^*)+S_{i-1,ps}(\alpha_{1,n}^*)S_{i-1,qr}(\alpha_{1,n}^*))\\
&\quad\quad+R_{i-1}(h_n^3).
\end{align*}
Noting that $\displaystyle{\frac{h_n}{\frac{1}{\sqrt{n}}}=\sqrt{nh_n^2}<1}$, we obtain
\begin{align}
\mathbb{E}_{\theta_{1,n}^*}[\xi_{i,1}^{l_1}\xi_{i,1}^{l_2}|\mathcal{G}_{i-1}^n]&=\frac{1}{4h_n^2}\sum_{p,q,r,s}(\partial_{\alpha_{l_1}}S^{-1})_{i-1,pq}(\alpha_0)(\partial_{\alpha_{l_2}}S^{-1})_{i-1,rs}(\alpha_0)\notag\\
&\quad+\mathbb{E}_{\theta_{1,n}^*}[(X_{i,p}-X_{i-1,p})(X_{i,q}-X_{i-1,q})(X_{i,r}-X_{i-1,r})(X_{i,s}-X_{i-1,s})|\mathcal{G}_{i-1}^n]\notag\\
&=\frac{1}{4}\sum_{p,q,r,s}(\partial_{\alpha_{l_1}}S^{-1})_{i-1,pq}(\alpha_0)(\partial_{\alpha_{l_2}}S^{-1})_{i-1,rs}(\alpha_0)\notag\\
&\quad\times(S_{i-1,pq}(\alpha_{1,n}^*)S_{i-1,rs}(\alpha_{1,n}^*)+S_{i-1,pr}(\alpha_{1,n}^*)S_{i-1,qs}(\alpha_{1,n}^*)\notag\\
&\quad\quad+S_{i-1,ps}(\alpha_{1,n}^*)S_{i-1,qr}(\alpha_{1,n}^*))\notag\\
&\quad+R_{i-1}(h_n)\notag\\
&=\frac{1}{4}\sum_{p,q,r,s}(\partial_{\alpha_{l_1}}S^{-1})_{i-1,pq}(\alpha_0)(\partial_{\alpha_{l_2}}S^{-1})_{i-1,rs}(\alpha_0)\notag\\
&\quad\times(S_{i-1,pq}(\alpha_0)S_{i-1,rs}(\alpha_0)+S_{i-1,pr}(\alpha_0)S_{i-1,qs}(\alpha_0)\notag\\
&\quad\quad+S_{i-1,ps}(\alpha_0)S_{i-1,qr}(\alpha_0))\notag\\
&\quad+\frac{1}{\sqrt{n}}R_{i-1}(1).\label{ee:3}
\end{align}
Since
\begin{align*}
&\sum_{p,q,r,s}(\partial_{\alpha_{l_1}}S^{-1})_{i-1,pq}(\alpha_0)(\partial_{\alpha_{l_2}}S^{-1})_{i-1,rs}(\alpha_0)S_{i-1,pq}(\alpha_0)S_{i-1,rs}(\alpha_0)\\
&\quad=\mathrm{tr}\left(((\partial_{\alpha_{l_1}}S^{-1})S)_{i-1}(\alpha_0)\right)\mathrm{tr}\left(((\partial_{\alpha_{l_2}}S^{-1})S)_{i-1}(\alpha_0)\right),\\
&\sum_{p,q,r,s}(\partial_{\alpha_{l_1}}S^{-1})_{i-1,pq}(\alpha_0)(\partial_{\alpha_{l_2}}S^{-1})_{i-1,rs}(\alpha_0)S_{i-1,pr}(\alpha_0)S_{i-1,qs}(\alpha_0)\\
&\quad=\sum_{p,q,r,s}(\partial_{\alpha_{l_1}}S^{-1})_{i-1,pq}(\alpha_0)(\partial_{\alpha_{l_2}}S^{-1})_{i-1,rs}(\alpha_0)S_{i-1,ps}(\alpha_0)S_{i-1,qr}(\alpha_0)\\
&\quad=\mathrm{tr}\left(((\partial_{\alpha_{l_1}}S^{-1})S(\partial_{\alpha_{l_2}}S^{-1})S)_{i-1}(\alpha_0)\right),
\end{align*}
it follows from \eqref{ee:3} that 
\begin{align*}
\mathbb{E}_{\theta_{1,n}^*}[\xi_{i,1}^{l_1}\xi_{i,1}^{l_2}|\mathcal{G}_{i-1}^n]&\quad=\frac{1}{4}\mathrm{tr}\left(((\partial_{\alpha_{l_1}}S^{-1})S)_{i-1}(\alpha_0)\right)\mathrm{tr}\left(((\partial_{\alpha_{l_2}}S^{-1})S)_{i-1}(\alpha_0)\right)\\
&\quad\quad+\frac{1}{2}\mathrm{tr}\left(((\partial_{\alpha_{l_1}}S^{-1})S(\partial_{\alpha_{l_2}}S^{-1})S)_{i-1}(\alpha_0)\right)+\frac{1}{\sqrt{n}}R_{i-1}(1).
\end{align*}
Moreover, we have
\begin{align*}
\mathbb{E}_{\theta_{1,n}^*}[\xi_{i,2}^{l_1}\xi_{i,2}^{l_2}|\mathcal{G}_{i-1}^n]&=\frac{1}{4}\mathrm{tr}\left((S^{-1}\partial_{\alpha_{l_{1}}}S)_{i-1}(\alpha_0)\right)\mathrm{tr}\left((S^{-1}\partial_{\alpha_{l_{2}}}S)_{i-1}(\alpha_0)\right),\\
\mathbb{E}_{\theta_{1,n}^*}[\xi_{i,1}^{l_1}\xi_{i,2}^{l_2}|\mathcal{G}_{i-1}^n]&=-\frac{1}{4h_n}\mathrm{tr}\left((S^{-1}\partial_{\alpha_{l_{2}}}S)_{i-1}(\alpha_0)\right)\sum_{p,q}(S^{-1}(\partial_{\alpha_{l_1}}S)S^{-1})_{i-1,pq}(\alpha_0)\\
&\quad\times\mathbb{E}_{\theta_{1,n}^*}[(X_{i,p}-X_{i-1,p})(X_{i,q}-X_{i-1,q})|\mathcal{G}_{i-1}^n]\\
&=-\frac{1}{4}\mathrm{tr}\left((S^{-1}\partial_{\alpha_{l_{2}}}S)_{i-1}(\alpha_0)\right)\sum_{p,q}(S^{-1}(\partial_{\alpha_{l_1}}S)S^{-1})_{i-1,pq}(\alpha_0)S_{i-1,pq}(\alpha_0)\\
&\quad+\frac{1}{\sqrt{n}}R_{i-1}(1)\\
&=-\frac{1}{4}\mathrm{tr}\left((S^{-1}\partial_{\alpha_{l_{2}}}S)_{i-1}(\alpha_0)\right)\mathrm{tr}\left((S^{-1}\partial_{\alpha_{l_{1}}}S)_{i-1}(\alpha_0)\right)+\frac{1}{\sqrt{n}}R_{i-1}(1).
\end{align*}
Therefore, it holds that under $H_{0,n}^{(1)}$,  
\begin{align*}
\sum_{i=1}^{n}\mathbb{E}_{\theta_{1,n}^*}\left[\frac{1}{n}(\xi_{i,1}^{l_1}+\xi_{i,2}^{l_1})(\xi_{i,1}^{l_2}+\xi_{i,2}^{l_2})| \mathcal{G}_{i-1}^n\right]&=\frac{1}{n}\sum_{i=1}^{n}\frac{1}{2}\mathrm{tr}\left(((\partial_{\alpha_{l_1}}S^{-1})S(\partial_{\alpha_{l_2}}S^{-1})S)_{i-1}(\alpha_0)\right)\\&\quad+\frac{1}{n}\sum_{i=1}^{n}\frac{1}{\sqrt{n}}R_{i-1}(1)\\
&\overset{P}{\to}\frac{1}{2}\int\mathrm{tr}\left(((\partial_{\alpha_{l_1}}S^{-1})S(\partial_{\alpha_{l_2}}S^{-1})S)(x,\alpha_0)\right)\mu_{\theta_0}(dx)\\
&=\frac{1}{2}\int\mathrm{tr}\left((S^{-1}(\partial_{\alpha_{l_1}}S)S^{-1}(\partial_{\alpha_{l_2}}S))(x,\alpha_0)\right)\mu_{\theta_0}(dx)\\
&=I_a^{(l_1l_2)}(\alpha_0;\alpha_0).
\end{align*}
From $\mathbf{C1}$, this convergence also holds under $H_{1,n}^{(1)}$.

Proof of \eqref{MCLT:3}. From the proof of \eqref{MCLT:1}, we have
\begin{align*}
\mathbb{E}_{\theta_{1,n}^*}\left[\frac{1}{\sqrt{n}}(\xi_{i,1}^{l_1}+\xi_{i,2}^{l_1})| \mathcal{G}_{i-1}^n\right]&=\frac{1}{n}R_{i-1}(1),\\
\mathbb{E}_{\theta_{1,n}^*}\left[\frac{1}{\sqrt{n}}(\xi_{i,1}^{l_2}+\xi_{i,2}^{l_2})| \mathcal{G}_{i-1}^n\right]&=\frac{1}{n}R_{i-1}(1).
\end{align*}
Therefore, it holds that under $H_{0,n}^{(1)}$,
\begin{align*}
\sum_{i=1}^{n}\mathbb{E}_{\theta_{1,n}^*}\left[\frac{1}{\sqrt{n}}(\xi_{i,1}^{l_1}+\xi_{i,2}^{l_1})| \mathcal{G}_{i-1}^n\right]\mathbb{E}_{\theta_{1,n}^*}\left[\frac{1}{\sqrt{n}}(\xi_{i,1}^{l_2}+\xi_{i,2}^{l_2})| \mathcal{G}_{i-1}^n\right]&=\frac{1}{n}\sum_{i=1}^{n}\frac{1}{n}R_{i-1}(1)\overset{P}{\to}0,
\end{align*}
and this also holds under $H_{1,n}^{(1)}$.

Proof of \eqref{MCLT:4}. It is easy to show
$$
\mathbb{E}_{\theta_{1,n}^*}\left[\frac{1}{n^2}(\xi_{i,1}^{l_1}+\xi_{i,2}^{l_1})^4| \mathcal{G}_{i-1}^n\right]\leq\frac{2^3}{n^2}\left(\mathbb{E}_{\theta_{1,n}^*}\left[(\xi_{i,1}^{l_1})^4| \mathcal{G}_{i-1}^n\right]+\mathbb{E}_{\theta_{1,n}^*}\left[(\xi_{i,2}^{l_1})^4| \mathcal{G}_{i-1}^n\right]\right).
$$
We can evaluate
\begin{align}
\mathbb{E}_{\theta_{1,n}^*}\left[(\xi_{i,1}^{l_1})^4| \mathcal{G}_{i-1}^n\right]&=\frac{1}{16h_n^4}\mathbb{E}_{\theta_{1,n}^*}[(\sum_{p,q}(\partial_{\alpha_{l_1}}S^{-1})_{i-1,pq}(\alpha_0)\notag\\
&\quad\quad\quad\quad\quad\quad\quad\times(X_{i,p}-X_{i-1,q})(X_{i,q}-X_{i-1,q}))^4| \mathcal{G}_{i-1}^n]\notag\\
&\leq\frac{C}{16h_n^4}(1+|X_{t_{i-1}^n}|)^C\sum_{1\leq p_1,\ldots p_8\leq d}\mathbb{E}_{\theta_{1,n}^*}\left[\prod_{j=1}^8(X_{i,p_j}-X_{i-1,p_j})|\mathcal{G}_{i-1}^n\right]\notag\\
&\leq R_{i-1}(h_n^{-4})\sum_{1\leq p_1,\ldots p_8\leq d}\prod_{j=1}^8\left(\mathbb{E}_{\theta_{1,n}^*}[(X_{i,p_j}-X_{i-1,p_j})^8|\mathcal{G}_{i-1}^n]\right)^{\frac{1}{8}}\notag\\
&\leq d^8R_{i-1}(h_n^{-4})\mathbb{E}_{\theta_{1,n}^*}[|X_{t_{i}^n}-X_{{t}_{i-1}}^n|^8|\mathcal{G}_{i-1}^n]\notag\\
&\leq R_{i-1}(1),\notag\\
\mathbb{E}_{\theta_{1,n}^*}\left[(\xi_{i,2}^{l_1})^4| \mathcal{G}_{i-1}^n\right]&=\frac{1}{16}\left(\mathrm{tr}\left((S^{-1}\partial_{\alpha_{l_{1}}}S)_{i-1}(\alpha_0)\right)\right)^4\notag\\&=R_{i-1}(1).\notag
\end{align}
Hence
$$
\sum_{i=1}^{n}\frac{2^3}{n^2}\left(\mathbb{E}_{\theta_{1,n}^*}\left[(\xi_{i,1}^{l_1})^4| \mathcal{G}_{i-1}^n\right]+\mathbb{E}_{\theta_{1,n}^*}\left[(\xi_{i,2}^{l_1})^4| \mathcal{G}_{i-1}^n\right]\right)=\frac{1}{n}\sum_{i=1}^{n}\frac{1}{n}R_{i-1}(1)\overset{P}{\to}0,
$$
and
$$
\sum_{i=1}^{n}\mathbb{E}_{\theta_{1,n}^*}\left[\frac{1}{n^2}(\xi_{i,1}^{l_1}+\xi_{i,2}^{l_1})^4| \mathcal{G}_{i-1}^n\right]\overset{P}{\to}0
\quad(\text{under}\ H_{0,n}^{(1)}).
$$
From $\mathbf{C1}$, the last convergence also holds under $H_{1,n}^{(1)}$. 
By using \eqref{MCLT:1}-\eqref{MCLT:4}, we get \eqref{ggoal:1}.

Next we prove $(\mathrm{i})$. It follows from Taylor's theorem that
$$
-\frac{1}{\sqrt{n}}\partial_{\alpha}U_n^{(1)}(\alpha_0)=\int_{0}^{1}{\frac{1}{n}\partial_{\alpha}^2U_n^{(1)}(\alpha_0+u(\hat{\alpha}_n-\alpha_0))}du\sqrt{n}(\hat{\alpha}_n-\alpha_0).
$$
It holds
$$
\int_{0}^{1}{\frac{1}{n}\partial_{\alpha}^2U_n^{(1)}(\alpha_0+u(\hat{\alpha}_n-\alpha_0))}du\overset{P}{\to}-I_a(\alpha_0;\alpha_0)\quad(\text{under}\ H_{0,n}^{(1)}),
$$
which 
also holds under $H_{1,n}^{(1)}$. 
By noting that $I_a(\alpha_0;\alpha_0)$ is non-singular, it follows from Slutsky's theorem that
$$
\sqrt{n}(\hat{\alpha}_n-\alpha_0)\overset{d}{\to}I_a^{-1}(\alpha_0;\alpha_0)L_1=Y_1'\quad(\text{under}\ H_{1,n}^{(1)}).
$$
This completes the proof of $(\mathrm{i})$.

$(\mathrm{ii})$ First of all, 
we will prove that
\begin{align}\label{ggoal:2}
\frac{1}{\sqrt{nh_n}}\partial_\beta U_n^{(2)}(\beta_0|\alpha_0)\overset{d}{\to}L_2\sim N(I_b(\theta_0;\theta_0)u_\beta,I_b(\theta_0;\theta_0))\quad(\text{under}\ H_{1,n}^{(2)}). 
\end{align}
Letting 
$b(X_{t_{i-1}^n},\beta)=(b_{i-1,1}(\beta),\ldots,b_{i-1,d}(\beta))^\top$
and $u_\beta=(u_{\beta,1},\ldots,u_{\beta,p_2})^\top$,
one has that  
\begin{align*}
U_n^{(2)}(\beta|\bar{\alpha})&=-\frac{1}{2}\sum_{i=1}^{n}\left\{h_n^{-1}S^{-1}(X_{t_{i-1}^n},\bar{\alpha})\left[(X_{t_{i}^n}-X_{t_{i-1}^n}-h_nb(X_{t_{i-1}^n},\beta))^{\otimes 2}\right]\right\}\\
&=-\frac{1}{2h_n}\sum_{i=1}^{n}\sum_{p,q}S_{i-1,pq}^{-1}(\bar{\alpha})(X_{i,p}-X_{i-1,p}-h_nb_{i-1,p}(\beta))(X_{i,q}-X_{i-1,q}-h_nb_{i-1,q}(\beta)).
\end{align*}
For $m_1=1,\ldots,p_2$, 
we have
\begin{align*}
&\partial_{\beta_{m_1}}U_n^{(2)}(\beta_0|\alpha_0)=\sum_{i=1}^{n}{\eta_{i,1}^{m_1}}\\
&\eta_{i,1}^{m_1}=\sum_{p,q}S_{i-1,pq}^{-1}(\alpha_0)(\partial_{\beta_{m_1}}b)_{i-1,q}(\beta_0)(X_{i,p}-X_{i-1,p}-h_nb_{i-1,p}(\beta_0)).
\end{align*}
From Theorems $3.2$ and $3.4$ of Hall and Heyde \cite{hall}  and $\mathbf{C1}$, 
it is sufficient to show that under $H_{0,n}^{(2)}$, 
\begin{align}
&\sum_{i=1}^{n}\mathbb{E}_{\theta_{1,n}^*}\left[\frac{1}{\sqrt{nh_n}}\eta_{i,1}^{m_1}| \mathcal{G}_{i-1}^n\right]\overset{P}{\to}\sum_{j=1}^{p_2}I_b^{(m_1j)}(\theta_0;\theta_0)u_{\beta,j}, \label{MCLT:5}\\
&\sum_{i=1}^{n}\mathbb{E}_{\theta_{1,n}^*}\left[\frac{1}{nh_n}\eta_{i,1}^{m_1}\eta_{i,1}^{m_2}| \mathcal{G}_{i-1}^n\right]\overset{P}{\to}I_b^{(m_1m_2)}(\theta_0;\theta_0), \label{MCLT:6}\\
&\sum_{i=1}^{n}\mathbb{E}_{\theta_{1,n}^*}\left[\frac{1}{\sqrt{nh_n}}\eta_{i,1}^{m_1}| \mathcal{G}_{i-1}^n\right]\mathbb{E}_{\theta_{1,n}^*}\left[\frac{1}{\sqrt{nh_n}}\eta_{i,1}^{m_2}| \mathcal{G}_{i-1}^n\right]\overset{P}{\to}0, \label{MCLT:7}\\
&\sum_{i=1}^{n}\mathbb{E}_{\theta_{1,n}^*}\left[\frac{1}{(nh_n)^2}(\eta_{i,1}^{m_1})^4| \mathcal{G}_{i-1}^n\right]\overset{P}{\to}0,\label{MCLT:8}
\end{align}
where $m_1,m_2=1,\ldots,p_2$.

Proof of \eqref{MCLT:5}. By the It\^{o}-Taylor expansion, we have
$$
\mathbb{E}_{\theta_{1,n}^*}[X_{i,p}-X_{i-1,p}|\mathcal{G}_{i-1}^n]=h_nb_{i-1,p}(\beta_{1,n}^*)+R_{i-1}(h_n^2),
$$
and
\begin{align}\label{tay:beta}
\mathbb{E}_{\theta_{1,n}^*}[\eta_{i,1}^{m_1}|\mathcal{G}_{i-1}^n]&=\sum_{p,q}S_{i-1,pq}^{-1}(\alpha_0)(\partial_{\beta_{m_1}}b)_{i-1,q}(\beta_0)\left(\mathbb{E}_{\theta_{1,n}^*}[X_{i,p}-X_{i-1,p}|\mathcal{G}_{i-1}^n]-h_nb_{i-1,p}(\beta_0)\right)\notag\\
&=h_n\sum_{p,q}S_{i-1,pq}^{-1}(\alpha_0)(\partial_{\beta_{m_1}}b)_{i-1,q}(\beta_0)(b_{i-1,p}(\beta_{1,n}^*)-b_{i-1,p}(\beta_0))+R_{i-1}(h_n^2).
\end{align}
Since it follows from Taylor's theorem that
\begin{align*}
b_{i-1,p}(\beta_{1,n}^*)-b_{i-1,p}(\beta_0)&=(\partial_\beta b)_{i-1,p}(\beta_0)(\beta_{1,n}^*-\beta_0)\\
&\quad+(\beta_{1,n}^*-\beta_0)^\top\left(\int_{0}^{1}u(\partial_\beta^2 b)_{i-1,p}(\beta_0+u(\beta_{1,n}^*-\beta_0)) du\right)(\beta_{1,n}^*-\beta_0)\\
&=\frac{1}{\sqrt{nh_n}}(\partial_{\beta}b)_{i-1,p}(\beta_0)u_\beta+\frac{1}{nh_n}R_{i-1}(1), 
\end{align*}
\eqref{tay:beta} implies that
\begin{align*}
\mathbb{E}_{\theta_{1,n}^*}[\eta_{i,1}^{m_1}|\mathcal{G}_{i-1}^n]&=\sqrt{\frac{h_n}{n}}\sum_{p,q}S_{i-1,pq}^{-1}(\alpha_0)(\partial_{\beta_{m_1}}b)_{i-1,q}(\beta_0)(\partial_{\beta}b)_{i-1,p}(\beta_0)u_\beta+\frac{1}{n}R_{i-1}(1)\\
&=\sqrt{\frac{h_n}{n}}\sum_{j=1}^{p_2}\sum_{p,q}S_{i-1,pq}^{-1}(\alpha_0)(\partial_{\beta_{m_1}}b)_{i-1,q}(\beta_0)(\partial_{\beta_{j}}b)_{i-1,p}(\beta_0)u_{\beta,j}+\frac{1}{n}R_{i-1}(1)\\
&=\sqrt{\frac{h_n}{n}}\sum_{j=1}^{p_2}\left((\partial_{\beta_{m_1}}b)_{i-1}(\beta_0)\right)^\top S^{-1}_{i-1}(\alpha_0)(\partial_{\beta_{j}}b)_{i-1}(\beta_0)u_{\beta,j}+\frac{1}{n}R_{i-1}(1).
\end{align*}
Therefore, 
it follows from Lemma $8$ of Kessler \cite{Kessler_1997} that under $H_{0,n}^{(1)}$,
\begin{align*}
\sum_{i=1}^{n}\mathbb{E}_{\theta_{1,n}^*}\left[\frac{1}{\sqrt{nh_n}}\eta_{i,1}^{m_1}| \mathcal{G}_{i-1}^n\right]&=\frac{1}{n}\sum_{i=1}^{n}\sum_{j=1}^{p_2}\left((\partial_{\beta_{m_1}}b)_{i-1}(\beta_0)\right)^\top S^{-1}_{i-1}(\alpha_0)(\partial_{\beta_{j}}b)_{i-1}(\beta_0)u_{\beta,j}\\
&\quad+\frac{1}{n}\sum_{i=1}^{n}\frac{1}{\sqrt{nh_n}}R_{i-1}(1)\\
&\overset{P}{\to}\sum_{j=1}^{p_2}\int\left(\partial_{\beta_{m_1}}b(x,\beta_0)\right)^\top S^{-1}(x,\alpha_0)\left(\partial_{\beta_{j}}b(x,\beta_0)\right)\mu_{\theta_0}(dx)u_{\beta,j}\\
&=\sum_{j=1}^{p_2}{I_{b}^{(m_1j)}(\theta_0;\theta_0)u_{\beta,j}}.
\end{align*}

Proof of \eqref{MCLT:6}. By the It\^{o}-Taylor expansion, one has that
$$
\mathbb{E}_{\theta_{1,n}^*}[(X_{i,p}-X_{i-1,p}-h_nb_{i-1,p}(\beta_0))(X_{i,q}-X_{i-1,q}-h_nb_{i-1,q}(\beta_0))|\mathcal{G}_{i-1}^n]=h_nS_{i-1,pq}(\alpha_0)+\frac{1}{\sqrt{n}}R_{i-1}(h_n).
$$
Setting
$\delta_{ps}:=\begin{cases}1&(p=s), \\0&(p\neq s), \end{cases}$ 
we obtain that
\begin{align*}
\mathbb{E}_{\theta_{1,n}^*}[\eta_{i,1}^{m_1}\eta_{i,1}^{m_2}|\mathcal{G}_{i-1}^n]&=\sum_{p,q,r,s}S^{-1}_{i-1,pq}(\alpha_0)S^{-1}_{i-1,rs}(\alpha_0)(\partial_{\beta_{m_1}}b)_{i-1,q}(\beta_0)(\partial_{\beta_{m_2}}b)_{i-1,s}(\beta_0)\\
&\ \ \times\mathbb{E}_{\theta_{1,n}^*}[(X_{i,p}-X_{i-1,p}-h_nb_{i-1,p}(\beta_0))(X_{i,r}-X_{i-1,r}-h_nb_{i-1,r}(\beta_0))|\mathcal{G}_{i-1}^n]\\
&=h_n\sum_{p,q,r,s}S^{-1}_{i-1,pq}(\alpha_0)S^{-1}_{i-1,rs}(\alpha_0)S_{i-1,pr}(\alpha_0)\\
&\ \ \times(\partial_{\beta_{m_1}}b)_{i-1,q}(\beta_0)(\partial_{\beta_{m_2}}b)_{i-1,s}(\beta_0)+\frac{1}{\sqrt{n}}R_{i-1}(h_n)\\
&=h_n\sum_{p,q,s}S^{-1}_{i-1,pq}(\alpha_0)\left(\sum_{r=1}^{d}S^{-1}_{i-1,rs}(\alpha_0)S_{i-1,rp}(\alpha_0)\right)\\
&\ \ \times(\partial_{\beta_{m_1}}b)_{i-1,q}(\beta_0)(\partial_{\beta_{m_2}}b)_{i-1,s}(\beta_0)+\frac{1}{\sqrt{n}}R_{i-1}(h_n)\\
&=h_n\sum_{p,q,s}S^{-1}_{i-1,pq}(\alpha_0)\delta_{ps}(\partial_{\beta_{m_1}}b)_{i-1,q}(\beta_0)(\partial_{\beta_{m_2}}b)_{i-1,s}(\beta_0)+\frac{1}{\sqrt{n}}R_{i-1}(h_n)\\
&=h_n\sum_{p,q}S^{-1}_{i-1,pq}(\alpha_0)(\partial_{\beta_{m_1}}b)_{i-1,q}(\beta_0)(\partial_{\beta_{m_2}}b)_{i-1,p}(\beta_0)+\frac{1}{\sqrt{n}}R_{i-1}(h_n)\\
&=h_n\left((\partial_{\beta_{m_1}}b)_{i-1}(\beta_0)\right)^\top S^{-1}_{i-1}(\alpha_0)\left((\partial_{\beta_{m_2}}b)_{i-1}(\beta_0)\right)+\frac{1}{\sqrt{n}}R_{i-1}(h_n).
\end{align*}
Hence it holds
\begin{align*}
\sum_{i=1}^{n}\mathbb{E}_{\theta_{1,n}^*}\left[\frac{1}{nh_n}\eta_{i,1}^{m_1}\eta_{i,1}^{m_2}| \mathcal{G}_{i-1}^n\right]&=\frac{1}{n}\sum_{i=1}^{n}\left((\partial_{\beta_{m_1}}b)_{i-1}(\beta_0)\right)^\top S^{-1}_{i-1}(\alpha_0)\left((\partial_{\beta_{m_2}}b)_{i-1}(\beta_0)\right)\\
&\quad+\frac{1}{n}\sum_{i=1}^{n}\frac{1}{\sqrt{n}}R_{i-1}(1)\\
&\overset{P}{\to}\int\left(\partial_{\beta_{m_1}}b(x,\beta_0)\right)^\top S^{-1}(x,\alpha_0)\left(\partial_{\beta_{m_2}}b(x,\beta_0)\right)\mu_{\theta_0}(dx)\\
&=I_b^{(m_1m_2)}(\theta_0;\theta_0).
\end{align*}

Proof of \eqref{MCLT:7}. From the proof of \eqref{MCLT:5}, we have
\begin{align*}
&\mathbb{E}_{\theta_{1,n}^*}\left[\frac{1}{\sqrt{nh_n}}\eta_{i,1}^{m_1}| \mathcal{G}_{i-1}^n\right]=\frac{1}{n}R_{i-1}(1),\\
&\mathbb{E}_{\theta_{1,n}^*}\left[\frac{1}{\sqrt{nh_n}}\eta_{i,1}^{m_2}| \mathcal{G}_{i-1}^n\right]=\frac{1}{n}R_{i-1}(1).
\end{align*}
Hence
$$
\sum_{i=1}^{n}\mathbb{E}_{\theta_{1,n}^*}\left[\frac{1}{\sqrt{nh_n}}\eta_{i,1}^{m_1}| \mathcal{G}_{i-1}^n\right]\mathbb{E}_{\theta_{1,n}^*}\left[\frac{1}{\sqrt{nh_n}}\eta_{i,1}^{m_2}| \mathcal{G}_{i-1}^n\right]=\frac{1}{n}\sum_{i=1}^{n}\frac{1}{n}R_{i-1}(1)\overset{P}{\to}0.
$$

Proof of \eqref{MCLT:8}. 
It follows from the It\^{o}-Taylor expansion that
$$
\mathbb{E}_{\theta_{1,n}^*}\left[\prod_{j=1}^4(X_{i,p_j}-X_{i-1,p_j}-h_nb_{i-1,p_j}(\beta_0))|\mathcal{G}_{i-1}^n\right]=R_{i-1}(h_n^2).
$$
Hence we can evaluate
\begin{align*}
\mathbb{E}_{\theta_{1,n}^*}[(\eta_{i,1}^{m_1})^4|\mathcal{G}_{i-1}^n]&=\mathbb{E}_{\theta_{1,n}^*}[(\sum_{p,q}S_{i-1,pq}^{-1}(\alpha_0)(\partial_{\beta_{m_1}}b)_{i-1,q}(\beta_0)\\
&\quad\quad\quad\quad\quad\times(X_{i,p}-X_{i-1,p}-h_nb_{i-1,p}(\beta_0))
)_{_{}}^4|\mathcal{G}_{i-1}^n]\\
&\leq R_{i-1}(1)\sum_{1\leq p_1,\ldots,p_4\leq d}\mathbb{E}_{\theta_{1,n}^*}\left[\prod_{j=1}^4(X_{i,p_j}-X_{i-1,p_j}-h_nb_{i-1,p_j}(\beta_0))|\mathcal{G}_{i-1}^n\right]\\
&\leq R_{i-1}(h_n^2).
\end{align*}
Therefore it holds
$$
0\leq\sum_{i=1}^{n}\mathbb{E}_{\theta_{1,n}^*}\left[\frac{1}{(nh_n)^2}(\eta_{i,1}^{m_1})^4| \mathcal{G}_{i-1}^n\right]\leq\frac{1}{n}\sum_{i=1}^{n}\frac{1}{n}R_{i-1}(1)\overset{P}{\to}0,
$$
and 
$$
\sum_{i=1}^{n}\mathbb{E}_{\theta_{1,n}^*}\left[\frac{1}{(nh_n)^2}(\eta_{i,1}^{m_1})^4| \mathcal{G}_{i-1}^n\right]\overset{P}{\to}0.
$$
Using \eqref{MCLT:5}-\eqref{MCLT:8},
we obtain \eqref{ggoal:2}. 

Next, one shows $(\mathrm{ii})$. 
From Taylor's theorem with respect to $\beta$, we have
\begin{align}\label{lem3:1}
-\frac{1}{\sqrt{nh_n}}\partial_\beta U_n^{(2)}(\beta_0|\hat{\alpha}_n)=\int_{0}^{1}{\frac{1}{nh_n}\partial_\beta^2U_n^{(2)}(\beta_0+u(\hat{\beta}_n-\beta_0)|\hat{\alpha}_n)}du\sqrt{nh_n}(\hat{\beta}_n-\beta_0),
\end{align}
and Taylor's theorem with respect to  $\alpha$ yields that 
\begin{align}\label{lem3:2}
\frac{1}{\sqrt{nh_n}}\partial_\beta U_n^{(2)}(\beta_0|\hat{\alpha}_n)&=\frac{1}{\sqrt{nh_n}}\partial_\beta U_n^{(2)}(\beta_0|\alpha_0)\notag\\
&\quad+\int_{0}^{1}{\frac{1}{n\sqrt{h_n}}\partial_{\alpha\beta}^2U_n^{(2)}(\beta_0|\alpha_0+u(\hat{\alpha}_n-\alpha_0))}du\sqrt{n}(\hat{\alpha}_n-\alpha_0).
\end{align}
One has that under $H_{0,n}^{(2)}$,
\begin{align*}
&\int_{0}^{1}{\frac{1}{nh_n}\partial_\beta^2U_n^{(2)}(\beta_0+u(\hat{\beta}_n-\beta_0)|\hat{\alpha}_n)}du\overset{P}{\to}-I_b(\theta_0;\theta_0),\\
&\int_{0}^{1}{\frac{1}{n\sqrt{h_n}}\partial_{\alpha\beta}^2U_n^{(2)}(\beta_0|\alpha_0+u(\hat{\alpha}_n-\alpha_0))}du\overset{P}{\to}0,
\end{align*}
which also hold under $H_{1,n}^{(2)}$ by $\mathbf{C1}$. 
Therefore,
it follows from \eqref{lem3:1} and \eqref{lem3:2}  that 
$$
-\int_{0}^{1}{\frac{1}{nh_n}\partial_\beta^2U_n^{(2)}(\beta_0+u(\hat{\beta}_n-\beta_0)|\hat{\alpha}_n)}du\sqrt{nh_n}(\hat{\beta}_n-\beta_0)=\frac{1}{\sqrt{nh_n}}\partial_\beta U_n^{(2)}(\beta_0|\alpha_0)+o_P(1).
$$
Since $I_b(\theta_0;\theta_0)$ is non-singular, 
by Slutsky's theorem, one has that
$$
\sqrt{nh_n}(\hat{\beta}_n-\beta_0)\overset{d}{\to}I_b^{-1}(\theta_0;\theta_0)L_2=Y_2'\sim N(u_\beta,I_b^{-1}(\theta_0;\theta_0))\quad(\text{under}\ H_{1,n}^{(2)}).
$$
This completes the proof of $(\mathrm{ii})$.
\end{proof}
\begin{proof}[Proof of Theorem 3]\ $(\mathrm{i})$ By Taylor's theorem,
$$
U_n^{(1)}(\alpha_0)-U_n^{(1)}(\hat{\alpha}_n)=\left(\sqrt{n}(\hat{\alpha}_n-\alpha_0)\right)^\top J_{a,n}(\alpha_0,\hat{\alpha}_n)\sqrt{n}(\hat{\alpha}_n-\alpha_0).
$$
It follows from Lemma \ref{lem:3} that
$$
\sqrt{n}(\hat{\alpha}_n-\alpha_0)\overset{d}{\to}Y_1'\sim N(u_\alpha,I_a^{-1}(\alpha_0;\alpha_0))\quad(\text{under}\ H_{1,n}^{(1)}).
$$
Moreover, $\mathbf{C1}$ and
$$
J_{a,n}(\alpha_0,\hat{\alpha}_n)\overset{P}{\to}-\frac{1}{2}I_a(\alpha_0;\alpha_0)\quad(\text{under}\ H_{0,n}^{(1)})
$$
implies
$$
J_{a,n}(\alpha_0,\hat{\alpha}_n)\overset{P}{\to}-\frac{1}{2}I_a(\alpha_0;\alpha_0)\quad(\text{under}\ H_{1,n}^{(1)}).
$$
By noting that $Y_1'=I_a^{-\frac{1}{2}}(\alpha_0;\alpha_0)Z_1',\ Z_1'\sim N(I_a^{\frac{1}{2}}(\alpha_0;\alpha_0)u_\alpha,E_{p_1})$, it follows from the continuous mapping theorem that
\begin{align*}
\Lambda_n^{(1)}&=-2(U_n^{(1)}(\alpha_0)-U_n^{(1)}(\hat{\alpha}_n))\\
&\overset{d}{\to}{Y_1'}^\top I_a(\alpha_0;\alpha_0)Y_1'\\
&={Z_1'}^\top Z_1'\\
&\sim \chi_{p_1}^2(u_\alpha^\top I_a(\alpha_0;\alpha_0)u_\alpha)\quad(\text{under}\ H_{1,n}^{(1)}).
\end{align*}
Moreover,
 by Taylor's theorem with respect to $\beta$, we have
$$
U_n^{(2)}(\beta_0|\hat{\alpha}_n)-U_n^{(2)}(\hat{\beta}_n|\hat{\alpha}_n)=\left(\sqrt{nh_n}(\hat{\beta}_n-\beta_0)\right)^\top J_{b,n}(\beta_0,\hat{\beta}_n)\sqrt{nh_n}(\hat{\beta}_n-\beta_0).
$$
It follows from Lemma \ref{lem:3}  that
$$
\sqrt{nh_n}(\hat{\beta}_n-\beta_0)\overset{d}{\to}Y_2'\sim N(u_\beta,I_b^{-1}(\theta_0;\theta_0))\quad(\text{under}\ H_{1,n}^{(2)}).
$$
Next, we can check
$$
J_{b,n}(\beta_0,\hat{\beta}_n)\overset{P}{\to}-\frac{1}{2}I_b(\theta_0;\theta_0)\quad(\text{under}\ H_{1,n}^{(2)}).
$$
Note that $Y_2'=I_b^{-\frac{1}{2}}(\theta_0;\theta_0)Z_2',\ Z_2'\sim N(I_b^{\frac{1}{2}}(\theta_0;\theta_0)u_\beta,E_{p_2})$. 
By the continuous mapping theorem,  one has that 
\begin{align*}
\Lambda_n^{(2)}&=-2(U_n^{(2)}(\beta_0|\hat{\alpha}_n)-U_n^{(2)}(\hat{\beta}_n|\hat{\alpha}_n))\\
&\overset{d}{\to}{Y_2'}^\top I_b(\theta_0;\theta_0)Y_2'\\
&={Z_2'}^\top Z_2'\\
&\sim \chi_{p_2}^2(u_\beta^\top I_b(\theta_0;\theta_0)u_\beta)\quad(\text{under}\ H_{1,n}^{(2)}).
\end{align*}
This completes the proof of the Likelihood type.

$(\mathrm{ii})$ From $\mathbf{C1}$, we have
$$
I_{a,n}(\hat{\alpha}_n)\overset{P}{\to}I_a(\alpha_0;\alpha_0)\quad(\text{under}\ H_{1,n}^{(1)}).
$$
Therefore, 
it follows from Lemma \ref{lem:3} and the continuous mapping theorem that
\begin{align*}
W_n^{(1)}&=\left(\sqrt{n}(\hat{\alpha}_n-\alpha_0)\right)^\top I_{a,n}(\hat{\alpha}_n)\sqrt{n}(\hat{\alpha}_n-\alpha_0)\\
&\overset{d}{\to}{Y_1'}^\top I_a(\alpha_0;\alpha_0)Y_1'\\
&\sim\chi^2_{p_1}(u_\alpha^\top I_a(\alpha_0;\alpha_0)u_\alpha)\quad(\text{under}\ H_{1,n}^{(1)}).
\end{align*}
Furthermore, we obtain 
$$
I_{b,n}(\hat{\beta}_n|\hat{\alpha}_n)\overset{P}{\to}I_b(\theta_0;\theta_0)\quad(\text{under}\ H_{0,n}^{(2)}),
$$
and 
\begin{align*}
W_n^{(2)}&=\left(\sqrt{nh_n}(\hat{\beta}_n-\beta_0)\right)^\top I_{b,n}(\hat{\beta}_n|\hat{\alpha}_n)\sqrt{nh_n}(\hat{\beta}_n-\beta_0)\\
&\overset{d}{\to}{Y_2'}^\top I_b(\theta_0;\theta_0)Y_2'\\
&\sim\chi^2_{p_2}(u_\beta^\top I_b(\theta_0;\theta_0)u_\beta)\quad(\text{under}\ H_{1,n}^{(2)}).
\end{align*}
This completes the proof of the Wald type.

$(\mathrm{iii})$ From the proof of Lemma \ref{lem:3}, 
it is shown that
$$
\frac{1}{\sqrt{n}}\partial_{\alpha}U_n^{(1)}(\alpha_0)\overset{d}{\to}L_1=I_a(\alpha_0;\alpha_0)Y_1'\sim N(I_a(\alpha_0;\alpha_0)u_\alpha,I_a(\alpha_0;\alpha_0))\quad(\text{under}\ H_{1,n}^{(1)}).
$$
On the other hand, we can check
$$
\bar{I}_{a,n}(\hat{\alpha}_n)\overset{P}{\to}I^{-1}_a(\alpha_0;\alpha_0)\quad(\text{under}\ H_{1,n}^{(1)}).
$$
Therefore, it follows from the continuous mapping theorem that
\begin{align*}
R_n^{(1)}&=\left(\frac{1}{\sqrt{n}}\partial_{\alpha}U_n^{(1)}(\alpha_0)\right)^\top \bar{I}_{a,n}(\hat{\alpha}_n)\frac{1}{\sqrt{n}}\partial_{\alpha}U_n^{(1)}(\alpha_0)\\
&\overset{d}{\to}L_1^\top I^{-1}_{a}(\alpha_0;\alpha_0)L_1\\
&={Y_1'}^\top I_{a}(\alpha_0;\alpha_0)Y_1'\\
&\sim\chi_{p_1}^2(u_\alpha^\top I_a(\alpha_0;\alpha_0)u_\alpha)\quad(\text{under}\ H_{1,n}^{(1)}).
\end{align*}
Moreover, it follows from the proof of  Lemma \ref{lem:3} that
$$
\frac{1}{\sqrt{nh_n}}\partial_{\beta}U_n^{(2)}(\beta_0|\alpha_0)\overset{d}{\to}L_2=I_b(\theta_0;\theta_0)Y_2'\sim N(I_b(\theta_0;\theta_0)u_\beta,I_b(\theta_0;\theta_0))\quad(\text{under}\ H_{1,n}^{(2)}).
$$
By $\mathbf{C1}$ and Taylor's theorem with respect to $\alpha$, 
we have
\begin{align*}
\frac{1}{\sqrt{nh_n}}\partial_{\beta}U_n^{(2)}(\beta_0|\hat{\alpha}_n)&=\frac{1}{\sqrt{nh_n}}\partial_{\beta}U_n^{(2)}(\beta_0|\alpha_0)\\
&\quad+\int_{0}^{1}{\frac{1}{n\sqrt{h_n}}\partial_\alpha\partial_\beta U_n^{(2)}(\beta_0|\alpha_0+u(\hat{\alpha}_n-\alpha_0))}du\sqrt{n}(\hat{\alpha}_n-\alpha_0)\\
&=\frac{1}{\sqrt{nh_n}}\partial_{\beta}U_n^{(2)}(\beta_0|\alpha_0)+o_P(1)\quad(\text{under}\ H_{1,n}^{(2)}).
\end{align*}
Next, we can show that
$$
\bar{I}_{b,n}(\hat{\beta}_n|\hat{\alpha}_n)\overset{P}{\to}I^{-1}_b(\theta_0;\theta_0)\quad(\text{under}\ H_{1,n}^{(2)}).
$$
Therefore, by the continuous mapping theorem, it holds 
\begin{align*}
R_n^{(2)}&=\left(\frac{1}{\sqrt{nh_n}}\partial_{\beta}U_n^{(2)}(\beta_0|\hat{\alpha}_n)\right)^\top \bar{I}_{b,n}(\hat{\beta}_n|\hat{\alpha}_n)\frac{1}{\sqrt{nh_n}}\partial_{\beta}U_n^{(2)}(\beta_0|\hat{\alpha}_n)\\
&\overset{d}{\to}L_2^\top I^{-1}_{b}(\theta_0;\theta_0)L_2\\
&={Y_2'}^\top I_{b}(\theta_0;\theta_0)Y_2'\\
&\sim\chi_{p_2}^2(u_\beta^\top I_b(\theta_0;\theta_0)u_\beta)\quad(\text{under}\ H_{1,n}^{(2)}), 
\end{align*}
which completes the proof of the Rao type.
\end{proof}


\bibliographystyle{abbrv}
\nocite{*}
\bibliography{main.bbl}

\end{document}